\date{} 
\title{Dimension~of~the~$\SLE$~light~cone, the~$\SLE$~fan, and
$\SLE_\kappa(\rho)$~for~$\kappa \in (0,4)$~and~$\rho \in [\tfrac{\kappa}{2}-4,-2)$}
\author{Jason Miller}
\documentclass[12pt,naturalnames]{article}
\usepackage{amsmath}
\usepackage{amssymb}
\usepackage{amsthm}
\usepackage{amsfonts}
\usepackage{graphicx}
\usepackage{tabularx}
\usepackage{subfigure}
\usepackage{enumerate}
\usepackage{color}
\usepackage{xparse}
\usepackage{comment}
\usepackage[normalem]{ulem}
\usepackage{microtype}

\usepackage[margin=1.2in]{geometry}

\usepackage{booktabs}
\newcommand{\head}[1]{\textnormal{\textbf{#1}}}

\def\@rst #1 #2other{#1}
\newcommand\MR[1]{\relax\ifhmode\unskip\spacefactor3000 \space\fi
  \MRhref{\expandafter\@rst #1 other}{#1}}
\newcommand{\MRhref}[2]{\href{http://www.ams.org/mathscinet-getitem?mr=#1}{MR#2}}

\usepackage[usenames,dvipsnames]{xcolor}
\usepackage[pdftitle={Dimension of the SLE light cone, the SLE fan, and SLE$_\kappa(\rho)$ for $\kappa \in (0,4)$ and $\rho \in [\kappa/2-4,-2)$},
  pdfauthor={Jason Miller},
colorlinks=true,linkcolor=NavyBlue,urlcolor=RoyalBlue,citecolor=PineGreen,bookmarks=true,bookmarksopen=true,bookmarksopenlevel=2,unicode=true,linktocpage]{hyperref}

\allowdisplaybreaks

\newif\ifhyper\IfFileExists{hyperref.sty}{\hypertrue}{\hyperfalse}

\ifhyper\usepackage{hyperref}\fi

\newif\ifdraft
\drafttrue
\numberwithin{equation}{section}
\numberwithin{figure}{section}

\newtheorem{theorem}{Theorem}
\numberwithin{theorem}{section}

\newtheorem{lemma}[theorem]{Lemma}
\newtheorem{proposition}[theorem]{Proposition}

\theoremstyle{remark}
\theoremstyle{remark}\newtheorem{remark}[theorem]{Remark}

\newcommand{\R}{\mathbf{R}}
\renewcommand{\C}{\mathbf{C}}
\newcommand{\D}{\mathbf{D}}
\newcommand{\Z}{\mathbf{Z}}
\newcommand{\N}{\mathbf{N}}

\newcommand{\h}{\mathbf{H}}

\newcommand{\CC}{\mathcal {C}}
\newcommand{\CD}{\mathcal {D}}
\newcommand{\CE}{\mathcal {E}}
\newcommand{\CF}{\mathcal {F}}

\newcommand{\CP}{\mathcal {P}}

\newcommand{\CS}{\mathcal {S}}

\newcommand{\CU}{\mathcal {U}}
\newcommand{\CV}{\mathcal {V}}

\newcommand{\CG}{\mathcal {G}}
\newcommand{\CH}{\mathcal {H}}

\def\dist{\mathop{\mathrm{dist}}}

\newcommand{\giv}{\,|\,}

\newcommand{\SLE}{{\rm SLE}}

\newcommand{\ol}{\overline}
\newcommand{\ul}{\underline}
\newcommand{\wh}{\widehat}
\newcommand{\wt}{\widetilde}
\newcommand{\hstrip}{\CS}
\newcommand{\dimH}{\mathrm{dim}_{\mathcal H}}
\newcommand{\fan}{{\mathbf F}}
\newcommand{\one}{{\mathbf 1}}

\newcommand{\im}{{\rm Im}}

\definecolor{purple}{rgb}{0.7,0,0.7}
\definecolor{gray}{rgb}{0.6,0.6,0.6}
\definecolor{dgreen}{rgb}{0.0,0.4,0.0}
\definecolor{dblue}{rgb}{0.0,0.0,0.5}

\newcommand{\lightcone}{{\mathbf L}}

\NewDocumentCommand{\pocket}{m+g}{
\IfNoValueTF{#2}
 {P(#1)}
 {P_{#1}(#2)}
}

\NewDocumentCommand{\side}{m+g}{
\IfNoValueTF{#2}
 {S_{#1}}
 {S_{#1}(#2)}
}

\NewDocumentCommand{\sideflow}{m+g}{
\IfNoValueTF{#2}
 {\eta_{#1}}
 {\eta_{#1}(#2)}
}

\def\Ito/{It\^o}

\def \p {{\mathbf P}}
\def \E {{\bf E}}

\begin{document}
\maketitle

\begin{abstract}
Suppose that $h$ is a Gaussian free field (GFF) on a planar domain.  Fix $\kappa \in (0,4)$.  The $\SLE_\kappa$ light cone $\lightcone(\theta)$ of $h$ with opening angle $\theta \in [0,\pi]$ is the set of points reachable from a given boundary point by angle-varying flow lines of the (formal) vector field $e^{i h/\chi}$, $\chi = \tfrac{2}{\sqrt{\kappa}} - \tfrac{\sqrt{\kappa}}{2}$, with angles in $[-\tfrac{\theta}{2},\tfrac{\theta}{2}]$.  We derive the Hausdorff dimension of $\lightcone(\theta)$.

If $\theta =0$ then $\lightcone(\theta)$ is an ordinary $\SLE_{\kappa}$ curve (with $\kappa < 4$); if $\theta = \pi$ then $\lightcone(\theta)$ is the range of an $\SLE_{\kappa'}$ curve ($\kappa' = 16/\kappa > 4$). In these extremes, this leads to a new proof of the Hausdorff dimension formula for~$\SLE$.

We also consider $\SLE_\kappa(\rho)$ processes, which were originally only defined for $\rho > -2$, but which can also be defined for $\rho \leq -2$ using L\'evy compensation. The range of an $\SLE_\kappa(\rho)$ is qualitatively different when $\rho \leq -2$.  In particular, these curves are self-intersecting for $\kappa < 4$ and double points are dense, while ordinary $\SLE_\kappa$ is simple.  It was previously shown (Miller-Sheffield, 2016) that certain $\SLE_\kappa(\rho)$ curves agree in law with certain light cones. Combining this with other known results, we obtain a general formula for the Hausdorff dimension of $\SLE_\kappa(\rho)$ for all values of $\rho$.

Finally, we show that the Hausdorff dimension of the so-called $\SLE_\kappa$ {\em fan} is the same as that of ordinary $\SLE_\kappa$.
\end{abstract}

\newpage
\tableofcontents
\newpage

\medbreak {\noindent\bf Acknowledgements.}  This research was partially supported by NSF grant DMS-1204894.  We thank Scott Sheffield and Wendelin Werner for helpful discussions.  We also thank two anonymous referees for helpful suggestions which led to many improvements to this article.

\section{Introduction}

Suppose that~$h$ is an instance of the Gaussian free field (GFF) on a planar domain~$D$.  Although~$h$ is not a function and does not take values at points, one can still make sense of the flow lines of the (formal) vector field~$e^{i h / \chi}$ where~$\chi > 0$, i.e., the (formal) solutions to the equation $\eta'(t) = e^{i h(\eta(t))/\chi}$ \cite{SHE_WELD,DUB_PART,MS_IMAG,MS_IMAG4}.  These paths turn out to be forms of the Schramm-Loewner evolution ($\SLE$) \cite{S0}.

The purpose of this work is to compute the almost sure Hausdorff dimension of certain sets which naturally fit into the~imaginary geometry (i.e., $\SLE$/GFF coupling) framework. These sets can either be described as {\em light cones} associated to an imaginary geometry or as {\em ranges} of $\SLE_{\kappa}(\rho)$ processes with $\rho < -2$.

Specifically, suppose that~$h$ is a GFF on the upper half plane~$\h$ with piecewise constant boundary conditions which change values at most a finite number of times.  Fix $\kappa \in (0,4)$ and $\theta \in [0,\pi]$.  The~$\SLE_\kappa$ {\bf light cone}~$\lightcone(\theta)$ of~$h$ starting from~$0$ is the closure of the set of points accessible by traveling along angle-varying flow lines of the (formal) vector field~$e^{ih/\chi}$, $\chi = \tfrac{2}{\sqrt{\kappa}}-\tfrac{\sqrt{\kappa}}{2}$, starting from~$0$ with angles contained in $[-\tfrac{\theta}{2},\tfrac{\theta}{2}]$.  When $\theta = 0$, the light cone is equal to the range of an~$\SLE_\kappa$ process.  It is shown in \cite{MS_IMAG} that when~$\theta=\pi$, the light cone is equal to the range of an~$\SLE_{16/\kappa}$ process.  By varying $\theta \in (0,\pi)$, the sets $\lightcone(\theta)$ continuously interpolate between the range of an~$\SLE_\kappa$ process ($\theta = 0$) and the range of an~$\SLE_{16/\kappa}$ process ($\theta = \pi$) \cite{MS_LIGHTCONE}.  See \cite[Section~1]{MS_IMAG} for simulations of the light cone.

Let~$\dimH(A)$ denote the Hausdorff dimension of a set~$A$.  The purpose of this work is to compute the almost sure value of~$\dimH(\lightcone(\theta))$.  Throughout, we write
\begin{equation}
\label{eqn::lightcone_dimension_formula}
 d(\kappa,\theta) = \frac{(\kappa(1-\ol{\theta}) + 4\ol{\theta})(\kappa+8+(\kappa-4)\ol{\theta})}{8\kappa} \quad\text{where}\quad \ol{\theta} = \frac{\theta}{\pi}.
\end{equation}

The main result of this article is the following:

\begin{theorem}
\label{thm::lightcone_dimension}
Suppose that~$h$ is a GFF on~$\h$ with piecewise constant boundary data which changes values at most a finite number of times.  Let~$\lightcone(\theta)$ be the~$\SLE_\kappa$ light cone ($\kappa \in (0,4)$) of~$h$ starting from~$0$ with opening angle $\theta \in [0,\pi]$ and assume that the boundary data of $h$ is such that~$\p[\lightcone(\theta) \neq \emptyset]=1$.  Almost surely,
\begin{equation}
\label{eqn::lightcone_dimension}
 \dimH(\lightcone(\theta)) = d(\kappa,\theta) \wedge 2.
\end{equation}
\end{theorem}
The angle $\theta_c = \theta_c(\kappa)$ which solves $d(\kappa,\theta) = 2$ is given by
\begin{equation}
\label{eqn::space_filling_angle}
 \theta_c = \frac{\pi \kappa}{4-\kappa}.
\end{equation}
We note that~$\theta_c$ is equal to the so-called {\bf critical angle} introduced in \cite{MS_IMAG,MS_IMAG4}. Two GFF flow lines --- with a common starting point and a given angle difference --- intersect each other away from the starting point if and only if the angle difference is less than or equal to this critical angle (see \cite[Theorem~1.5]{MS_IMAG}).  Note that $\theta_c \in (0,\pi)$ for $\kappa \in (0,2)$, $\theta_c = \pi$ for $\kappa = 2$, and $\theta_c > \pi$ for $\kappa \in (2,4)$. Since we only define light cones $\lightcone(\theta)$ for $\theta \in [0,\pi]$, this implies that~$\SLE_\kappa$ light cones can be space-filling if and only if $\kappa \in (0,2]$.  This corresponds to the fact that~$\SLE_{16/\kappa}$ is space-filling if and only if $\kappa \in (0,2]$ \cite{RS05}.  We will provide additional explanation in Remark~\ref{rem::critical_angle} for why~$\theta_c$ naturally appears in Theorem~\ref{thm::lightcone_dimension}.

The $\SLE_\kappa(\rho)$ processes, first introduced in \cite[Section~8.3]{LSW_RESTRICTION}, are an important variant of~$\SLE_\kappa$ in which one keeps track of an extra marked point $V$ called a {\bf force point} in addition to the Loewner driving function $W$.  (See Section~\ref{subsec::sle}.) The force point can be located either in the interior of the domain or on its boundary.  Throughout this article, we will primarily restrict ourselves to the case in which the force point is on the boundary of the upper half plane, so that $V$ is a real-valued process like $W$.

 The parameter~$\rho$ determines the strength of the ``interaction'' between $W$ and $V$.  When~$\rho=0$, $\SLE_\kappa(\rho)$ is the same as ordinary~$\SLE_\kappa$.  When $\rho > 0$ (resp.\ $\rho < 0$), $W$ is pushed away from (resp.\ pulled towards) $V$.  Like ordinary~$\SLE_\kappa$, the~$\SLE_\kappa(\rho)$ processes are described in terms of the Loewner evolution driven by $W$.  However, the law of $W$ is different, and is determined by the fact that $V - W$ is a positive multiple of a Bessel process whose dimension depends on both~$\kappa$ and~$\rho$ and is explicitly given by $\delta(\kappa,\rho) = 1+\tfrac{2(\rho+2)}{\kappa}$,  see Section~\ref{subsec::sle}.

\begin{remark} There are variants of $\SLE_{\kappa}(\rho)$ in which the {\em sign} of each excursion of $V - W$ away from zero is chosen independently at random with a fixed biased coin; but throughout this paper we will always assume that the sign of $V-W$ is the same for all excursions---in other words, in this paper we consider only {\em one-sided} $\SLE_\kappa(\rho)$ and not {\em side-swapping} $\SLE_\kappa(\rho)$.
\end{remark}

In order to define the process for all time (as opposed to having it stop when $V$ and $W$ first collide) most treatments of $\SLE_\kappa(\rho)$ (including \cite{LSW_RESTRICTION}) require that $\rho > -2$, so that $\delta = \delta(\kappa, \rho) > 1$. But the processes with $\rho \leq -2$ can also be defined (using an appropriate L\'evy compensation) and are also important. As explained in \cite{cle_percolations}, when $\kappa \in (2,4)$ there are certain ranges of $\rho$ values for which $\SLE_{\kappa}(\rho)$ can be described as the concatenation of a countable collection of $\SLE_{\kappa}$ loops, all attached to an $\SLE_{\kappa'}$ ``trunk'' ($\kappa'=16/\kappa > 4$) and in these cases the dimension of the whole range of the path is the dimension of the trunk, namely $(1+\kappa'/8) \wedge 2$.  As explained in \cite{MS_LIGHTCONE}, there are other values of $\rho$ such that the range of an $\SLE_\kappa(\rho)$ process agrees in law with a light cone $\lightcone(\theta)$ (defined from a GFF with particular boundary values) where the relationship between~$\rho$ and~$\theta$ is given by the formula
\begin{equation} 
\label{eqn::lightcone_angle}
 \theta = \theta_\rho = \pi\left(\frac{\rho+2}{\tfrac{\kappa}{2}-2} \right).
\end{equation}
Table~\ref{tab::rho_values} presents a phase diagram for the~$\rho$ values and the corresponding Bessel process dimensions~$\delta(\kappa,\rho)$ for~$\SLE_\kappa(\rho)$ processes with $\kappa \in (0,4)$, see also Figure~\ref{fig::kapprhophasechart}.  The dimensions in the table are obtained by combining Theorem~\ref{thm::lightcone_dimension} with the main result of \cite{MS_LIGHTCONE}. Let us state this as a theorem:

\begin{table}
\begin{center}
{\footnotesize
\begin{tabular}{llllc}
\toprule
\head{$\rho$} & \head{$\delta(\kappa,\rho)$}  & \head{$\dimH(\text{Range})$} & \head{Process type} & \head{Simple}\\
\toprule

$(-\infty,-2-\tfrac{\kappa}{2}]$ & $(-\infty,0]$ & --- & Not defined & ---\\
$(-2-\tfrac{\kappa}{2},\tfrac{\kappa}{2}-4]$ & $(0,2-\tfrac{4}{\kappa}] $ & $1+\tfrac{2}{\kappa} = 1+\tfrac{\kappa'}{8}$ & Trunk plus loops & $\text{\sffamily X}$\\ \vspace{.02in}
$(\tfrac{\kappa}{2}-4,-2)$ & $(2-\tfrac{4}{\kappa},1)$ & $\tfrac{(\kappa-2(2+\rho))(\kappa+2(6+\rho))}{8\kappa}$ & Light cone & $\text{\sffamily X}$ \\
$-2$ & $1$ & $1$ & Boundary tracing & $\checkmark$\\
$(-2,\tfrac{\kappa}{2}-2)$ & $(1,2)$ & $1+\tfrac{\kappa}{8}$ & Boundary hitting & $\checkmark$\\
$ [\tfrac{\kappa}{2}-2,\infty)$ & $[2,\infty)$ & $1+\tfrac{\kappa}{8}$ & Boundary avoiding & $\checkmark$\\
\bottomrule
\end{tabular}}
\end{center}
\medskip
\caption{\label{tab::rho_values} Phases of~$\rho$ values and corresponding $\delta(\kappa,\rho)$ (driving Bessel process dimension) values for $\SLE_\kappa(\rho)$ processes with a single boundary force point of weight~$\rho$, assuming $\kappa \in (2,4)$. When $\kappa \in (0, 2]$, the phase diagram is the same except that the second two phases are replaced by a single ``light cone'' phase with $\rho \in (-2-\kappa/2,-2)$ and $\delta \in (0,1)$. When the dimension value shown in the table is greater than or equal to $2$, the curve is space-filling, so the actual dimension is $2$. See Figure~\ref{fig::kapprhophasechart}.
}
\end{table}

\begin{figure}[ht!]
\begin{center}
\includegraphics[width=0.4\textwidth]{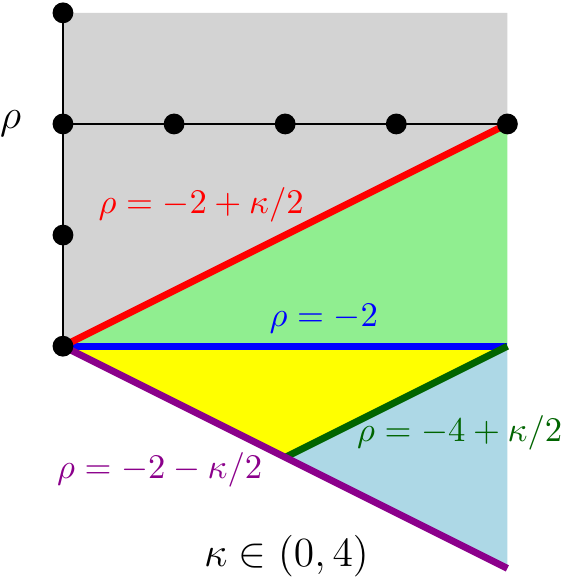}
\end{center}
\caption{\label{fig::kapprhophasechart}
$\SLE_{\kappa}(\rho)$ phases from Table~\ref{tab::rho_values}. This paper computes dimensions within the yellow ``light cone'' region described in more detail in \cite{MS_LIGHTCONE}. As one approaches the $\rho = -2$ {\em upper boundary} from below, the curve converges in law to a boundary-tracing curve, but the dimension converges to $1 + \kappa/8$ (which is also the dimension above the $\rho = -2$ line). As one approaches the {\em lower right} edge of the yellow triangle from above, $\theta$ tends to $\pi$, the {\em range} of the path (though not the path itself) converges in law to that of an $\SLE_{\kappa'}$ type curve, and the dimension converges to $1+\kappa'/8$, which is also the dimension throughout the blue triangle. As one approaches the {\em lower left} edge of the yellow triangle from above, $\theta$ tends to $\theta_c$, the corresponding $\SLE_{\kappa}(\rho)$ curves becomes space-filling, and the dimension tends to $2$.  In short, as one approaches the yellow three triangle edges (clockwise from the top, respectively) the dimensions converge to $1+\kappa/8$ and $1+\kappa'/8$ and $2$.}
\end{figure}

\begin{theorem}
\label{thm::sle_kappa_rho_dimension}
Suppose $\kappa \in (0,4)$ and that $\eta$ is an $\SLE_\kappa(\rho)$ process with $\rho \in ( (\tfrac{\kappa}{2}-4) \vee (-2 - \frac{\kappa}{2}), -2)$. (These $\rho$ values correspond to the light cone phase described in Table~\ref{tab::rho_values} and its caption.) Then almost surely,
\begin{equation}
\label{eqn::sle_kappa_rho_dimension}
 \dimH(\eta) = \frac{(\kappa-2(2+\rho))(\kappa+2(6+\rho))}{8\kappa}.
\end{equation}
\end{theorem}

The almost sure Hausdorff dimension for ordinary $\SLE_\kappa$ is given by $1+\tfrac{\kappa}{8}$ for $\kappa \in (0,8)$ and by $2$ for $\kappa \geq 8$.  The upper bound for this result was first obtained by Rohde and Schramm \cite{RS05} and the lower bound was established by Beffara \cite{BEF_DIM}. Now suppose that $\eta$ is the trace of an $\SLE_\kappa(\rho)$ process $\eta$ with driving function $W$ and force point process $V$.  By the Girsanov theorem \cite{RY04}, the evolution of an $\SLE_\kappa(\rho)$ process $\eta$ --- started at a time when $V_t \not = W_t$ and stopped at a time before $V$ and $W$ collide --- has a law that is absolutely continuous with respect to that of an ordinary~$\SLE_\kappa$ process restricted to the same time interval. From this it is easy to show that the dimension of $\eta(\{t : V_t \not = W_t \})$ is a.s.\ the same as the dimension of ordinary $\SLE_\kappa$

But what about the set  $\eta(\{t : V_t = W_t \})$? When $\rho > -2$, the times when $V_t = W_t$ correspond to times when $\eta$ is hitting the boundary, so this set is a subset of $\R$. Consequently, $\dimH(\eta)$ in this case is the same as the dimension of ordinary $\SLE_\kappa$ because we trivially have that $\dimH(\eta \cap \R) \leq 1 \leq 1 + \tfrac{\kappa}{8}$.  (The almost sure value of $\dimH(\eta \cap \R)$ as a function of $\rho$ and $\kappa$ is given in \cite[Theorem~1.6]{MW_INTERSECTIONS}.)  For $\rho \in [\tfrac{\kappa}{2}-4,-2)$, the problem is more interesting because the set $\eta (\{ t: V_t = W_t\})$ includes points in the interior of the domain.  In fact, Theorem~\ref{thm::sle_kappa_rho_dimension} implies that the dimension of this set is strictly larger than the dimension of $\eta (\{ t: V_t \not = W_t\})$.

In \cite{cle_percolations}, it is shown that the same is true for $\rho \in (-2-\tfrac{\kappa}{2},\tfrac{\kappa}{2}-4]$.  In this case, the dimension of the range turns out to be $1+\tfrac{2}{\kappa}$, the same as the dimension of an $\SLE_{\kappa'}$ process, for all $\rho \in (-2-\tfrac{\kappa}{2},\tfrac{\kappa}{2}-4]$.  Together with this work, this covers the entire range of possible $\rho$ values.

Theorem~\ref{thm::sle_kappa_rho_dimension} also implies that the dimension of an $\SLE_\kappa(\rho)$ process, $\kappa \in (0,4)$, continuously interpolates between that of ordinary $\SLE_\kappa$ and that of $\SLE_{\kappa'}$.

The method we use to derive the so-called one point estimate (the exponent for the probability that the path gets within distance $\epsilon > 0$ of a given point as $\epsilon \to 0$) which, in turn, leads to the upper bounds in Theorem~\ref{thm::lightcone_dimension} and Theorem~\ref{thm::sle_kappa_rho_dimension}, is rather different in spirit from the method used by Rohde and Schramm \cite{RS05} to derive the corresponding one point estimate for ordinary $\SLE_\kappa$.  The strategy employed in \cite{RS05} is to try to find a martingale which becomes large on the event that an $\SLE_\kappa$ process gets close to a given point.  This leads one to derive and solve a certain PDE.  In the setting of an $\SLE_\kappa(\rho)$ process with $\rho < -2$, extending this method seems to be technically challenging because the presence of the force point introduces a second spatial variable into the corresponding PDE and, as we remarked earlier, one cannot use absolute continuity to compare to ordinary $\SLE_\kappa$.  To circumvent this difficulty, in the present article we will relate the event that $\lightcone(\theta)$ (or an $\SLE_\kappa(\rho)$ process with $\rho \in [\tfrac{\kappa}{2}-4,-2)$) gets close to a given point to the local structure of the flow lines of the GFF starting from that point.  One of the highlights of this approach is that it is conceptual in nature rather than computational.  The basic idea is illustrated in more detail in Figures~\ref{fig::flowline_hit_point}--\ref{fig::counterflowline_hit_conformal_maps}.  We will then use the martingales from \cite{SCHRAMM_WILSON} to estimate the probability that the local structure of the flow lines at a given point exhibits the necessary behavior for $\lightcone(\theta)$ to hit.

\begin{figure}[ht!]
\begin{center}
\includegraphics[width=0.8\textwidth,clip=true, trim = 1mm 1mm 1mm 1mm]{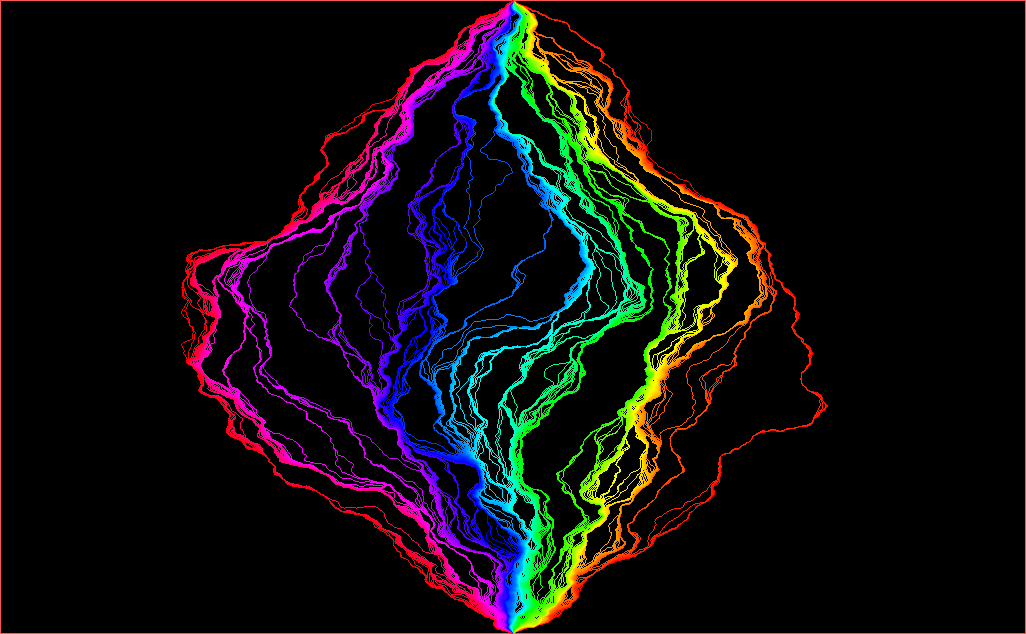}
\end{center}
\caption{\label{fig::fan}
Numerically generated flow lines, started at $-i$ of $e^{i(h/\chi+\theta)}$ where $h$ is the projection of a GFF on $[-1,1]^2$ onto the space of functions piecewise linear on the triangles of a $300 \times 300$ grid; $\kappa=1/4$.  Different colors indicate different values of $\theta \in [-\tfrac{\pi}{2},\tfrac{\pi}{2}]$.  The boundary data for $h$ is chosen so that the central (``north-going'') curve shown should approximate an $\SLE_{1/4}$ process.  The other paths should approximate $\SLE_{1/4}(\rho_1;\rho_2)$ processes where the values of $\rho_1,\rho_2 > -2$ are a function of $\theta$.  We prove in Theorem~\ref{thm::fan_dimension} that the almost sure Hausdorff dimension of the entire set shown is equal to the dimension of each of the individual paths.
}
\end{figure}

The lower bound is proved by relating the correlation structure of the points in $\lightcone(\theta)$ to the correlation structure of the values of $h$.  Roughly speaking, the approximate ``tree structure'' used in the lower bound arises because the collection of flow lines of $h$ with a common angle themselves form a tree (see \cite[Theorem~1.5 and Figure~1.7]{MS_IMAG} as well as \cite[Figures~1.4--1.6]{MS_IMAG4}).  Since $\lightcone(0)$ is equal to the range of an $\SLE_\kappa$ process for $\kappa \in (0,4)$ and $\lightcone(\pi)$ is equal to the range of an $\SLE_{\kappa'}$ process, we obtain as a special case of Theorem~\ref{thm::lightcone_dimension} the almost sure dimension of ordinary $\SLE$.  We remark that this is not the first article in which the imaginary geometry framework is used to compute dimensions related to $\SLE$:  it is used in \cite{MW_INTERSECTIONS} to derive the cut point, double point, and other dimensions associated with intersection sets of $\SLE$ paths; in \cite{multifractal_spectrum} to derive the almost sure multifractal spectrum of $\SLE$; and in \cite{brownian_motion_kpz,dimension_transformation} to derive certain KPZ-type formulas (using also the tools of \cite{dms2014mating}).

Fix $\theta \in [0,\pi]$.  The $\SLE_\kappa$ {\bf fan} $\fan(\theta)$ is the set of points accessible by flow lines of $h$ starting from $0$ with \emph{fixed} angles in $[-\tfrac{\theta}{2},\tfrac{\theta}{2}]$.  (This is in contrast to the paths which generate $\lightcone(\theta)$, since they are allowed to change angles.)  See Figure~\ref{fig::fan} for a numerical simulation of $\fan(\pi)$ for $\kappa=\tfrac{1}{4}$.  Obviously $\dimH(\fan(\theta)) \geq 1+\tfrac{\kappa}{8}$ because $\fan(\theta)$ contains the range of the $0$ angle flow line of $h$ which is itself an $\SLE_\kappa$ process (or possibly an $\SLE_\kappa(\ul{\rho})$ process depending on the boundary data of $h$).  It was shown in \cite{MS_IMAG} that the Lebesgue measure of $\fan(\theta)$ is almost surely zero.  Our final result gives that $\dimH(\fan(\theta)) = 1+\tfrac{\kappa}{8}$:

\begin{theorem}
\label{thm::fan_dimension}
For each $\kappa \in (0,4)$ and $\theta \in [0,\pi]$, the almost sure Hausdorff dimension of the $\SLE_\kappa$ fan $\fan(\theta)$ is $1+\tfrac{\kappa}{8}$ (assuming that the boundary data of $h$ is such that $\p[ \fan(\theta) \neq \emptyset] = 1$), the same as that of ordinary $\SLE_\kappa$.
\end{theorem}

The reader might find Theorem~\ref{thm::fan_dimension} surprising because $\fan(\theta)$ consists of \emph{many} $\SLE_\kappa$ paths and one might suspect that the limit points of these paths would make the dimension strictly larger than that of a single $\SLE_\kappa$ path.  Theorem~\ref{thm::fan_dimension}, however, implies that this is not the case.  Another reason that the reader may find the result to be surprising is that the numerical simulations \cite[Figures~1.3---1.5]{MS_IMAG} suggest that as $\kappa \downarrow 0$ (but for a fixed value of $\theta \in (0,\pi)$), $\fan(\theta)$ converges almost surely in the Hausdorff topology to a \emph{two-dimensional} set.  However, Theorem~\ref{thm::fan_dimension} implies that $\dimH(\fan(\theta))$ almost surely converges to $1$ as $\kappa \downarrow 0$.

In the proofs of this work, we will assume that the reader has some familiarity with imaginary geometry as presented in \cite{MS_IMAG,MS_IMAG2,MS_IMAG3,MS_IMAG4} (though we will provide a reminder of the basic facts in Section~\ref{subsec:ig_review}).  We will in particular make use of the notation introduced in \cite[Figure~1.10]{MS_IMAG}.  Throughout, we assume that $\kappa \in (0,4)$, $\kappa'=16/\kappa \in (4,\infty)$, and let
\begin{equation}
\label{eqn:ig_constants}	
\chi = \frac{2}{\sqrt{\kappa}} - \frac{\sqrt{\kappa}}{2},\quad \lambda = \frac{\pi}{\sqrt{\kappa}},\quad\text{and}\quad \lambda' = \frac{\pi}{\sqrt{\kappa'}} = \lambda - \frac{\pi}{2} \chi.
\end{equation}
We will also use $\eta$ to refer to an $\SLE_\kappa$ process and $\eta'$ to refer to an $\SLE_{\kappa'}$ process.

\subsection*{Outline}

The remainder of this article is structured as follows.  In Section~\ref{sec::preliminaries}, we will collect several estimates which are used throughout this article as well as give a brief review of the results from \cite{MS_IMAG,MS_IMAG2,MS_IMAG3,MS_IMAG4} which will be used in this article.  Next, in Section~\ref{sec::upper_bound} we will prove the upper bound for Theorem~\ref{thm::lightcone_dimension}, hence also Theorem~\ref{thm::sle_kappa_rho_dimension}, and complete the proof of Theorem~\ref{thm::fan_dimension}.  Finally, in Section~\ref{sec::lowerbound} we will complete the proof of the lower bound for Theorem~\ref{thm::lightcone_dimension} hence also Theorem~\ref{thm::sle_kappa_rho_dimension}.

\section{Preliminaries}
\label{sec::preliminaries}

\subsection{$\SLE_\kappa(\rho)$ processes}
\label{subsec::sle}

We will now give a very brief introduction to~$\SLE$.  More detailed introductions can be found in many excellent surveys of the subject, e.g., \cite{W03, LAW05}.  Chordal~$\SLE_\kappa$ in~$\h$ from~$0$ to~$\infty$ is defined by the random family of conformal maps $(g_t)$ obtained by solving the Loewner ODE
\begin{equation}
\label{eqn::loewner_ode}
\partial_t g_t(z) = \frac{2}{g_t(z) - W_t},\quad\quad g_0(z) = z
\end{equation}
with $W = \sqrt{\kappa} B$ and~$B$ a standard Brownian motion.  Write $K_t := \{z \in \h: \tau(z) \leq t \}$ where $\tau(z)$ is the swallowing time of~$z$ defined by $\sup\{t\ge 0: \inf_{s\in[0,t]}|g_s(z)-W_s|>0\}$.  Then~$g_t$ is the unique conformal map from $\h_t := \h \setminus K_t$ to~$\h$ satisfying $\lim_{|z| \to \infty} |g_t(z) - z| = 0$.

Rohde and Schramm \cite{RS05} showed that there almost surely exists a curve~$\eta$ (the so-called~$\SLE$ {\bf trace}) such that for each $t \geq 0$ the domain~$\h_t$ of~$g_t$ is the unbounded connected component of $\h \setminus \eta([0,t])$, in which case the (necessarily simply connected and closed) set~$K_t$ is called the ``filling'' of $\eta([0,t])$ \cite{RS05}.  An~$\SLE_\kappa$ connecting boundary points~$x$ and~$y$ of an arbitrary simply connected Jordan domain can be constructed as the image of an~$\SLE_\kappa$ on~$\h$ under a conformal transformation $\varphi \colon \h \to D$ sending~$0$ to~$x$ and~$\infty$ to~$y$.  (The choice of~$\varphi$ does not affect the law of this image path, since the law of~$\SLE_\kappa$ on~$\h$ is scale invariant.)  For $\kappa \in [0,4]$, $\SLE_\kappa$ is simple and, for $\kappa > 4$, $\SLE_\kappa$ is self-intersecting \cite{RS05}.  The dimension of the path is $1+\tfrac{\kappa}{8}$ for $\kappa \in [0,8]$ and~$2$ for $\kappa > 8$ \cite{BEF_DIM}.

An $\SLE_{\kappa}(\underline{\rho}_L;\underline{\rho}_R)$ process is a generalization of
$\SLE_{\kappa}$ in which one keeps track of additional marked points which are called {\bf force points}.  These processes were first introduced in \cite[Section~8.3]{LSW_RESTRICTION}.  Fix $\underline{x}_L=(x_{\ell,L}<\cdots<x_{1,L}\le 0)$ and $\underline{x}_R=(0\le x_{1,R}<\cdots<x_{r,R})$.  We associate with each~$x_{i,q}$ for $q \in \{L,R\}$ a weight $\rho_{i,q} \in \R$.  An $\SLE_{\kappa}(\underline{\rho}_L; \underline{\rho}_R)$ process with force points $(\underline{x}_L; \underline{x}_R)$ is the measure on continuously growing compact hulls~$K_t$ generated by the Loewner chain with~$W_t$ replaced by the solution to the system of SDEs:
\begin{equation}
\label{eqn::slesde}
\begin{split}
d W_t &= \sum_{i=1}^{\ell} \frac{\rho_{i,L}}{W_t-V_t^{i,L}} dt + \sum_{i=1}^{r} \frac{\rho_{i,R} }{W_t-V_t^{i,R}} dt + \sqrt{\kappa}d B_t,\\
d V_t^{i,q}&= \frac{2}{V_t^{i,q}- W_t} dt,\quad V_0^{i,q} = x_{i,q},\quad i\in\N,\quad q\in\{L, R\}.
\end{split}
\end{equation}
It is explained in \cite[Section~2]{MS_IMAG} that for all~$\kappa>0$, there is a unique solution to~\eqref{eqn::slesde} up until the {\bf continuation threshold} is hit --- the first time~$t$ for which either
\begin{equation*}
\sum_{i:V^{i,L}_t=W_t} \rho_{i,L}\le -2 \quad\mbox{ or }\quad \sum_{i:V^{i,R}_t=W_t} \rho_{i,R}\le -2.
\end{equation*}
In the case of a single boundary force point, the existence of a unique solution to~\eqref{eqn::slesde} can be derived by relating $W_t-V_t$ to a Bessel process; see \cite{SHE_CLE}.  The almost sure continuity of the $\SLE_\kappa(\ul{\rho})$ processes up until the continuation threshold is reached is proved in \cite[Theorem~1.3]{MS_IMAG}.  It is possible to make sense of the solution to~\eqref{eqn::slesde} even after the continuation threshold is reached.  These processes are analyzed and shown to be continuous in \cite{MS_LIGHTCONE,cle_percolations} (see also \cite{SHE_CLE}).

\subsection{Radon-Nikodym derivatives}
\label{subsec::rn}

Let $c = (D,z_0,\ul{x}_L,\ul{x}_R,z_\infty)$ be a configuration consisting of a Jordan domain~$D$ in~$\C$ with~$\ell + r + 2$ marked points on $\partial D$.  An $\SLE_\kappa(\ul{\rho}_L;\ul{\rho}_R)$ process~$\eta$ with configuration $c$ is given by the image of an $\SLE_\kappa(\ul{\rho}_L;\ul{\rho}_R)$ process $\wt{\eta}$ in $\h$ which takes the force points of $\wt{\eta}$ to those of $\eta$.  Suppose that $c = (D,z_0,\ul{x}_L,\ul{x}_R,z_\infty)$ and $\wt{c} = (\wt{D},z_0,\wt{\ul{x}}_L,\wt{\ul{x}}_R,\wt{z}_\infty)$ are two configurations such that $\wt{D}$ agrees with $D$ in a neighborhood $U$ of $z_0$.  Let $\mu_c^U$ denote the law of an $\SLE_\kappa(\ul{\rho}_L;\ul{\rho}_R)$ process in $c$ stopped at the first time $\tau$ that it exits $U$ and define $\mu_{\wt{c}}^U$ analogously.  The following estimate is a restatement of \cite[Lemma~2.8]{MW_INTERSECTIONS} which, in turn, is based on extending \cite[Lemma~13]{DUB_DUAL} to the setting of boundary-intersecting $\SLE_\kappa(\ul{\rho})$ processes using the $\SLE$/GFF coupling.

\begin{lemma}
\label{lem::rn}
Assume that we have the setup described just above where $D = \h$, $\wt{D} \subseteq \h$, $U \subseteq \h$ is bounded, and $z_0 = 0$.  Fix $\zeta > 0$ and suppose that the distance between $U$ and $\h \setminus \wt{D}$ is at least $\zeta$, the force points of $c$, $\wt{c}$ in $\ol{U}$ are identical, the corresponding weights are also equal, and the force points which are outside of $U$ are at distance at least $\zeta$ from $U$.  There exists a constant $C \geq 1$ depending on $U$, $\zeta$, $\kappa$, and the weights of the force points such that
\[ \frac{1}{C} \leq \frac{d\mu_{\wt{c}}^U}{d\mu_c^U} \leq C.\]
\end{lemma}
\begin{proof}
See \cite[Lemma~2.7 and Lemma~2.8]{MW_INTERSECTIONS} as well as \cite[Lemma~13]{DUB_DUAL}.
\end{proof}

\subsection{Estimates for conformal maps}
\label{subsec::conformal}

Throughout, we will make frequent use of the following three estimates for conformal maps.  The first is \cite[Corollary~3.18]{LAW05}:

\begin{lemma}
\label{lem::distortion}
Suppose that $f \colon D \to D'$ is a conformal transformation with $f(z) = z'$.  Then
\[ \frac{d'}{4d} \leq |f'(z)| \leq \frac{4d'}{d}\]
where $d = \dist(z,\partial D)$ and $d' = \dist(z',\partial D')$.
\end{lemma}

\noindent The second is \cite[Corollary~3.23]{LAW05}:

\begin{lemma}
\label{lem::ball_size}
Suppose that $f \colon D \to D'$ is a conformal transformation with $f(z) = z'$.  For all $r \in (0,1)$ and all $|w-z| \leq r \times \dist(z,\partial D)$, we have that
\[ |f(w) - z'| \leq \frac{4|w-z|}{1-r^2} \times \frac{\dist(z',\partial D')}{\dist(z,\partial D)}.\]
\end{lemma}

\noindent Finally, we state the Beurling estimate \cite[Theorem~3.76]{LAW05} which we will frequently use in conjunction with the conformal invariance of Brownian motion.

\begin{theorem}[Beurling Estimate]
\label{thm::beurling}
Suppose that $B$ is a Brownian motion in $\C$ and $\tau_\D =\inf\{t \geq 0: B(t) \in \partial \D\}$.  There exists a constant $c < \infty$ such that if $\gamma \colon [0,1] \to \C$ is a curve with $\gamma(0) = 0$ and $|\gamma(1)| = 1$, $z \in \D$, and $\p^z$ is the law of $B$ when started at $z$, then
\[ \p^z[ B([0,\tau_\D]) \cap \gamma([0,1]) = \emptyset] \leq c |z|^{1/2}.\]
\end{theorem}

\subsection{Imaginary geometry review}
\label{subsec:ig_review}

Throughout this work, we assume that the reader is familiar with the GFF as well as with imaginary geometry.  We refer the reader to \cite{SHE06} for a more in-depth introduction to the former and to \cite{MS_IMAG} for more on the latter.  For the convenience of the reader, we will review some of the results from \cite{MS_IMAG} which will be used repeatedly throughout the present work.

We begin by describing the coupling of the $\SLE_\kappa(\ul{\rho})$ processes as flow lines of the GFF.  Fix $\kappa \in (0,4)$ and let $\kappa'=\tfrac{16}{\kappa} \in (4,\infty)$.  Recall the constants $\lambda=\tfrac{\pi}{\sqrt{\kappa}}$, $\chi = \tfrac{2}{\sqrt{\kappa}}-\tfrac{\sqrt{\kappa}}{2}$, and $\lambda' = \tfrac{\pi}{\sqrt{\kappa'}} = \lambda - \tfrac{\pi}{2} \chi$ as defined in~\eqref{eqn:ig_constants}.

Suppose that $\eta$ is an $\SLE_\kappa(\ul{\rho})$ process in $\h$ from $0$ to $\infty$.  Then $\eta$ can be coupled with a GFF $h$ on $\h$ so that it may be interpreted as a flow line of the (formal) vector field $e^{ih/\chi}$.  The boundary data for $h$ is given by
\begin{align*}
   -\lambda \left(1+\sum_{i=0}^j  \rho_{i,L}\right) \quad\text{on}\quad (x_{j+1,L},x_{j,L}] \quad\text{and}\quad \lambda\left(1+ \sum_{i=0}^j  \rho_{i,R}\right) \quad\text{on}\quad (x_{j,R},x_{j+1,R}].
\end{align*}
Here, we have taken $\rho_{0,L} = \rho_{0,R} = 0$, $x_{\ell+1,L}=-\infty$, and $x_{r+1,R} = +\infty$.  On the left (resp.\ right) hand side, $j$ varies between $0$ and $\ell$ (resp.\ $r$).  If $(g_t)$ denotes the Loewner evolution associated with $\eta$ and $V_t^{i,q}$ denotes the evolution of the force points, then we have for each stopping time $\tau$ for $\eta$ which is almost surely finite and before the continuation threshold is hit that $h \circ g_\tau^{-1} - \chi \arg (g_\tau^{-1})'$ is a GFF on $\h$ with boundary conditions given by
\begin{align*}
   -\lambda\left(1+ \sum_{i=0}^j \rho_{i,L}\right) \quad\text{on}\quad (V_\tau^{j+1,L},V_\tau^{j,L}] \quad\text{and}\quad \lambda \left(1+\sum_{i=0}^j \rho_{i,R}\right) \quad\text{on}\quad (V_\tau^{j,R},V_\tau^{j+1,R}].
\end{align*}
Here, we take $V_\tau^{0,L} = V_\tau^{0,R} = W_\tau$, $V_\tau^{\ell+1,L}=-\infty$, and $V_\tau^{r+1,R} = +\infty$.  As before, on the left (resp.\ right) hand side, $j$ varies between $0$ and $\ell$ (resp.\ $r$).

\begin{figure}
\begin{center}
\includegraphics[scale=0.85]{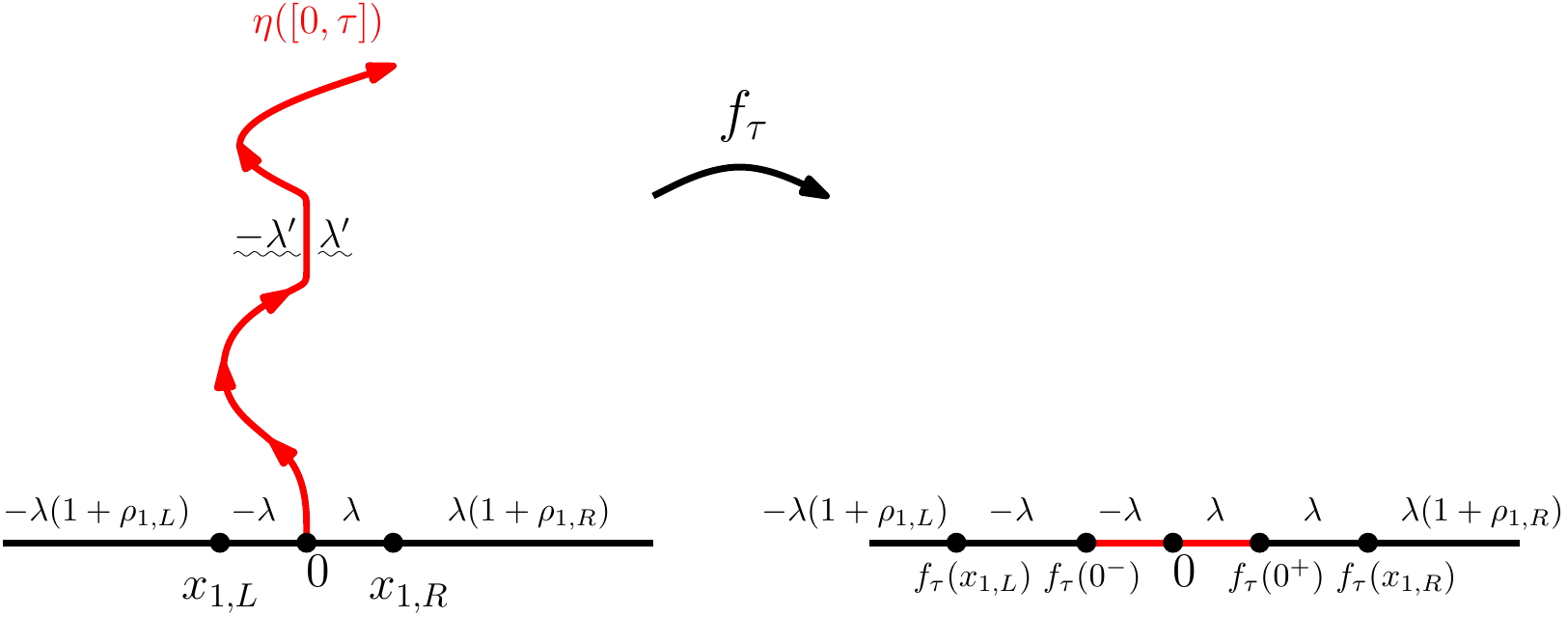}	
\end{center}
\vspace{-0.03\textheight}
\caption{\label{fig:bd} An illustration of the coupling of an $\SLE_\kappa(\ul{\rho})$ process as the flow line of a GFF $h$ on $\h$ from $0$ to $\infty$.  In this case, $\ell=r=1$ (i.e., the process has one force point on each side of the seed).  The conformal map $f_\tau \colon \h \setminus \eta([0,\tau]) \to \h$ shown is given by $f_\tau = g_\tau - W_\tau$.}
\end{figure}

See Figure~\ref{fig:bd} for an illustration in the case that $\ell=r=1$.

If $D$ is a simply connected domain and $x,y \in \partial D$ are distinct, then one can also realize the flow line of a GFF $h$ on $D$ as an $\SLE_\kappa(\ul{\rho})$ process from $x$ to $y$ provided one chooses the boundary data for $h$ appropriately.  Namely, one needs to take
\begin{equation}
\label{eqn:ig_change_of_coordinates}
h = \wt{h} \circ \varphi^{-1} - \chi \arg (\varphi^{-1})'
\end{equation}
where $\varphi \colon \h \to D$ is a conformal transformation with $\varphi(0) = x$ and $\varphi(x) = \infty$ and $\wt{h}$ is a GFF on $\h$ whose flow line from $0$ to $\infty$ is an $\SLE_\kappa(\ul{\rho})$ process.

The formula~\eqref{eqn:ig_change_of_coordinates} is the change of coordinates formula for imaginary geometry.

One can also consider flow lines of a GFF with different angles.  More specifically, the flow line of a GFF $h$ with angle $\theta$ is the $\SLE_\kappa(\ul{\rho})$ process coupled with the field $h+ \theta \chi$.  It has the interpretation as being the flow line of the (formal) vector field $e^{i(h/\chi + \theta)}$ (i.e., where all of the arrows have been rotated by the angle $\theta$).

In \cite{MS_IMAG}, it is described how flow lines with different angles and starting points interact with each other.  In particular, if $\eta_1,\eta_2$ are flow lines of a GFF $h$ on $\h$ starting from $x_1 \leq x_2$ with angles $\theta_1,\theta_2$, then \cite[Theorem~1.5]{MS_IMAG} implies that:
\begin{itemize}
\item $\eta_1$ stays to the left of $\eta_2$ if $\theta_1 > \theta_2$.  If $\theta_1 \in (\theta_2,\theta_2+2\lambda'/\chi)$, then $\eta_1$ can intersect and bounce off $\eta_2$.  If $\theta_1 \geq \theta_2+2\lambda'/\chi$, then $\eta_1$ and $\eta_2$ do not intersect.  (Note that $2\lambda'/\chi = \pi \kappa / (4-\kappa)$.)
\item $\eta_1$ merges with $\eta_2$ (and the paths do not subsequently separate) upon their first intersection if $\theta_1=\theta_2$.
\item $\eta_1$ crosses (and does not cross back but may bounce off) $\eta_2$ from left to right upon intersecting if $\theta_1 \in (\theta_2-\pi,\theta_2)$.
\end{itemize}
Thus to determine the manner in which flow lines interact with each other, one needs to compute the difference between their angles and then check in which of the aforementioned three ranges the difference falls into.

\begin{figure}
\begin{center}
\includegraphics[scale=0.85]{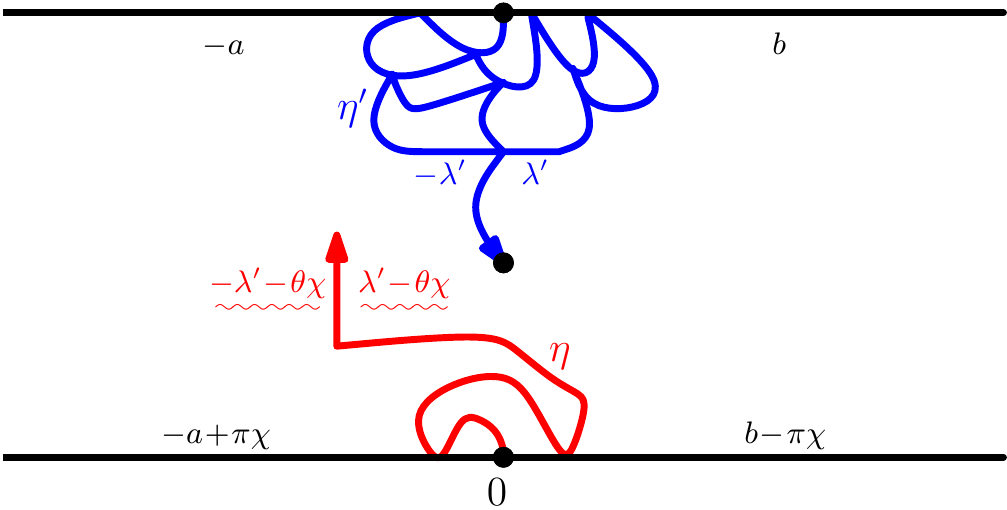}
\end{center}
\vspace{-0.03\textheight}
\caption{\label{fig:cfl} Illustration of a counterflow line $\eta'$ and a flow line $\eta$ with angle $\theta$ on the strip $\R \times [0,\pi]$.  In this case, $\eta'$ is an $\SLE_{\kappa'}(\rho_L';\rho_R')$ process with $\rho_L' = \tfrac{b}{\lambda'}-1$ and $\rho_R' = \tfrac{a}{\lambda'}-1$ and $\eta$ is an $\SLE_\kappa(\rho_L;\rho_R)$ with $\rho_L = \tfrac{a-\theta\chi}{\lambda} + \tfrac{\kappa}{2}-3$ and $\rho_R = \tfrac{b+\theta\chi}{\lambda} + \tfrac{\kappa}{2}-3$.  If $\theta=-\tfrac{\pi}{2}$, then $\eta$ is equal to the right boundary of $\eta'$ and if $\theta =\tfrac{\pi}{2}$, then $\eta$ is equal to the left boundary of $\eta'$.  For all intermediate values of $\theta \in [-\tfrac{\pi}{2},\tfrac{\pi}{2}]$, $\eta$ is contained in the range of $\eta'$ and $\eta'$ visits the range of $\eta$ in the opposite order in which the points are drawn by $\eta$.  If $\theta \notin [-\tfrac{\pi}{2},\tfrac{\pi}{2}]$, then $\eta$ is either to the left or right of $\eta'$.}	
\end{figure}

An $\SLE_{\kappa'}(\ul{\rho}')$ process $\eta'$ can similarly be coupled with a GFF $h$ on $\h$.  In this case, the boundary data is given by
\begin{align*}
   \lambda'\left(1+ \sum_{i=0}^j  \rho_{i,L}'\right) \quad\text{on}\quad (x_{j+1,L},x_{j,L}] \quad\text{and}\quad -\lambda'\left(1+ \sum_{i=0}^j  \rho_{i,R}'\right) \quad\text{on}\quad (x_{j,R},x_{j+1,R}].
\end{align*}
Such an $\SLE_{\kappa'}(\ul{\rho}')$ process is referred to as a {\bf counterflow line} of $h$, the reason being that it can be realized as a tree of flow lines.  One makes sense of $\SLE_{\kappa'}(\ul{\rho}')$ processes coupled with GFFs on other domains $D$ as counterflow lines using the change of coordinates formula~\eqref{eqn:ig_change_of_coordinates}.  Here, $\chi = \tfrac{\sqrt{\kappa'}}{2} - \tfrac{2}{\sqrt{\kappa'}}$.  Note that this is the same as the value of $\chi$ associated with $\kappa=\tfrac{16}{\kappa'}$.  It is often convenient to apply the change of coordinates $z \mapsto -1/z$ so that the counterflow line grows from $\infty$ to $0$.  In this case, the left (resp.\ right) boundary of $\eta'$ is given by the flow line of $h$ from $0$ to $\infty$ with angle $\tfrac{\pi}{2}$ (resp.\ $-\tfrac{\pi}{2}$) \cite[Theorem~1.4]{MS_IMAG}.  More generally, it follows from \cite[Theorem~1.4]{MS_IMAG} that the entire range of $\eta'$ can be realized as the light cone of flow lines starting from $0$ which are allowed to change angles but with angle always constrained to be in $[-\tfrac{\pi}{2},\tfrac{\pi}{2}]$.

When illustrating GFF flow lines, it is often convenient to use the notation $\uwave{x}$ to indicate the boundary data for the GFF.  It is used to indicate the boundary data along a flat segment of the domain boundary and means that the boundary data changes according to $\chi$ times its winding relative to the flat part.  This notation is described in detail in \cite[Figure~1.10]{MS_IMAG}.

\section{Upper bound}
\label{sec::upper_bound}

In this section, we will prove the upper bound of Theorem~\ref{thm::lightcone_dimension} and Theorem~\ref{thm::sle_kappa_rho_dimension}.  We will then explain how to extract Theorem~\ref{thm::fan_dimension} (in its entirety) from the upper bound.  We will begin in Section~\ref{subsec::derivative_estimate} by recording an estimate of the moments of the derivative of the Loewner map when an $\SLE_\kappa$ process gets close to a given point (Proposition~\ref{prop::derivative_estimate}).  Next, in Section~\ref{subsec::martingale_conditioning} we will derive the exponent for the probability that two flow lines of the GFF (with a particular choice of boundary data) starting from $\pm \tfrac{1}{2} \epsilon$ do not intersect before hitting $\partial \D$ as $\epsilon \to 0$ (Lemma~\ref{lem::non_intersection_exponent}).  We will then combine these results to establish the upper bounds for the dimensions in Section~\ref{subsec::dimension_upper_bound}.  

\label{subsec::overview}

\begin{figure}[ht!]
\begin{center}
\includegraphics[scale=0.85]{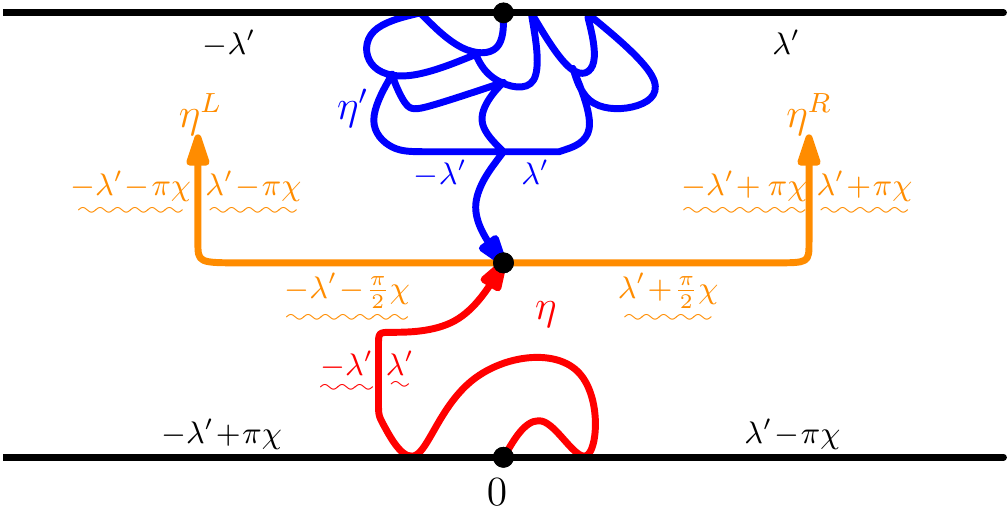}
\end{center}
\vspace{-0.03\textheight}
\caption{\label{fig::flowline_hit_point}
Suppose that $h$ is a GFF on the strip $\hstrip = \R \times [0,1]$ with the illustrated boundary data.  Then the counterflow line $\eta'$ of $h$ starting from $i$ is an $\SLE_{\kappa'}$ process and the flow line starting from $0$ is an $\SLE_\kappa(\tfrac{3\kappa}{4}-3;\tfrac{3\kappa}{4}-3)$ process.  (Note that $\eta$ hits both $\R_-$ and $\R_+$ since $\tfrac{3\kappa}{4}-3 < \tfrac{\kappa}{2}-2$ for all $\kappa \in (0,4)$.)  Since the range of $\eta'$ almost surely contains the range of $\eta$, in order for $\eta$ to hit a given point $z$, it must be that $\eta'$ hits $z$.  Moreover, on the event that $\eta$ hits $z$, the flow lines $\eta^L$ and $\eta^R$ starting from $z$ with angles $\pi$ and $-\pi$, respectively, stay to the left and right of $\eta$, respectively, and the left side of $\eta^L$ does not intersect the right side of $\eta^R$.  The converse statement also holds: if $\eta'$ hits a given point $z$ and the left side of the flow line $\eta^L$ starting from $z$ with angle $\pi$ does not intersect the right side of the flow line $\eta^R$ starting from $z$ with angle $-\pi$, then $\eta$ must hit $z$ because $\eta$ cannot hit the left (resp.~right) side of $\eta^L$ (resp.~$\eta^R$).
}
\end{figure}

\begin{figure}[ht!]
\begin{center}
\includegraphics[scale=0.85]{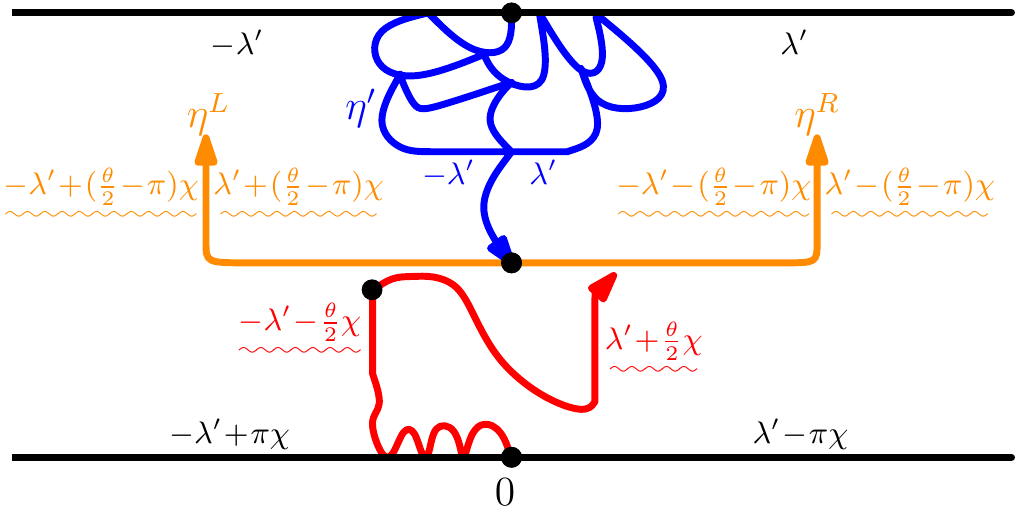}
\end{center}
\vspace{-0.03\textheight}
\caption{\label{fig::lightcone_hit_point}
(Continuation of Figure~\ref{fig::flowline_hit_point}.)  In order for the light cone $\lightcone(\theta)$ of $h$ starting from $0$ with opening angle $\theta \in [0,\pi]$ to a hit point $z$, it must be that $\eta'$ hits $z$ and that the left side of the flow line $\eta^L$ starting from $z$ with angle $-\tfrac{\theta}{2}+\pi$ does not hit the right side of the flow line $\eta^R$ starting from $z$ with angle $\tfrac{\theta}{2}-\pi$.  Note that the angles $-\tfrac{\theta}{2}+\pi$ and $\tfrac{\theta}{2}-\pi$ are dual to the angles of the left and right sides of $\lightcone(\theta)$, respectively.  The converse statement also holds: if $\eta'$ hits a given point $z$ and the left side of the flow line $\eta^L$ with angle $-\tfrac{\theta}{2}+\pi$ starting from $z$ does not intersect the right side of the flow line $\eta^R$ with angle $\tfrac{\theta}{2}-\pi$ starting from $z$ then $\lightcone(\theta)$ must contain $z$.  The reason is that a flow line with angle $-\tfrac{\theta}{2}$ cannot hit the left side of $\eta^L$ and a flow line with angle $\tfrac{\theta}{2}$ cannot hit the right side of $\eta^R$.  Hence, an angle-varying flow line which gets arbitrarily close to $z$ can be generated by taking a path which starts off with angle $\tfrac{\theta}{2}$ until it gets very close to $\eta^L$, then travels with angle $-\tfrac{\theta}{2}$ until getting very close to $\eta^R$, etc.  Indeed, such a path will travel back and forth between $\eta^L$ and $\eta^R$ and get progressively closer to $z$ with each pass.  The figure shows such an angle-varying flow line in red where the angle changes from $ \tfrac{\theta}{2}$ to $-\tfrac{\theta}{2}$ at the black dot.
}
\end{figure}

\begin{figure}[ht!]
\begin{center}
\includegraphics[scale=0.85]{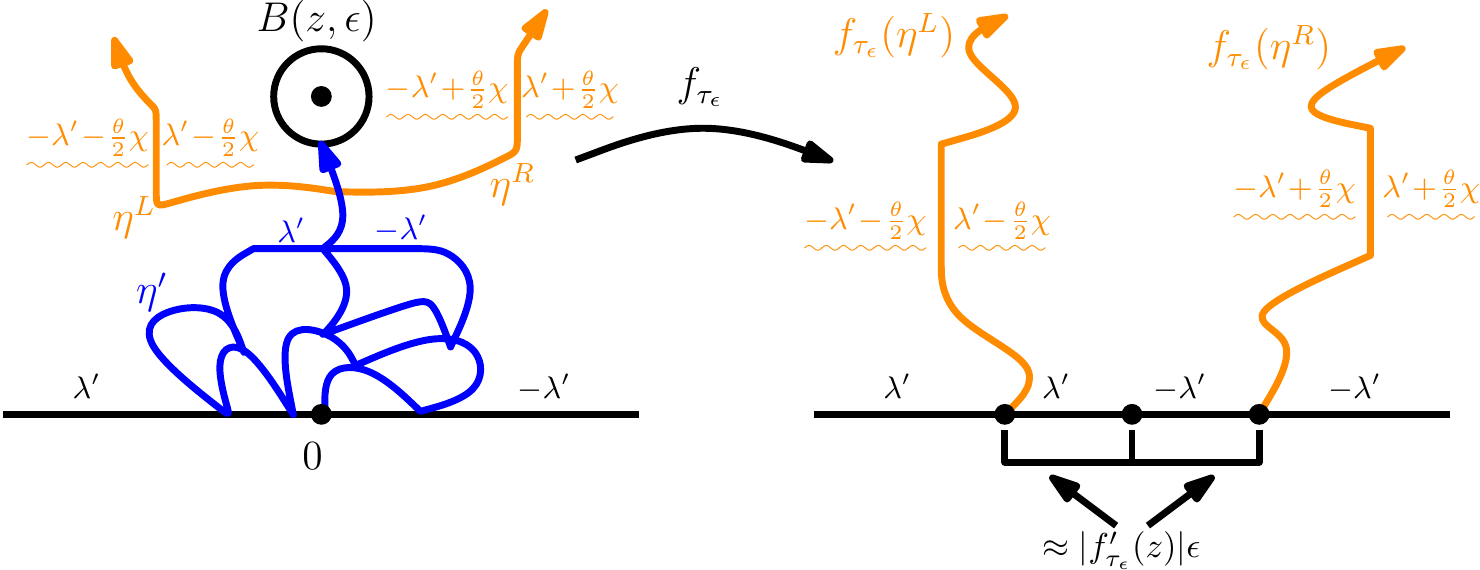}
\end{center}
\vspace{-0.03\textheight}
\caption{\label{fig::counterflowline_hit_conformal_maps}
This is a rotation of the picture in Figure~\ref{fig::flowline_hit_point} and Figure~\ref{fig::lightcone_hit_point} by $180$ degrees and conformally mapped so that we are working in $\h$.  Here, $h$ is a GFF on $\h$ with the boundary data shown, $\eta'$ is its counterflow line starting from $0$, and $\eta^L$ (resp.\ $\eta^R$) is the flow line with angle $\tfrac{\theta}{2}$ (resp.\ $-\tfrac{\theta}{2}$) starting from a point on the left (resp.\ right) side of $\eta'$.  To estimate the probability that $\lightcone(\theta)$ hits $\partial B(z,\epsilon)$ for a given point $z$, we first estimate the probability that $\eta'$ hits $\partial B(z,\epsilon)$, say for the first time at time $\tau_\epsilon$.  Given this, we estimate the probability that $\eta^L$ and $\eta^R$ both reach a macroscopic distance before hitting each other.  Let $(g_t)$ denote the chordal Loewner evolution of $\eta'$, $W$ its Loewner driving function, and let $f_t = g_t - W_t$ be its centered Loewner evolution.  We can estimate the latter probability by first conformally mapping away $\eta'|_{[0,\tau_\epsilon]}$ using $f_{\tau_\epsilon}$ and then estimating the probability that the images of $\eta^L$ and $\eta^R$ avoid each other until reaching a macroscopic size using Lemma~\ref{lem::non_intersection_exponent}.
}
\end{figure}

The main result of this section is the following proposition.

\begin{proposition}
\label{prop::dimension_upper_bound}
Suppose that $h$ is a GFF on $\h$ with piecewise constant boundary data which changes values at most a finite number of times.  Let $\lightcone(\theta)$ be the $\SLE_\kappa$ light cone ($\kappa \in (0,4)$) of $h$ starting from $0$ with opening angle $\theta \in [0,\pi]$.  Almost surely,
\[ \dimH(\lightcone(\theta)) \leq d(\kappa,\theta) \wedge 2\]
where $d(\kappa,\theta)$ is as in~\eqref{eqn::lightcone_dimension_formula}.  In particular, the dimension of an $\SLE_\kappa(\rho)$ process with $\rho \in ([\tfrac{\kappa}{2}-4) \vee (-2-\tfrac{\kappa}{2}),-2)$ is bounded from above by the expression in~\eqref{eqn::sle_kappa_rho_dimension}.
\end{proposition}

We are now going to give an overview of how the estimates proved in this section will be used to establish Proposition~\ref{prop::dimension_upper_bound}.  Fix $\kappa \in (0,4)$ and $\theta \in [0,\pi]$.  If $d(\kappa,\theta) \geq 2$, then the upper bound given is trivially true.  Consequently, we may assume without loss of generality that $\theta \in [0,\pi]$ is such that $d(\kappa,\theta) < 2$.  We are going to prove the result by combining Proposition~\ref{prop::derivative_estimate} with Lemma~\ref{lem::non_intersection_exponent}.  For the proof of Proposition~\ref{prop::dimension_upper_bound} it will be more convenient to perform a change of coordinates which swaps~$0$ and~$\infty$ so that~$\lightcone(\theta)$ grows from~$\infty$ towards~$0$ rather than from~$0$ towards~$\infty$.  By the absolute continuity properties of the GFF \cite[Proposition~3.2]{MS_IMAG}, we may assume without loss of generality that the boundary data for~$h$ is as described in the left side of Figure~\ref{fig::counterflowline_hit_conformal_maps}.  Let~$\eta'$ be the counterflow line of~$h$ from~$0$ to~$\infty$, $(g_t)$ its chordal Loewner evolution, $W$ its Loewner driving function, and let $f_t = g_t - W_t$ be its centered Loewner evolution.  It is explained in Figures~\ref{fig::flowline_hit_point}--\ref{fig::counterflowline_hit_conformal_maps} that in order for~$\lightcone(\theta)$ to get within distance~$\epsilon$ of a given point $z \in \h$, it must be that~$\eta'$ gets within distance~$\epsilon$ of~$z$ and the flow lines~$\eta^L$ and~$\eta^R$ with angles~$\tfrac{\theta}{2}$ and~$-\tfrac{\theta}{2}$, respectively, starting near the tip of~$\eta'$ do not intersect each other.  As explained in Figure~\ref{fig::counterflowline_hit_conformal_maps}, the exponent for this probability can be estimated by computing the moments of~$|f_{\tau_\epsilon}'(z)|$ and computing the exponent for the probability of the event that two GFF flow lines starting close to each other do not intersect before reaching a macroscopic distance from their starting points.

\begin{remark}
\label{rem::critical_angle}
Recall that the critical angle~$\theta_c$ from~\eqref{eqn::space_filling_angle} is the angle at or above which it is claimed in Theorem~\ref{thm::lightcone_dimension} that~$\dimH(\lightcone(\theta))$ is almost surely equal to~$2$.  The reason that this is the case is that when $\theta \geq \theta_c$, the left side of the flow line~$\eta^L$ cannot intersect the right side of the flow line~$\eta^R$ from Figures~\ref{fig::flowline_hit_point}--\ref{fig::counterflowline_hit_conformal_maps}.
\end{remark}

\subsection{Derivative estimate}
\label{subsec::derivative_estimate}

Fix~$\kappa > 0$ and suppose that~$\eta$ is an~$\SLE_\kappa$ process in~$\h$ from~$0$ to~$\infty$.  For each $r \in \R$, we let
\[ \nu = \nu(r) = \frac{r^2}{4} \kappa + r\left(1-\frac{\kappa}{4}\right) \quad\text{and}\quad \xi = \xi(r) = \frac{r^2}{8} \kappa.\]
Let $(g_t)$ be the Loewner evolution associated with~$\eta$, let~$W$ be its Loewner driving function, and let $f_t = g_t - W_t$ be its centered Loewner evolution.  For each $z \in \h$, we let
\[ Z_t = Z_t(z) = X_t + iY_t =f_t(z)\]
and
\begin{equation}
\label{eqn::mg_processes}
 \Delta_t = |g_t'(z)|,\quad \Upsilon_t = \frac{Y_t}{|g_t'(z)|},\quad \Theta_t = \arg Z_t, \quad\text{and}\quad S_t = \sin \Theta_t.
\end{equation}

\begin{proposition}
\label{prop::derivative_estimate}
For each~$\epsilon > 0$, let $\tau_\epsilon = \inf\{ t \geq 0 : |\eta(t) - z| \leq \epsilon\}$.  We have that
\begin{equation}
\label{eqn::derivative_estimate}
 \E \big[|g_{\tau_\epsilon}'(z)|^{\nu+r} \one_{\{\tau_\epsilon < \infty\}} \big] \asymp \epsilon^{-\xi-r}
\end{equation}
where the constants in~$\asymp$ depend only on~$\kappa$ and~$r$.  For each $R \geq 2|z|$, we also let~$\sigma_R = \inf\{t \geq 0: |\eta(t)| = R\}$.  Fix~$\delta \in (0,\tfrac{\pi}{2})$ and assume that $\arg(z) \in (\delta,\pi-\delta)$.  We also have that
\begin{equation}
\label{eqn::derivative_estimate_escape}
 \E \big[|g_{\tau_\epsilon}'(z)|^{\nu+r} \one_{\{\tau_\epsilon \leq \sigma_R < \infty\}} \big] \asymp \epsilon^{-\xi-r}
\end{equation}
where the constants in~$\asymp$ depend only on~$\kappa$,~$r$, and~$\delta$.
\end{proposition}
\begin{proof}
See \cite[Section~6.3]{LV09} for a proof of~\eqref{eqn::derivative_estimate} and \cite[Lemma~4.2]{MW_INTERSECTIONS} for a proof of~\eqref{eqn::derivative_estimate_escape}.
\end{proof}

This result can be derived from \cite[Proposition~6.1]{LV09}, which gives that
\[ M_t = |Z_t|^r Y_t^\xi \Delta_t^\nu = S_t^{-r} \Upsilon_t^{\xi+r} \Delta_t^{\nu+r}\]
is a local martingale.  This martingale also appears in \cite[Theorem~3 and Theorem~6]{SCHRAMM_WILSON} and is part of the same family of martingales that we will use in Section~\ref{subsec::martingale_conditioning} to get the exponent for the probability that two GFF flow lines do not intersect each other before making it to a macroscopic distance from their starting points.

\subsection{Non-intersection exponent}
\label{subsec::martingale_conditioning}

\begin{figure}[ht!]
\begin{center}
\includegraphics[scale=0.85]{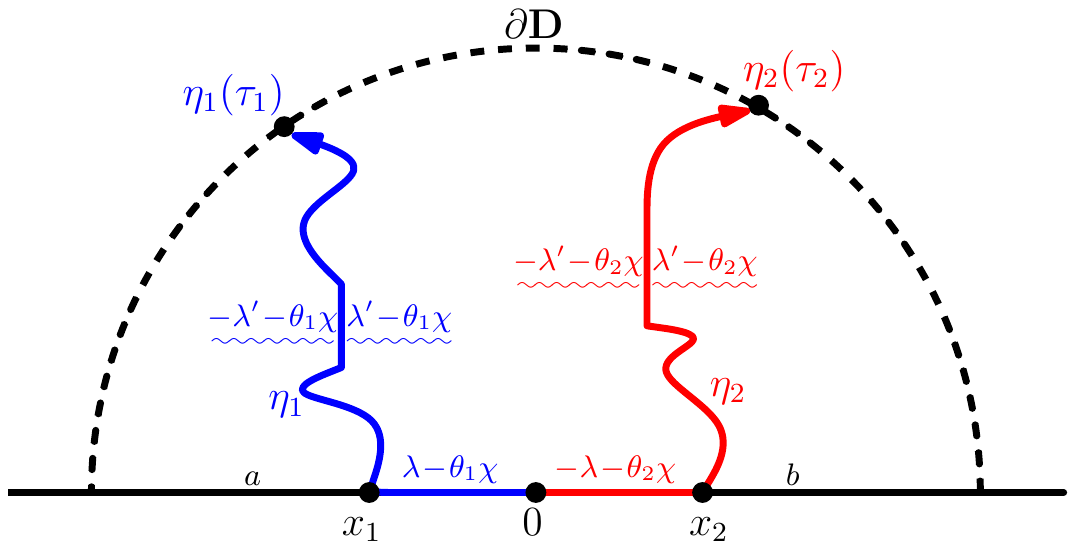}
\end{center}
\caption{\label{fig::non_intersection_exponent} Suppose that~$h$ is a GFF on~$\h$ with the boundary data depicted above.  For angles $\theta_1,\theta_2$, we let $\eta_i$ for $i=1,2$ be the flow line of $h$ starting at $x_i$ with angle $\theta_i$ and let $\tau_i$ (resp.\ $\wh{\tau}_i$) be the first time that $\eta_i$ hits $\partial \D$ (resp.\ $\partial (\tfrac{1}{2} \D)$).  Assume that $\theta_1,\theta_2$ are chosen so that $\eta_1$ can hit $\eta_2$, i.e. $\theta_1 \in (\theta_2-\pi,\theta_2+2\lambda'/\chi)$.  In Proposition~\ref{prop::non_intersection_exponent} we compute the exponent for the probability of the event that $\eta_1([0,\tau_1]) \cap \eta_2([0,\tau_2]) \neq \emptyset$ as $x_1,x_2 \to 0$.   The main step of the proof is Lemma~\ref{lem::non_intersection_exponent}, which gives the exponent in the special case that $\dist(\eta_1([\wh{\tau}_1,\tau_1]),\eta_2([\wh{\tau}_2,\tau_2]))$ is bounded from below by a positive constant and $\eta_i|_{[0,\tau_i]}$ for $i=1,2$ both do not get too close to $0$.
}
\end{figure}

\begin{figure}[ht!]
\begin{center}
\includegraphics[scale=0.85]{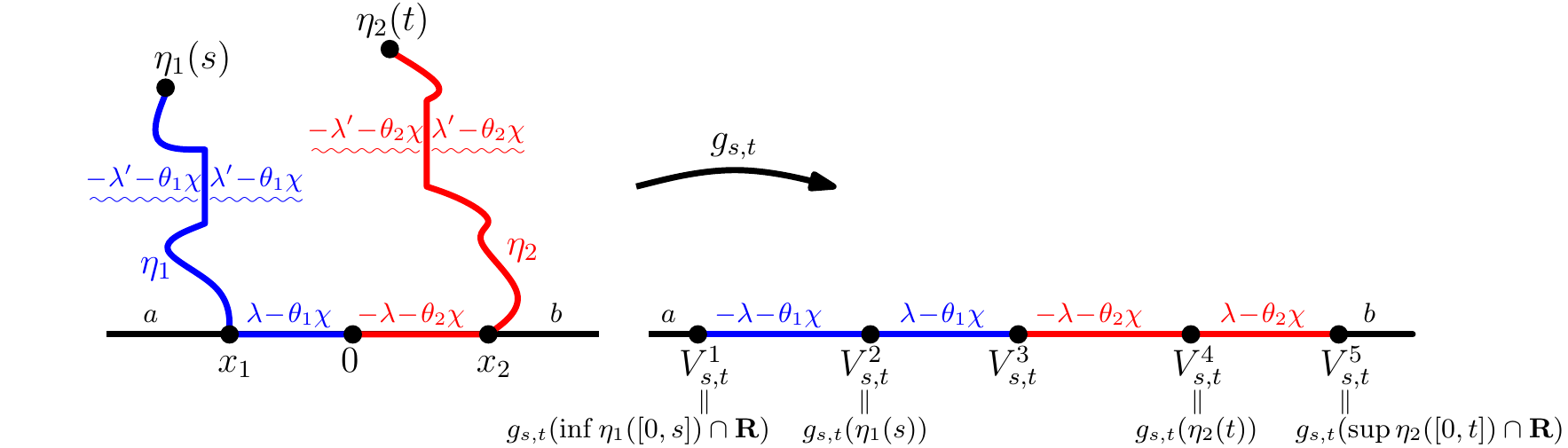}
\end{center}
\vspace{-0.025\textheight}
\caption{\label{fig::two_path_martingale}
(Continuation of Figure~\ref{fig::non_intersection_exponent}.)  For each $s,t \geq 0$, let $g_{s,t}$ be the unique conformal map which takes the unbounded connected component of $\h \setminus (\eta_1([0,s]) \cup \eta_2([0,t]))$ back to $\h$ satisfying $g_{s,t}(z) = z+o(1)$ as $z \to \infty$.  Let $V_{s,t}^1 = g_{s,t}(\inf \eta_1([0,s]) \cap \R)$, $V_{s,t}^2 = g_{s,t}(\eta_1(s))$, $V_{s,t}^3$ be the image of the most recent intersection of $\eta_1|_{[0,s]}$ and $\eta_2|_{[0,t]}$ if they intersect, or $\eta_1|_{[0,s]}$ and $[0,x_2]$ if they intersect, or $\eta_2|_{[0,t]}$ and $[x_1,0]$ if they intersect, and otherwise be equal to $g_{s,t}(0)$, $V_{s,t}^4 = g_{s,t}(\eta_2(t))$, and $V_{s,t}^5 = g_{s,t}(\sup \eta_2([0,t]) \cap \R)$.  Let $\rho_i$ be such that as one traces $\R$ from left to right, the heights in the right side jump by $\rho_i \lambda$ at the points $V_{s,t}^i$.  The product $M_{s,t} := \prod_{j \neq 3} |V_{s,t}^j - V_{s,t}^3|^{(\wt{\rho}_3 - \rho_3)\rho_j/(2\kappa)}$ evolves as a martingale in each of $s$ and $t$ separately where $\wt{\rho}_3=\kappa-4-\rho_3$ \cite{SCHRAMM_WILSON}.  Reweighting the law of the pair of paths $(\eta_1,\eta_2)$ by $M_{s,t}$ yields the law of a new pair of paths which almost surely do not intersect each other.  When one path is fixed the evolution of the other in the weighted law is the same as in the unweighted law except with $\rho_3$ replaced by $\wt{\rho}_3$.  As explained in more detail in the proof of Lemma~\ref{lem::non_intersection_exponent}, this new pair of paths can be constructed as flow lines of a GFF with modified boundary data and angles.
}
\end{figure}

In this section, we are going to derive the exponent for the probability that two flow lines of the GFF starting from $\pm \tfrac{1}{2} \epsilon$ do not intersect before hitting $\partial \D$ as $\epsilon \to 0$ (see Figure~\ref{fig::non_intersection_exponent} for an illustration of the setup).  The main result is:

\begin{proposition}
\label{prop::non_intersection_exponent}
Fix $\epsilon > 0$ and let $x_1 = -\tfrac{1}{2}\epsilon$ and $x_2 = \tfrac{1}{2}\epsilon$.  Let $\theta_1,\theta_2$ be angles with $\theta_1 \in (\theta_2-\pi,\theta_2+2\lambda'/\chi)$.  Suppose that $h$ is a GFF on $\h$ with the boundary data illustrated in Figure~\ref{fig::non_intersection_exponent} where $a,b \in \R$ are constants so that $\eta_1$, $\eta_2$ do not hit the continuation threshold immediately almost surely.  For $i=1,2$, let $\tau_i$ be the first time that $\eta_i$ hits $\partial \D$ and let $F = \{\eta_1([0,\tau_1]) \cap \eta_2([0,\tau_2]) = \emptyset\}$.
Let
\begin{equation}
\label{eqn::non_intersection_exponent}
\alpha = \frac{1}{2\kappa}\left(\kappa-4-2\rho \right)\left(\frac{b-a}{\lambda} - \rho \right) \quad\text{and}\quad \rho = \frac{(\theta_1-\theta_2) \chi}{\lambda} -2.
\end{equation}
Then we have that
\[ \p[F] = \epsilon^{\alpha+o(1)}\]
where the $o(1)$ term tends to zero as $\epsilon \to 0$ at a rate depending only on $\theta_1$, $\theta_2$, $\kappa$, and $a,b$.
\end{proposition}

We will not use Proposition~\ref{prop::non_intersection_exponent} as stated in the proof of Proposition~\ref{prop::dimension_upper_bound} and have included it just for completeness.  The main ingredient in its proof is Lemma~\ref{lem::non_intersection_exponent} which gives the corresponding estimate in the special case that the segments of the paths between first hitting $\partial (\tfrac{1}{2} \D)$ and $\partial \D$ have positive distance from each other and neither path gets too close to $0$ before exiting $\partial \D$.  From Lemma~\ref{lem::non_intersection_exponent}, we will prove Lemma~\ref{lem::non_intersection_exponent_general} which is a version which holds with more general boundary data and is the estimate that we will actually make use of in this article.

\begin{lemma}
\label{lem::non_intersection_exponent}
Suppose that we have the same setup as in Proposition~\ref{prop::non_intersection_exponent}.  Let $\wh{\tau}_i$ be the first time that $\eta_i$ hits $\partial (\tfrac{1}{2} \D)$ for $i=1,2$.  For each $\delta \in (0,\tfrac{1}{2})$, let $E_\delta$ be the event that
\begin{enumerate}[(i)]
\item\label{it::non_intersection} $\eta_1([0,\tau_1]) \cap \eta_2([0,\tau_2]) = \emptyset$,
\item\label{it::good_exit_distance} $\dist(\eta_1([\wh{\tau}_1,\tau_1]), \eta_2([\wh{\tau}_2,\tau_2])) \geq \delta$, and
\item\label{it::no_small_ball} $\eta_i([0,\tau_i]) \cap (\delta \epsilon \D) = \emptyset$ for $i=1,2$.
\end{enumerate}
Then
\begin{equation}
\label{eqn::non_intersection_probability}
 \p[E_\delta] \asymp \epsilon^\alpha
\end{equation}
where $\alpha$ is as in~\eqref{eqn::non_intersection_exponent} and the constants in $\asymp$ depend only on $\delta$, $\theta_1$, $\theta_2$, $\kappa$, and $a,b$.
\end{lemma}

Before we prove Lemma~\ref{lem::non_intersection_exponent}, we need to collect the following lemma.

\begin{lemma}
\label{lem::force_points_separated}
Suppose that we have the same setup as in Lemma~\ref{lem::non_intersection_exponent} and let $V_{s,t}^j$ for $s,t \geq 0$ and $j \in \{1,\ldots,5\}$ be as in Figure~\ref{fig::two_path_martingale}.  There exists a universal constant $A > 0$ such that
\begin{align}
\label{eqn::vst_ubd}
V_{\tau_1,\tau_2}^{j+1} - V_{\tau_1,\tau_2}^j &\leq A \quad\text{for}\quad j \in \{1,\ldots, 4\}.
\intertext{Moreover, there exists a constant $B_\delta > 0$ depending only on $\delta \in (0,\tfrac{1}{2})$ such that}
\label{eqn::vst_lbd}
\big| V_{\tau_1,\tau_2}^j - V_{\tau_1,\tau_2}^3 \big| &\geq B_\delta \quad\text{for}\quad j \in \{1,2,4,5\} \quad\text{on}\quad E_\delta.
\end{align}
\end{lemma}
\begin{proof}
For each $z \in \C$, let $\p^z$ denote the law of a standard planar Brownian motion $B$ starting from $z$ which is independent of $\eta_1$ and $\eta_2$.  For each $s,t \geq 0$, let $\h_{s,t}$ be the unbounded connected component of $\h \setminus (\eta_1([0,s]) \cup \eta_2([0,t]))$ and, for each $1 \leq j \leq 4$, let $A_{s,t}^{j+1}$ be the segment of $\partial \h_{s,t}$ which connects $g_{s,t}^{-1}(V_{s,t}^j)$ to $g_{s,t}^{-1}(V_{s,t}^{j+1})$ in the clockwise direction.  By \cite[Remark~3.50]{LAW05}, we have that
\[ V_{s,t}^{j+1} - V_{s,t}^j = \lim_{y \to \infty} \pi y \p^{iy}[ B \text{ exits } \h_{s,t} \text{ in } A_{s,t}^{j+1}].\]
Since $A_{s,t}^{j+1} \subseteq \ol{\D}$, this, in turn, implies for $0 \leq s \leq \tau_1$ and $0 \leq t \leq \tau_2$ that
\[ V_{s,t}^{j+1} - V_{s,t}^j \leq \lim_{y \to \infty} \pi y \p^{iy}[ B \text{ exits } \h \setminus \D \text{ in } \partial \D] < \infty.\]
This proves~\eqref{eqn::vst_ubd}.

Let $\phi_i = \arg(\eta_i(\tau_i))$ for $i=1,2$ and let $\psi_1 = \phi_1 - \tfrac{\delta}{4}$ and $\psi_2 = \phi_2 + \tfrac{\delta}{4}$.  We note that the probability that a Brownian motion starting from $iy$ exits $\h \setminus \D$ in $\partial \D$ with argument contained in $[\psi_2,\psi_1]$ is at least a $\delta$-dependent constant times $1/y$.  From this and Condition~\eqref{it::good_exit_distance} in the definition of $E_\delta$, \eqref{eqn::vst_lbd} follows.
\end{proof}

\begin{proof}[Proof of Lemma~\ref{lem::non_intersection_exponent}]
Let $\rho_1,\ldots,\rho_5$ be such that the jumps in the heights from left to right in the right side of Figure~\ref{fig::two_path_martingale} are equal to $\rho_i \lambda$.  Explicitly, the values of the $\rho_i$ are given by
\begin{align*}
 \rho_1 &= \frac{-\theta_1 \chi-a}{\lambda} -1,\quad \rho_2 = 2,\quad \rho_3 = \frac{(\theta_1-\theta_2)\chi}{\lambda}-2,\\
 &\quad\quad\quad \rho_4 = 2 \quad\text{and}\quad \rho_5 = \frac{b+\theta_2\chi}{\lambda} - 1.
\end{align*}
Let 
\[ \wt{\rho}_3 = \kappa-4-\rho_3 = \kappa-2 + \frac{(\theta_2-\theta_1)\chi}{\lambda}\]
be the reflection of $\rho_3$ about the value $\tfrac{\kappa}{2}-2$.  By \cite[Theorem~6]{SCHRAMM_WILSON}, reweighting the law of $(\eta_1,\eta_2)$ by the local martingale
\begin{equation}
\label{eqn::martingale_condition_paths_not_to_hit}
 M_{s,t} = \prod_{j \neq 3} |V_{s,t}^j - V_{s,t}^3|^{(\wt{\rho}_3 - \rho_3)\rho_j/(2\kappa)}
\end{equation}
corresponds to changing $\rho_3$ to $\wt{\rho}_3$.  This yields a pair of paths $(\wh{\eta}_1,\wh{\eta}_2)$ which are flow lines of the GFF as shown in the left side of Figure~\ref{fig::two_path_martingale} where the values $\theta_1$ and $\theta_2$ (both as angles and as indicated in the boundary conditions) are replaced by
\[ \wh{\theta}_1 = \theta_1 \quad\text{and}\quad \wh{\theta}_2= 2\theta_1 - \theta_2 - \frac{4\lambda'}{\chi},\]
respectively, and $b$ is replaced by $\wh{b} = b+\lambda(\wt{\rho}_3-\rho_3)$.  In particular, the angle gap between $(\wh{\eta}_1,\wh{\eta}_2)$ is given by
\[ \wh{\theta}_1 - \wh{\theta}_2 = \theta_2 - \theta_1 + \frac{4\lambda'}{\chi} > \frac{2\lambda'}{\chi}\]
since we assumed that $\theta_2-\theta_1 > -2\lambda'/\chi$.  Thus, $(\wh{\eta}_1,\wh{\eta}_2)$ almost surely do not intersect each other \cite[Theorem~1.5]{MS_IMAG}.  Observe that $\alpha$ is equal to the sum of the exponents in the definition of $M_{s,t}$ from~\eqref{eqn::martingale_condition_paths_not_to_hit}:
\[ \alpha = \left(\frac{\wt{\rho}_3-\rho_3}{2\kappa} \right) \sum_{j \neq 3} \rho_j.\]

Let
\[ \wt{M}_t = M_{t \wedge \tau_1, t \wedge \tau_2}\]
Since $M_{s,t}$ is a local martingale, it follows that $\wt{M}_t$ is also a local martingale.  Lemma~\ref{lem::force_points_separated} implies that
\begin{equation}
\label{eqn::mtau1tau2_eps}
 M_{\tau_1,\tau_2} \asymp 1 \quad\text{on}\quad E_\delta
\end{equation}
where the constants in $\asymp$ depend only on $\delta$, $\kappa$, $\theta_1$, $\theta_2$, and $a,b$.  For each $u \geq 0$, let $T_u = \inf\{t \geq 0 : \wt{M}_t = u\}$.  It follows from~\eqref{eqn::mtau1tau2_eps} that there exists a constant $u_1 > 0$ depending on $\delta$, $\kappa$, $\theta_1$, $\theta_2$, and $a,b$ such that $E_\delta \subseteq \{T_{u_1} < T_0\}$.  Consequently,
\begin{align*}
    \p[E_\delta] \leq \p[ T_{u_1} < T_0] = \frac{1}{u_1} \E[ \wt{M}_{T_{u_1} \wedge T_0}] \asymp \epsilon^\alpha
\end{align*}
where the constants in $\asymp$ depend on $\delta$, $\theta_1$,  $\theta_2$, $\kappa$, and $a,b$.  This proves the upper bound in~\eqref{eqn::non_intersection_probability}.

We will now give the lower bound for $\p[E_\delta]$.  Lemma~\ref{lem::force_points_separated} implies that there exists a constant $u_2 > 0$ depending only on $\delta$, $\theta_1$,  $\theta_2$, $\kappa$, and $a,b$ such that, on $E_\delta$, we have that $T_{u_2} \leq \tau_1 \wedge \tau_2 \wedge T_0$.  We have,
\begin{align*}
      \p[E_\delta]
\geq  \p[E_\delta \giv T_{u_2} < T_0] \p[ T_{u_2} < T_0]. 
\end{align*}
It is easy to see that $\p[E_\delta \giv T_{u_2} < T_0]$ is bounded from below by universal positive constant depending only on $\delta$, $\kappa$, $\theta_1$, $\theta_2$, and $a,b$ using the results of \cite[Section~2]{MW_INTERSECTIONS}.  This gives the lower bound since, arguing as in the proof of the upper bound, we know that $\p[T_{u_2} < T_0] \asymp \epsilon^\alpha$ where the constants in $\asymp$ depend only on $\delta$, $\theta_1$, $\theta_2$, $\kappa$, and $a,b$.  This proves the desired result for $E_\delta$.
\end{proof}

In Lemma~\ref{lem::non_intersection_exponent}, we computed the exponent for the probability that two GFF flow lines starting from $\pm \tfrac{1}{2}\epsilon$ hit $\partial \D$ before intersecting each other or hitting $\partial (\delta \epsilon \D)$ as $\epsilon \to 0$ when the field has the boundary data illustrated in Figure~\ref{fig::non_intersection_exponent}.  We are now going to deduce from this and the Radon-Nikodym derivative estimate Lemma~\ref{lem::rn} that the same is true if we consider a field which has the same boundary data as illustrated in Figure~\ref{fig::non_intersection_exponent} outside of the interval $(-\tfrac{1}{2} \delta \epsilon, \tfrac{1}{2} \delta \epsilon)$ and has general, piecewise constant boundary data in $(-\tfrac{1}{2} \delta \epsilon, \tfrac{1}{2} \delta \epsilon)$.

\begin{lemma}
\label{lem::non_intersection_exponent_general}
Suppose that we have the same setup as Lemma~\ref{lem::non_intersection_exponent} except we take $h$ to be a GFF on $\h$ whose boundary conditions are piecewise constant, change values at most a finite number of times, are at most $\delta^{-1}$ in magnitude, and take the form illustrated in Figure~\ref{fig::non_intersection_exponent} outside of the interval $(-\tfrac{1}{2}\delta \epsilon, \tfrac{1}{2} \delta \epsilon)$.  Then
\begin{equation}
\label{eqn::non_intersection_probability_general}
 \p[E_\delta] \asymp \epsilon^\alpha
\end{equation}
where the constants in $\asymp$ depend only on $\delta$, $\theta_1$, $\theta_2$, $\kappa$, and $a,b$.
\end{lemma}
\begin{proof}
Suppose that $\wt{h}$ is a GFF whose boundary conditions are as in the statement of Lemma~\ref{lem::non_intersection_exponent}, let $\wt{\eta}_i$ for $i=1,2$ be the flow line of $\wt{h}$ starting from $x_i$ and let $\wt{\tau}_i$ be the first time that $\wt{\eta}_i$ hits $\partial \D$.  We also let $\sigma_i$ (resp.\ $\wt{\sigma}_i$) for $i=1,2$ be the first time that $\eta_i$ (resp.\ $\wt{\eta}_i$) hits $\partial (\delta \epsilon \D)$.  Let $\mu$ denote the law of $(\eta_1|_{[0,\tau_1 \wedge \sigma_1]},\eta_2|_{[0,\tau_2 \wedge \sigma_2]})$ and let $\wt{\mu}$ denote the law of $(\wt{\eta}_1|_{[0,\wt{\tau}_1 \wedge \wt{\sigma}_1]}, \wt{\eta}_2|_{[0,\wt{\tau}_2 \wedge \wt{\sigma}_2]})$.  It follows from Lemma~\ref{lem::rn} that $\mu$ and $\wt{\mu}$ are mutually absolutely continuous with
\[ \frac{1}{C} \leq \frac{d\mu}{d\wt{\mu}} \leq C \quad\text{on}\quad E_\delta\]
where $C \geq 1$ is a constant depending only on $\delta$, $\kappa$, $\theta_1$, $\theta_2$, and $a,b$.  The desired result follows since $\tau_i \leq \sigma_i$ for $i=1,2$ on $E_\delta$.
\end{proof}

\begin{proof}[Proof of Proposition~\ref{prop::non_intersection_exponent}]
We are going to establish the upper bound by iteratively applying Lemma~\ref{lem::non_intersection_exponent}.  Fix $\beta \in (0,1)$ and let $n_\beta = \lfloor \beta \log \tfrac{1}{\epsilon} \rfloor$.  For each $i=1,2$ and $1 \leq j \leq n_\beta$, we let $\tau_{i,j}$ (resp.\ $\wh{\tau}_{i,j}$) be the first time that $\eta_i$ hits $\partial (\epsilon^\beta e^j \D)$ (resp.\ $\partial (\tfrac{1}{2} \epsilon^\beta e^j \D)$) and let $G_{j,\delta}$ be the event that
\begin{enumerate}
\item $\eta_1([0,\tau_{1,j}]) \cap \eta_2([0,\tau_{2,j}]) = \emptyset$ and
\item either $\dist(\eta_1([\wh{\tau}_{1,j},\tau_{1,j}]),\eta_2([\wh{\tau}_{2,j},\tau_{2,j}])) < \delta \epsilon^\beta e^j$ or $\big(\cup_{i=1}^2 \eta_i([0,\tau_{i,j}])\big) \cap (\delta \epsilon \D) \neq \emptyset$.
\end{enumerate}
Let $\CF_j$ be the $\sigma$-algebra generated by $\eta_i|_{[0,\tau_{i,j}]}$ for $i=1,2$.  It is easy to see that there exists a function $p \colon [0,1] \to [0,1]$ with $p(\delta) \downarrow 0$ as $\delta \downarrow 0$ such that 
\[ \p[G_{j+1,\delta} \giv \CF_j ] \one_{G_{j,\delta}}  \leq p(\delta) \quad\text{for}\quad 1 \leq j \leq n_\beta.\]
Consequently, it follows that $\p[\cap_{j=1}^{n_\beta} G_{j,\delta}] \leq (p(\delta))^{n_\beta}$.  Choose $\delta > 0$ sufficiently small so that $p(\delta)^{n_\beta} \leq \epsilon^{(1-\beta) \alpha}$.  For each $1 \leq j \leq n_\beta$, let $E_{j,\delta}$ be the event that $\eta_1([0,\tau_{1,j}]) \cap \eta_2([0,\tau_{2,j}]) = \emptyset$, $\dist(\eta_1([\wh{\tau}_{1,j},\tau_{1,j}]),\eta_2([\wh{\tau}_{2,j},\tau_{2,j}])) \geq \delta \epsilon^\beta e^j$, and $\eta_i([0,\tau_{i,j}]) \cap (\delta \epsilon \D) \neq \emptyset$ for $i=1,2$.  We have that
\begin{align*}
   \p[F]
&\leq \p\big[ F \cap \big( \cup_j G_{j,\delta}^c \big) \big] + \p[ \cap_j G_{j,\delta}]\\
&\leq \sum_{j=1}^{n_\beta} \p[E_{j,\delta}] + \epsilon^{(1-\beta) \alpha}\\
&\lesssim \big(\log \tfrac{1}{\epsilon}\big) \epsilon^{(1-\beta) \alpha} + \epsilon^{(1-\beta) \alpha} \quad\text{(Lemma~\ref{lem::non_intersection_exponent}).}
\end{align*}
The upper bound follows because this holds for every $\beta \in (0,1)$.  The lower bound follows because we have that $\p[F] \geq \p[E_\delta] \asymp \epsilon^\alpha$ by Lemma~\ref{lem::non_intersection_exponent}.
\end{proof}

\subsection{Proof of the upper bound}
\label{subsec::dimension_upper_bound}

\begin{figure}[ht!]
\begin{center}
\includegraphics[scale=0.85]{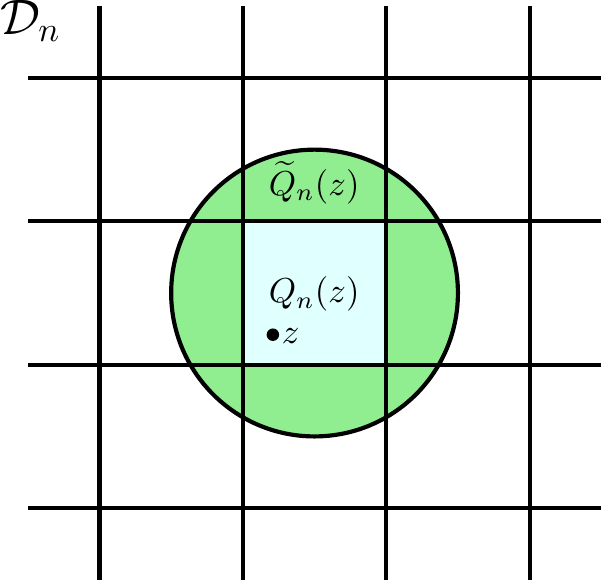}
\end{center}
\caption{
\label{fig::grids}
Shown in the illustration are $Q_n(z)$ and $\wt{Q}_n(z)$ for a given point $z \in \h$.
}
\end{figure}

We are now going to combine the estimates of Section~\ref{subsec::derivative_estimate} and Section~\ref{subsec::martingale_conditioning} to complete the proof of the upper bound.  Throughout, we suppose that $h$ is a GFF on $\h$ with the boundary data illustrated in Figure~\ref{fig::counterflowline_hit_conformal_maps} and let $\eta'$ be the counterflow line of $h$ starting from $0$.  For each $n \in \N$, we let $\CD_n$ be the set of squares with side length $2^{-n}$ and with corners in $2^{-n} \Z^2$ which are contained in $\ol{\h}$.  For each $Q \in \CD_n$, let $z(Q)$ be the center of~$Q$ and let $\wt{Q}_n(Q) = B(z(Q),2^{1-n})$.  Note that $Q \subseteq \wt{Q}_n(Q)$.  For each $z \in \h$, let $Q_n(z)$ be the element of $\CD_n$ which contains $z$ and let $\wt{Q}_n(z) = \wt{Q}_n(Q_n(z))$.  See Figure~\ref{fig::grids} for an illustration of these definitions.

For each $n \in \N$, we let $\zeta_{z,n} = \inf\{t \geq 0 : \eta'(t) \in \partial \wt{Q}_n(z)\}$ and
\begin{equation}
\label{eqn::rn_definition}
r_{z,n} = \frac{2^{6-n} |f_{\zeta_{z,n}}'(z)|}{ \sin(\Theta_{\zeta_{z,n}}^z)} \quad\text{on}\quad \{\zeta_{z,n} < \infty\}.
\end{equation}
We also let $\Theta_t^z$ be the process of~\eqref{eqn::mg_processes} with respect to $z$.

\begin{lemma}
\label{lem::image_contained}
Fix $z \in \h$ and $n \in \N$.  On $\{\zeta_{z,n} < \infty\}$, the following hold:
\begin{enumerate}[(i)]
\item\label{it::contained1} $f_{\zeta_{z,n}}(\wt{Q}_{n+3}(z)) \subseteq  r_{z,n} \D \cap \h$.
\item\label{it::contained2} There exists constants $c_1,c_2 > 0$ such that
\[ c_1 r_{z,n} \big( 2^{k/2} \sin(\Theta_{\zeta_{z,n}}^z) - c_2 \big) \D \cap \h \subseteq f_{\zeta_{z,n}}(\wt{Q}_{n-k}(z)) \quad\text{for each} \quad 1 \leq k \leq n.\]
\end{enumerate}
\end{lemma}
\begin{proof}
Throughout the proof, we shall assume that we are working on $\{\zeta_{z,n} < \infty\}$.  We first note that 
\[ \dist(z,\partial \wt{Q}_n(z)) \geq \dist(Q_n(z),\partial \wt{Q}_n(z)) = 2^{1-n} - 2^{-n-1/2} \geq 2^{-n}.\]
Consequently, we have that
\begin{equation}
\label{eqn::qz_dist}
  2^{-n} \leq   \dist(z,\partial \wt{Q}_n(z))  \leq 2^{1-n} \quad\text{for}\quad z \in \h \quad\text{and}\quad n \in \N.
\end{equation}
Hence applying Lemma~\ref{lem::ball_size} with $r = \tfrac{1}{2}$, we have that
\begin{equation}
\label{eqn::fzeta_diameter_bound}
\begin{split}
|f_{\zeta_{z,n}}(z) - f_{\zeta_{z,n}}(w)| 
&\leq 4 \times \frac{2^{-n-1}}{1-(1/2)^2} \times  \frac{\im(f_{\zeta_{z,n}}(z))}{2^{-n}}\\
&\leq 4 |f_{\zeta_{z,n}}(z)| \quad\text{for}\quad w \in \wt{Q}_{n+3}(z).
\end{split}
\end{equation}
This leaves us to bound $|f_{\zeta_{z,n}}(z)|$.  Applying Lemma~\ref{lem::distortion} and~\eqref{eqn::qz_dist}, we have that
\begin{equation}
\label{eqn::fzeta_im_bound}
\frac{\im(f_{\zeta_{z,n}}(z)) }{8} \leq \frac{|f_{\zeta_{z,n}}'(z)|}{2^n} \leq 4 \im(f_{\zeta_{z,n}}(z)).
\end{equation}
Applying the lower bound of~\eqref{eqn::fzeta_im_bound} in the inequality, we thus have that
\begin{align}
\label{eqn::fzeta_im_bound2}
    |f_{\zeta_{z,n}}(z)|
&= \frac{\im(f_{\zeta_{z,n}}(z))}{\sin(\Theta_{\zeta_{z,n}}^z)}
\leq \frac{2^{3-n} |f_{\zeta_{z,n}}'(z)|}{\sin (\Theta_{\zeta_{z,n}}^z)}.
\end{align}
Since 
\[ |f_{\zeta_{z,n}}(w)| \leq |f_{\zeta_{z,n}}(z) - f_{\zeta_{z,n}}(w)| + |f_{\zeta_{z,n}}(z)|,\]
combining~\eqref{eqn::fzeta_diameter_bound} with~\eqref{eqn::fzeta_im_bound2} gives~\eqref{it::contained1}.

To prove~\eqref{it::contained2}, we first note by the \hyperref[thm::beurling]{Beurling estimate} that there exists a constant $c_3 > 0$ such that the probability that a Brownian motion starting from $z$ hits $\partial \wt{Q}_{n-k}(z)$ before hitting $\eta'([0,\zeta_{z,n}])$ is at most $c_3 2^{-k/2}$.  Consequently, by the conformal invariance of Brownian motion, the probability that a Brownian motion starting from $f_{\zeta_{z,n}}(z)$ hits $f_{\zeta_{z,n}}(\partial \wt{Q}_{n-k}(z))$ before hitting $f_{\zeta_{z,n}}(\eta'([0,\zeta_{z,n}]))$ is also at most $c_3 2^{-k/2}$.  By standard estimates for Brownian motion, it follows that there exists a constant $c_4 > 0$ such that 
\begin{equation}
\label{eqn::fznk_lbd}
\dist(f_{\zeta_{z,n}}(z),f_{\zeta_{z,n}}(\partial \wt{Q}_{n-k}(z))) \geq c_4 \im(f_{\zeta_{z,n}}(z)) 2^{k/2}.
\end{equation}
Consequently, we have that
\begin{align}
           \dist(0, f_{\zeta_{z,n}}(\partial\wt{Q}_{n-k}(z)))
\geq& \dist(f_{\zeta_{z,n}}(z), f_{\zeta_{z,n}}(\partial\wt{Q}_{n-k}(z))) - |f_{\zeta_{z,n}}(z)| \notag\\
\geq& c_4 \im(f_{\zeta_{z,n}}(z)) 2^{k/2} - |f_{\zeta_{z,n}}(z)| \quad\text{(by~\eqref{eqn::fznk_lbd})} \notag\\
\geq &c_4 \im(f_{\zeta_{z,n}}(z)) 2^{k/2} - \frac{2^{3-n} |f_{\zeta_{z,n}}'(z)|}{\sin(\Theta_{\zeta_{z,n}}^z)} \quad\text{(by~\eqref{eqn::fzeta_im_bound2})}. \label{eqn::dist_q_0}
\end{align}
In analogy with~\eqref{eqn::fzeta_im_bound2}, the upper bound of~\eqref{eqn::fzeta_im_bound} implies that
\begin{equation}
\label{eqn::fzeta_im_bound3}
    |f_{\zeta_{z,n}}(z)|
= \frac{\im(f_{\zeta_{z,n}}(z))}{\sin(\Theta_{\zeta_{z,n}}^z)}
\geq \frac{2^{-2-n} |f_{\zeta_{z,n}}'(z)|}{\sin (\Theta_{\zeta_{z,n}}^z)}.
\end{equation}
By the definition of $r_{z,n}$ and~\eqref{eqn::fzeta_im_bound3}, it follows that there exists constants $c_1,c_2 > 0$ such that the expression in~\eqref{eqn::dist_q_0} is bounded from below by
\[ c_1 r_{z,n} \big( 2^{k/2} \sin(\Theta_{\zeta_{z,n}}^z) - c_2 \big).\]
This proves the desired result.
\end{proof}

\begin{figure}[ht!]
\begin{center}
\includegraphics[scale=0.85]{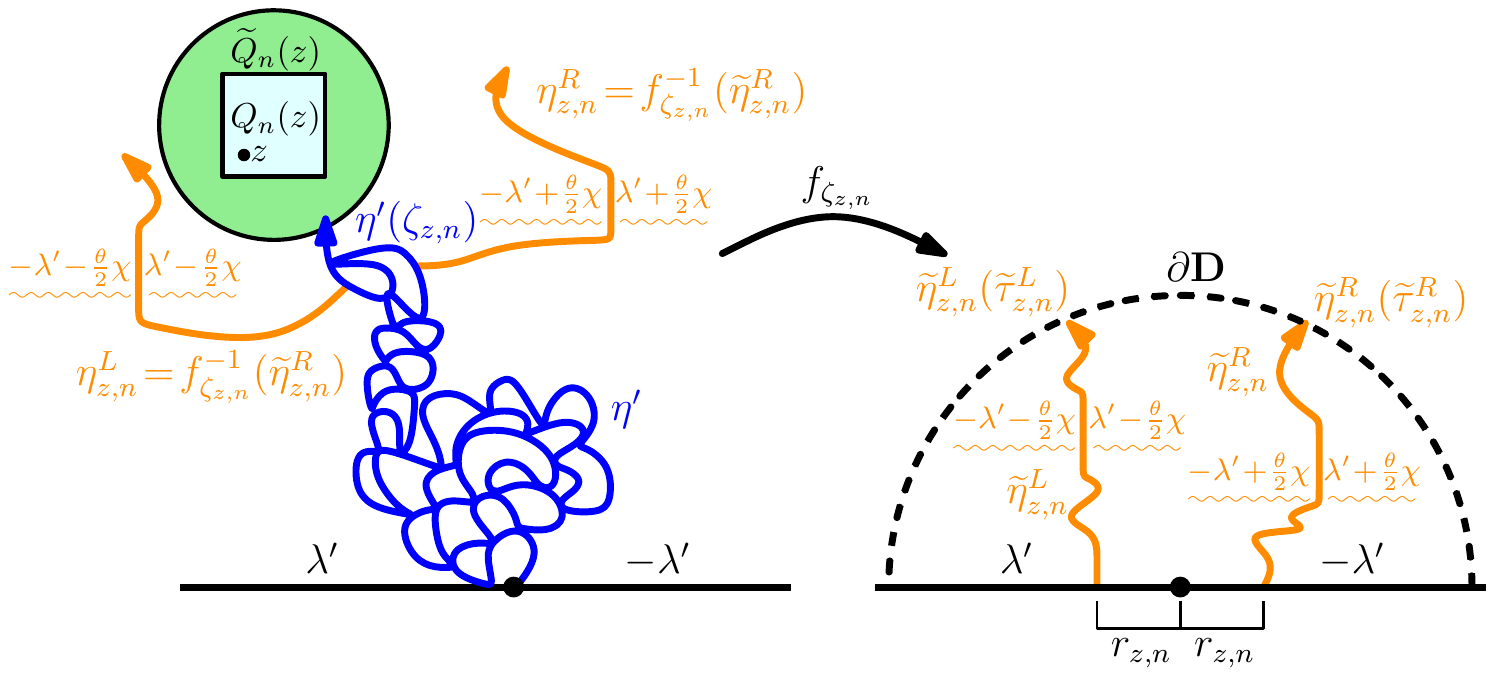}
\end{center}
\caption{
\label{fig::paths_sides}
An illustration of the construction of the paths $\eta_{z,n}^L$ and $\eta_{z,n}^R$.}
\end{figure}

On $\{\zeta_{z,n} < \infty\}$, we let $\wt{h}_n = h \circ f_{\zeta_{z,n}}^{-1} - \chi \arg (f_{\zeta_{z,n}}^{-1})'$ be the GFF which arises after conformally mapping away $\eta'([0,\zeta_{z,n}])$.  We then let $\wt{\eta}_{z,n}^L$ (resp.\ $\wt{\eta}_{z,n}^R$) be the flow line of $\wt{h}_n$ starting from $-r_{z,n}$ (resp.\ $r_{z,n}$) with angle $\tfrac{\theta}{2}$ (resp.\ $-\tfrac{\theta}{2}$).  Let $\wt{\tau}_{z,n}^L$ (resp.\ $\wt{\tau}_{z,n}^R$) be the first time that $\wt{\eta}_{z,n}^L$ (resp.\ $\wt{\eta}_{z,n}^R$) hits $\partial \D$.  For $q \in \{L,R\}$, let $\eta_{z,n}^q = f_{\zeta_{z,n}}^{-1}(\wt{\eta}_{z,n}^q)$.  See Figure~\ref{fig::paths_sides} for an illustration of the construction.  We are now going to show that the set of points $z$ at which the paths $\eta'$, $\eta_{z,n}^L$, $\eta_{z,n}^R$ behave in a consistently pathological manner as $\eta'$ approaches $z$ is almost surely empty.  In particular, we will prove in Lemma~\ref{lem::bad_hit_angle} that the set $\CS^{1,\delta}$ of points $z$ that $\eta'$ approaches at an angle which is consistently outside of $(\delta,\pi-\delta)$ is almost surely empty for a sufficiently small choice of $\delta > 0$.  Then we will show in Lemma~\ref{lem::bad_hit_small_ball} that the set $\CS^{2,k}$ of points $z$ that $\eta'$ approaches and for consistently large values of $n$ either $\wt{\eta}_{z,n-k}^L$ or $\wt{\eta}_{z,n-k}^R$ hits $\partial (r_{z,n} \D)$ is almost surely empty for a sufficiently large choice of $k$.  These results, in turn, will be used in the proof of Proposition~\ref{prop::dimension_upper_bound} to generate a cover of $\lightcone(\theta)$ in the manner described in Figures~\ref{fig::flowline_hit_point}--\ref{fig::counterflowline_hit_conformal_maps}.

\begin{lemma}
\label{lem::bad_hit_angle}
For each $z \in \h$, $n \in \N$, and $\delta > 0$, we let $E_{z,n}^{1,\delta} = \{ \zeta_{z,n} < \infty,\ \Theta_{\zeta_{z,n}}^z \notin (\delta,\pi-\delta)\}$.  Let $\CS_n^{1,\delta}$ be the set of points $z \in \h$ such that $E_{z,n}^{1,\delta}$ occurs and let $\CS^{1,\delta} = \cup_{n=1}^\infty \cap_{m=n}^\infty \CS_m^{1,\delta}$.  There exists $\delta_0 > 0$ such that for every $\delta \in (0,\delta_0)$ we have that $\CS^{1,\delta} = \emptyset$ almost surely.
\end{lemma}
\noindent We remark that a version of Lemma~\ref{lem::bad_hit_angle} is proved in \cite{MW_INTERSECTIONS} for $\kappa \in (0,4)$ using a different argument.
\begin{proof}[Proof of Lemma~\ref{lem::bad_hit_angle}]
By \cite[Theorem~3]{SCHRAMM_WILSON}, we can view $\eta'$ as a radial $\SLE_{\kappa'}(\kappa'-6)$ process targeted at $z$.  After reparameterizing the path by $\log$ conformal radius, $\Theta^z$ solves the SDE
\[ d \Theta_t^z = (\kappa'-4) \cot\left(\Theta_t^z\right) dt + \sqrt{\kappa'} dB_t\]
where $B$ is a standard Brownian motion (see \cite[Equation (4.1)]{SHE_CLE}).
When $\kappa' \in (4,8)$, $\kappa'-6 < \tfrac{\kappa'}{2}-2$ which means that $\Theta^z$ almost surely hits either $0$ or $\pi$ in finite time \cite[Lemma~1.26]{LAW05}.  In particular, if $\Theta_t^z = \theta \notin (\delta,\pi-\delta)$ for some fixed $t$ then the probability that $\Theta_s^z \in \{0,\pi\}$ for some $s \in [t,t+1]$ tends to $1$ as $\delta \to 0$ uniformly in $\theta \notin (\delta,\pi-\delta)$.  It follows that there exists a function $p_1 \colon [0,\tfrac{\pi}{2}] \to [0,1]$ with $p_1(\delta) \to 0$ as $\delta \to 0$ such that
\begin{equation}
\label{eqn::given_hit_probability}
 \p[ \zeta_{z,n+8} < \infty \giv \Theta_{\zeta_{z,n}}^z =\theta,\ \zeta_{z,n} < \infty] \leq p_1(\delta) \quad\text{for all} \quad  \theta \notin (\delta,\pi-\delta)
 \end{equation}
(by standard distortion estimates for conformal maps, it takes at least $1$ unit of $\log$ conformal radius time for the path to travel from $\partial \wt{Q}_n(z)$ to $\partial \wt{Q}_{n+8}(z)$).  Iterating~\eqref{eqn::given_hit_probability} implies that with $p(\delta) = p_1^{1/8}(\delta)$ we have that
\begin{equation}
\label{eqn::consistently_bad_hit}
\p[ \cap_{m=r}^n E_{z,m}^{1,\delta}] \leq p(\delta)^{n-r} \quad\text{for all}\quad n \geq r \geq -\log_2 \im(z) + 2
\end{equation}
($r \geq -\log_2 \im(z) + 2$ implies $\wt{Q}_r(z) \subseteq \h$.)

Note that for $Q \in \CD_j$, the function $Q \to \R$ given by $w \mapsto \Theta_{\zeta_{w,j}}^w$ is positive and harmonic.  Consequently, the Harnack inequality \cite[Proposition~2.22]{LAW05} implies that there exists a constant $K \geq 1$ such that for all $z,w \in Q$ we have that $\Theta_{\zeta_{w,j}}^w \leq K \Theta_{\zeta_{z,j}}^z$.  Thus letting $E_{Q,m}^{1,\delta} = \cup_{z \in Q} E_{z,m}^{1,\delta}$ for $m \leq j$, it follows from~\eqref{eqn::consistently_bad_hit} that
\begin{equation}
\label{eqn::consistently_bad_hit_square}
\p[ \cap_{m=r}^n E_{Q,m}^{1,\delta}] \leq p(K\delta)^{n-r} \quad\text{for all}\quad j \geq n \geq r.
\end{equation}
Fix $\varsigma \in (0,1)$ and let $r = -\log_2 \varsigma+2$, $U_\varsigma = (\varsigma^{-1} \D) \cap (\h + \varsigma i)$, and $\CV_j^{\varsigma,\delta}$ for $j \geq r$ consist of those $Q \in \CD_j$ with $Q  \subseteq U_\varsigma$ such that $\cap_{m=r}^j E_{Q,m}^\delta$ occurs.  It is easy to see that there exists a constant $C > 0$ such that
\begin{equation}
\label{eqn::consistently_bad_summation}
 \sum_{m=r}^\infty \E[|\CV_m^{\varsigma,\delta}|] \leq \frac{C}{\varsigma^2} \sum_{m=r}^\infty 2^{2m} p(K\delta)^{m-r}.
\end{equation}
Consequently, choosing $\delta > 0$ sufficiently small so that $4p(K \delta) < 1$, we see that the summations in~\eqref{eqn::consistently_bad_summation} are finite.  This implies that the set of squares in $\CV_m^{\varsigma,\delta}$ is non-empty for finitely many $m$ almost surely, from which the claimed result follows for $\kappa' \in (4,8)$.

For $\kappa' \geq 8$, we have that $\kappa'-6 \geq \tfrac{\kappa'}{2}-2$ which means that $\Theta^z$ almost surely does not hit $0$ or $\pi$ \cite[Lemma~1.26]{LAW05}.  In this case, it is easy to see from the form of the SDE that there exists a function $q_1 \colon [0,\tfrac{\pi}{2}] \to [0,1]$ such that
\[ \p[ E_{z,n+8}^{1,\delta} \giv \Theta_{\zeta_{z,n}}^z =\theta,\ \zeta_{z,n} < \infty] \leq q_1(\delta) \quad\text{for all}\quad \theta \notin (\delta,\pi-\delta).\]
Therefore the same argument we used to complete the proof for $\kappa' \in (4,8)$ also applies here, which proves the claimed result for $\kappa' \geq 8$.
\end{proof}

\begin{figure}[ht!]
\begin{center}
\includegraphics[scale=0.85]{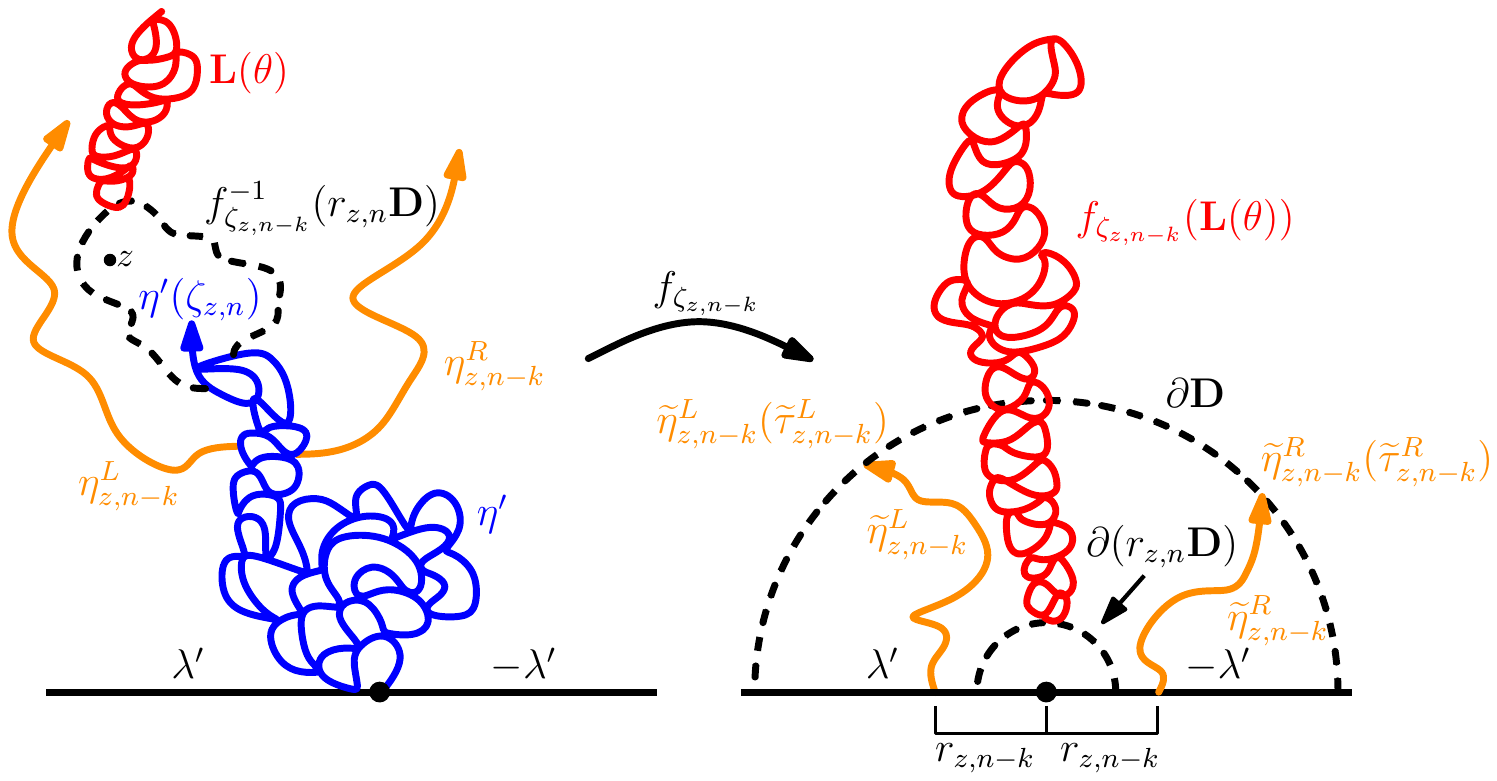}
\end{center}
\caption{\label{fig::avoid_small_ball} In Lemma~\ref{lem::bad_hit_small_ball}, we show that there exists $k_0$ such that for $k \geq k_0$ the set of points $z \in \h$ such that $\{\zeta_{z,n} < \infty\}$ and either $\wt{\eta}_{z,n-k}^L$ or $\wt{\eta}_{z,n-k}^R$ hits $\partial (r_{z,n} \D)$ before hitting $\partial \D$ for arbitrarily large values of $n$ is almost surely empty.  On the event that $\lightcone(\theta)$ hits $f_{\zeta_{z,n-k}}^{-1}(\partial (r_{z,n} \D))$ (equivalently, $\wt{\lightcone}(\theta) = f_{\zeta_{z,n-k}}(\lightcone(\theta))$ hits $\partial (r_{z,n} \D)$) and both $\wt{\eta}_{z,n-k}^L$, $\wt{\eta}_{z,n-k}^R$ do not hit $\partial (r_{z,n} \D)$ before hitting $\partial \D$, it follows that $\wt{\eta}_{z,n-k}^L$ and $\wt{\eta}_{z,n-k}^R$ do not intersect each other before hitting $\partial \D$.  Indeed, $\wt{\lightcone}(\theta)$ cannot enter into the region surrounded by $\wt{\eta}_{z,n-k}^L$ and $\wt{\eta}_{z,n-k}^R$, so if $\wt{\eta}_{z,n-k}^L$ and $\wt{\eta}_{z,n-k}^R$ did intersect then $\wt{\lightcone}(\theta)$ could not intersect $\partial (r_{z,n} \D)$.
}
\end{figure}

\begin{lemma}
\label{lem::bad_hit_small_ball}
For each $z \in \h$ and $k,n \in \N$ with $1 \leq k \leq n$, we let 
\[ E_{z,n}^{2,k} = \left\{ \zeta_{z,n} < \infty,\ \big( \cup_{q \in \{L,R\}}\wt{\eta}_q^{n-k}([0,\wt{\tau}_q^{n-k}]) \big) \cap ( r_{z,n} \D ) \neq \emptyset \right\}.\]
Let $\CS_n^{2,k}$ be the set of points $z \in \h$ such that $E_{z,n}^{2,k}$ occurs and let $\CS^{2,k} = \cup_{n=1}^\infty \cap_{m=n}^\infty \CS_m^{2,k}$.  There exists $k_0$ such that for every $k \geq k_0$ we have that $\CS^{2,k} = \emptyset$ almost surely.
\end{lemma}
See Figure~\ref{fig::avoid_small_ball} for an illustration of the setup of Lemma~\ref{lem::bad_hit_small_ball}.
\begin{proof}[Proof of Lemma~\ref{lem::bad_hit_small_ball}]
Fix $\delta > 0$ sufficiently small so that the statement of Lemma~\ref{lem::bad_hit_angle} holds.  Let $\wt{\CS}_n^{2,k}$ be the set of points such that $E_{z,n}^{2,k} \setminus E_{z,n}^{1,\delta}$ occurs and let $\wt{\CS}^{2,k} = \cup_{n=1}^\infty \cap_{m=n}^\infty \wt{\CS}_m^{2,k}$.  We are going to prove the lemma by showing that there exists $k_0 \in \N$ such that $k \geq k_0$ implies that $\wt{\CS}^{2,k} = \emptyset$ almost surely.  Let $\wt{\theta}_L \leq \tfrac{\theta}{2}$ be an angle such that a flow line of $\wt{h}_j$ starting from $\R_-$ can hit $\R_+$ and let $\wt{\theta}_R = -\wt{\theta}_L$ so that a flow line of $\wt{h}_j$ starting from $\R_+$ with angle $\wt{\theta}_R$ can hit $\R_-$.  For each $n \in \N$, we let $\wt{\gamma}_n^L$ (resp.\ $\wt{\gamma}_n^R$) be the flow line of $\wt{h}_n$ starting from $c_1 2^{k/4} r_{z,n}$ (resp.\ $-c_1 2^{k/4} r_{z,n}$) with angle $\wt{\theta}_L$ (resp.\ $\wt{\theta}_R)$ where $c_1$ is the constant from Lemma~\ref{lem::image_contained}.  Let $\wt{\tau}_{z,n}^L$ (resp.\ $\wt{\tau}_{z,n}^R$) be the first time that $\wt{\gamma}_n^L$ (resp.\ $\wt{\gamma}_n^R$) hits either $\partial B(0, r_{z,n})$ or $\partial B(0, c_1 2^{k/2} r_{z,n})$ and let $\wt{F}_{z,n}^k$ be the event that both $\wt{\gamma}_n^L([0,\wt{\tau}_{z,n}^L])$ and $\wt{\gamma}_n^R([0,\wt{\tau}_{z,n}^R])$ separate $B(0,r_{z,n})$ from $\partial B(0, c_1 2^{k/2} r_{z,n})$.  Note that $(E_{z,n}^{2,k})^c \subseteq \wt{F}_{z,n}^k$ because $\wt{\eta}_q^n$ cannot cross $\wt{\gamma}_n^q$ for $q \in \{L,R\}$.  

Let $\CF_n$ be the $\sigma$-algebra generated by $\eta'|_{[0,\zeta_{z,n}]}$ as well as the paths $\wt{\eta}_j^q|_{[0,\wt{\tau}_j^q]}$ for $1 \leq j \leq n$.  We next claim that there exists a function $p \colon \N \to [0,1]$ with $p(k) \downarrow 0$ as $k \to \infty$ such that
\begin{equation}
\label{eqn::extra_flowline_separate_bound}
\p[ \wt{F}_{z,n+k}^k \giv \CF_n] \one_{\{\zeta_{z,n} < \infty\}} \leq p(k).
\end{equation}
We are first going to explain why there exists a function $p_1 \colon \N \to [0,1]$ with $p_1(k) \downarrow 0$ as $k \to \infty$ such that
\begin{equation}
\label{eqn::extra_flowline_separate_bound_intermediate}
 \p[ \wt{F}_{z,n}^k] \leq p_1(k) \quad\text{for all}\quad n \in \N.
\end{equation}
We will then explain using the Radon-Nikodym derivative estimate Lemma~\ref{lem::rn} why~\eqref{eqn::extra_flowline_separate_bound} follows once we establish~\eqref{eqn::extra_flowline_separate_bound_intermediate}.  First of all, we note that the probability that $\wt{\gamma}_n^L$ hits $\R_+$ before hitting $\partial B(0, c_1 2^{k/2} r_{z,n})$ tends to $1$ as $k \to \infty$, the probability that it hits $B(0, r_{z,n})$ tends to $0$ as $k \to \infty$, and the analogous statements are likewise true with $\wt{\gamma}_n^R$ in place of $\wt{\gamma}_n^L$.  Indeed, this follow since the law of $\wt{\gamma}_n^L$ rescaled by $(c_1 2^{k/4} r_{z,n})^{-1}$ stopped upon hitting  $\partial \D$ is that of an $\SLE_\kappa(\rho_{1,L};\rho_{1,R},\rho_{2,R})$ process starting from $1$ with $\rho_{1,L}, \rho_{1,R} > -2$, $\rho_{1,R} + \rho_{2,R} \in (\tfrac{\kappa}{2}-4,\tfrac{\kappa}{2}-2)$, and with the force points located at $1^-$ and $0,1^+$, respectively.  This proves~\eqref{eqn::extra_flowline_separate_bound_intermediate}.  To extract~\eqref{eqn::extra_flowline_separate_bound} from~\eqref{eqn::extra_flowline_separate_bound_intermediate}, we note that part~\eqref{it::contained2} of Lemma~\ref{lem::image_contained} implies that the paths involved in the definition of $\wt{F}_{z,n+k}^k$ are disjoint and at a positive distance from those involved in the definition of $\wt{F}_{z,j}^k$ for all $j \leq n$.  Consequently, the claimed bound follows from Lemma~\ref{lem::rn}.  That there exists $k_0 \in \N$ such that $\CS^{2,k} = \emptyset$ almost surely for $k \geq k_0$ then follows from the same argument used to establish the corresponding result for $\CS^{1,\delta}$ in Lemma~\ref{lem::bad_hit_angle}.
\end{proof}

\begin{proof}[Proof of Proposition~\ref{prop::dimension_upper_bound}]
We begin by partitioning $\lightcone(\theta)$ as follows.  For each $\delta >0$, let $\lightcone^\delta(\theta)$ consist of those $z \in \lightcone(\theta)$ such that for every $n \in \N$ there exists $m \geq n$ such that the event $E_\delta$ of Lemma~\ref{lem::non_intersection_exponent} occurs for the pair of paths $(\wt{\eta}_{z,m}^L, \wt{\eta}_{z,m}^R)$.  It follows from Lemma~\ref{lem::bad_hit_small_ball} and the argument described in Figure~\ref{fig::avoid_small_ball} that $\lightcone(\theta) = \cup_{\delta > 0} \lightcone^\delta(\theta)$.  Moreover, note that $0 < \delta_1 < \delta_2$ implies that $\lightcone^{\delta_2}(\theta) \subseteq \lightcone^{\delta_1}(\theta)$.  Consequently, it suffices to show that there exists $\ol{\delta} > 0$ such that the desired upper bound for $\dimH(\lightcone^\delta(\theta))$ holds for each $\delta \in (0,\ol{\delta})$.  We are going to set the value of $\ol{\delta} > 0$ in the proof.  We begin by assuming that $\ol{\delta} \in (0,\delta_0)$ where $\delta_0 > 0$ is the constant from Lemma~\ref{lem::bad_hit_angle}.

Fix $\delta \in (0,\ol{\delta})$, $\varsigma \in (0,1)$, and let $U_\varsigma = (\varsigma^{-1} \D) \cap (\h + \varsigma i)$.  For each $n \in \N$, we are now going to construct a cover of $\lightcone^\delta(\theta) \cap U_\varsigma$ consisting of squares in $\cup_{m \geq n} \CD_m$.  Let $\CU_m^{\varsigma,\delta}$ be the set of squares in $Q \in \CD_m$ which are contained in $U_\varsigma$ such that the following hold:
\begin{enumerate}[(i)]
\item $\eta'$ hits $\wt{Q}=\wt{Q}_m(Q)$, say for the first time at time $\tau_{\wt{Q}}$,
\item $\Theta_{\tau_{\wt{Q}}}^{z(Q)} \in (\delta,\pi-\delta)$,
\item The event $E_\delta$ of Lemma~\ref{lem::non_intersection_exponent} defined in terms of the paths $\wt{\eta}_{z,m}^L$ and $\wt{\eta}_{z,m}^R$ occurs.
\end{enumerate}
For each $n \in \N$, we let $\CV_n^{\varsigma,\delta} = \cup_{m \geq n} \CU_m^{\varsigma,\delta}$.  To complete the proof, we need to show that~$\CV_n^{\varsigma,\delta}$ is a cover of $\lightcone^\delta(\theta) \cap U_\varsigma$ and then get a bound on the expected number of squares in $\CV_n^{\varsigma,\delta}$.

Fix $z \in \lightcone^\delta(\theta) \cap U_\varsigma$.  Since $\lightcone^\delta(\theta)$ is contained in the range of $\eta'$, it follows that $\zeta_{z,n} < \infty$ for all $n \in \N$.  For each $j \in \N$, let $Q_j \in \CD_j$ be the square which contains $z$ and let $\wt{Q}_j = \wt{Q}(Q_j)$.  It follows from Lemma~\ref{lem::bad_hit_angle} and Lemma~\ref{lem::bad_hit_small_ball}, possibly by decreasing the value of $\ol{\delta} > 0$, that there exists a sequence $(j_\ell)$ in $\N$ tending to $\infty$ such that $\Theta_{\tau_{\wt{Q}_{j_\ell}}}^{z} \in (\delta, \pi-\delta)$, $\wt{\eta}_{j_\ell}^L$ and $\wt{\eta}_{j_\ell}^R$ do not hit $B(0,\delta r_{j_\ell})$ for all $\ell \in \N$.  Therefore $Q_{j_\ell} \in \CV_n^{\varsigma,\delta}$ for all $\ell \in \N$ so that $j_\ell \geq n$, hence $\CV_n^{\varsigma,\delta}$ is a cover of $\lightcone^\delta(\theta) \cap U_\varsigma$, as desired.

We are now going to estimate $\p[Q \in \CU_n^{\varsigma,\delta}]$ for a given square $Q \in \CD_n$ which is contained in $U_\varsigma$.  Take $\theta_1=\tfrac{\theta}{2}$, $\theta_2 = -\tfrac{\theta}{2}$, $a = \lambda'$, and $b = -\lambda'$.  The exponent $\alpha$ from~\eqref{eqn::non_intersection_exponent} of Lemma~\ref{lem::non_intersection_exponent_general} corresponding to these parameters is given by
\begin{align*}
\alpha
&= \frac{1}{\kappa} \left(1-\frac{\kappa}{4}\right)\left(1-\ol{\theta}\right) \left(\kappa(1+\ol{\theta}) - 4 \ol{\theta} \right) \quad\text{where}\quad \ol{\theta} = \frac{\theta}{\pi}.
\end{align*}
Therefore
\begin{equation}
\label{eqn::prob_decomp_sle}
 \p[Q \in \CU_n^{\varsigma,\delta}] \asymp \E\big[ (|f_{\zeta_{z,n}}'(z)| r_{z,n})^\alpha \one_{\{\zeta_{z,n} < \infty\}} \big] = r_{z,n}^\alpha \E\big[ |f_{\zeta_{z,n}}'(z)|^\alpha \one_{\{\zeta_{z,n} < \infty\}} \big]
\end{equation}
where the constants in $\asymp$ depend only on $\delta$, $\kappa$, and $\theta$.
Recall that $\kappa'=\tfrac{16}{\kappa}$.  Set
\begin{equation}
\label{eqn::r_choice}
 r = -\left(\frac{4+4\ol{\theta}}{\kappa'}\right) + \ol{\theta} < \frac{1}{2} - \frac{4}{\kappa'}
 \end{equation}
so that 
\begin{equation}
\label{eqn::nu_choice}
 \nu(r)+r  = \alpha \quad\text{and}\quad  -\xi(r) - r = \frac{2}{\kappa'} - \frac{8}{\kappa'}\left((\kappa'-4)\ol{\theta} \right)^2.
\end{equation}
With this choice of $r$, we can apply Proposition~\ref{prop::derivative_estimate} and this leads to an exponent for $\p[Q \in \CU_n^{\varsigma,\delta}]$ given by
\begin{align*}
 \beta = \beta(\kappa,\theta) =&(\nu(r) +r) -(\xi(r)+r) = \nu(r)-\xi(r)\\
  =& \frac{\left(4 + (4-\kappa')\ol{\theta} \right) \left(2(\kappa'-2) + (4-\kappa')\ol{\theta}\right)}{8\kappa'}
\end{align*}
That is, 
\[ \p[ Q \in \CU_n^{\varsigma,\delta}] \asymp 2^{-n \beta}\]
where the constants in $\asymp$ depend only on $\theta$, $\delta$, and $\kappa$.  By making the substitution $\kappa'=\tfrac{16}{\kappa}$, we note that $2-\beta = d(\kappa,\theta)$.   Fix $\xi > 0$.  By performing a union bound over $\{Q \in \CU_m^{\varsigma,\delta}\}$ for $m \geq n$, we consequently have that
\begin{align*}
    \E[\CH_{d(\kappa,\theta)+\xi}(\lightcone^\delta(\theta) \cap U_\varsigma)]
&\lesssim \sum_{m=n}^\infty 2^{2m}  \times 2^{-m (d(\kappa,\theta)+\xi)} \times 2^{-m \beta}\\
&\asymp \sum_{m=n}^\infty 2^{-\xi m} < \infty
\end{align*}
where the constants in $\lesssim$ and $\asymp$ depend only on $\varsigma$, $\delta$, $\kappa$, and $\theta$.  Taking a limit as $n \to \infty$ implies that $\CH_{d(\kappa,\theta)+\xi}^\delta(\lightcone^\delta(\theta) \cap U_\varsigma) = 0$ almost surely.  Since $\varsigma \in (0,1)$ was arbitrary, we therefore have that $\dimH(\lightcone^\delta(\theta)) \leq d(\kappa,\theta)+\xi$ almost surely.  The result follows since $\delta,\xi > 0$ were arbitrary.
\end{proof}

Now that we have proved Proposition~\ref{prop::dimension_upper_bound}, hence the upper bounds of Theorem~\ref{thm::lightcone_dimension} and Theorem~\ref{thm::sle_kappa_rho_dimension}, we turn to complete the proof of Theorem~\ref{thm::fan_dimension}.

\begin{proof}[Proof of Theorem~\ref{thm::fan_dimension}]
Fix $\theta \in [0,\pi]$.  Suppose that $h$ is a GFF on $\h$ with piecewise constant boundary data which changes values at most a finite number of times and let $\fan(\theta)$ be the fan of $h$ with opening angle $\theta$ starting from $0$.  For each $\theta_1 \leq \theta_2$, we let $\lightcone(\theta_1,\theta_2)$ be the closure of the set of points accessible by angle-varying flow lines starting from $0$ with rational angles contained in $[\theta_1,\theta_2]$ and which change angles a finite number of times and only at positive rational times.  Using this notation, $\lightcone(\theta) = \lightcone(-\tfrac{\theta}{2},\tfrac{\theta}{2})$.  By Proposition~\ref{prop::dimension_upper_bound}, the dimension $\lightcone(\theta_1,\theta_2)$ is at most $d(\kappa,\theta_2-\theta_1)$.  Note that, as $\theta_2 - \theta_1$ decreases to $0$, $d(\kappa,\theta_2-\theta_1)$ decreases to $1+\tfrac{\kappa}{8}$, the dimension of ordinary $\SLE_\kappa$.  For each $\epsilon > 0$, we have that $\fan(\theta)$ is contained in the finite union $\cup_{j=0}^{\lceil \theta/\epsilon \rceil}\lightcone(-\tfrac{\theta}{2}+j \epsilon,-\tfrac{\theta}{2} + (j+1)\epsilon)$ of light cones.  By Proposition~\ref{prop::dimension_upper_bound}, the Hausdorff dimension of each of these light cones is almost surely at most $d(\kappa,\epsilon)$.  Therefore $\dimH(\fan(\theta)) \leq d(\kappa,\epsilon)$ almost surely.  Since this holds for each $\epsilon > 0$, $\dimH(\fan(\theta)) \leq d(\kappa,0) = 1+\tfrac{\kappa}{8}$ almost surely.  We also have that $\dimH(\fan(\theta)) \geq 1+\tfrac{\kappa}{8}$ almost surely since $\fan(\theta)$ contains the $0$ angle flow line of $h$ starting from $0$ which itself has dimension $1+\tfrac{\kappa}{8}$ almost surely by \cite{RS05,BEF_DIM}.
\end{proof}

\section{Lower bound}
\label{sec::lowerbound}

We are now going to finish the proof of Theorem~\ref{thm::sle_kappa_rho_dimension} by establishing the lower bound.  We will make use of a multi-scale refinement of the second moment method (see \cite{DPRZ01, HMP10, MSW_CLE_GASKET, MW_INTERSECTIONS,multifractal_spectrum,extreme_nesting} for similar applications of this technique).  In particular, we will introduce a special class of points --- so called ``perfect points'' --- which are contained in $\lightcone(\theta)$ whose correlation structure is easier to control than for general points in $\lightcone(\theta)$ and then get a lower bound for the dimension of this set of points.

\subsection{Definition of events}

\begin{figure}[ht!]
\begin{center}
\includegraphics[scale=0.85]{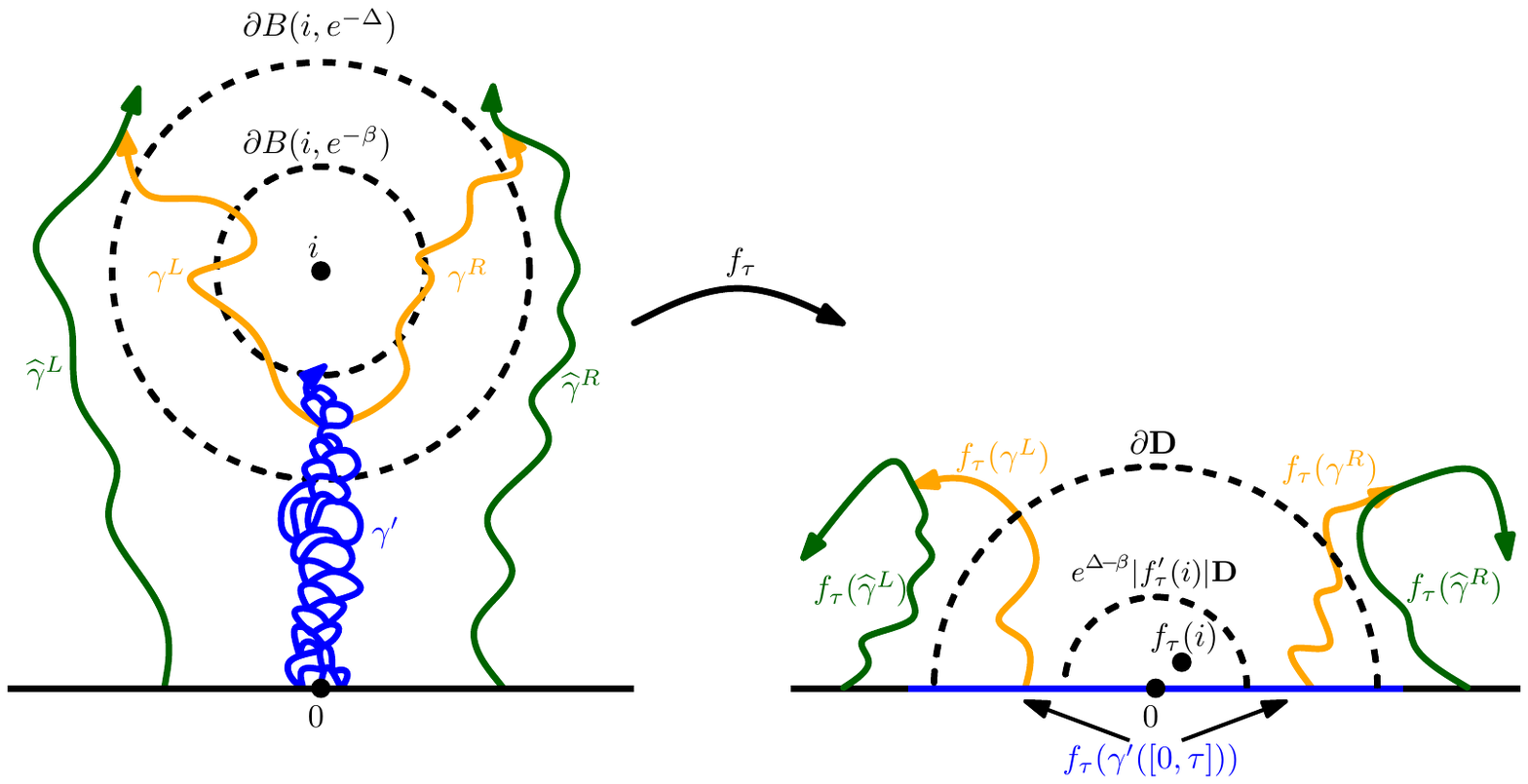}
\end{center}
\caption{\label{fig::perfect_definition} Illustration of the event $E_u^{\beta,\Delta}(\gamma',\gamma^L,\gamma^R,\wh{\gamma}^L,\wh{\gamma}^R)$ used to define a perfect point to establish the lower bound.  Here, $f_t = g_t - W_t$ where $(g_t)$ and $W$ are, respectively, the Loewner evolution and driving function of $\gamma'$ and $\tau$ is the first time that $\gamma'$ hits $\partial B(i,e^{-\beta})$.  In our particular application, $\gamma'$ will be a counterflow line hence self-intersecting and the other paths will be flow lines hence simple, as illustrated.}
\end{figure}

We will now work towards defining the perfect points.  See Figure~\ref{fig::perfect_definition} for an illustration of the event which is used to define a perfect point and which we will now describe.  Fix $u \in \partial \h \setminus \{0\}$ and $\beta > \Delta^2 > \Delta > 1$; we will eventually take a limit first as $\beta \to \infty$ and then as $\Delta \to \infty$.  Suppose that we have five non-crossing paths $\gamma'$, $\gamma^L$, $\wh{\gamma}^L$, $\gamma^R$, $\wh{\gamma}^R$ in $\ol{\h}$.  We assume that $\gamma'$ starts from $0$ and let $\tau$ be the first time that $\gamma'$ hits $\partial B(i,e^{-\beta})$.  We assume that $\gamma'$ admits a (chordal) Loewner evolution $(g_t)$ with continuous Loewner driving function $W$ and let $f_t = g_t - W_t$ be its centered Loewner evolution.  We also assume that $\gamma^L$ (resp.\ $\gamma^R$) starts on the left (resp.\ right) side of $\gamma'([0,\tau])$ and $\wh{\gamma}^L$ (resp.\ $\wh{\gamma}^R$) starts on $\R_-$ (resp.\ $\R_+$).  We then let $E_u^{\beta,\Delta}(\gamma',\gamma^L,\gamma^R,\wh{\gamma}^L,\wh{\gamma}^R)$ be the event that the following hold:

\begin{enumerate}[(i)]
\item $\gamma'$ hits $\partial B(i,e^{-\beta})$ and does so before hitting $\partial B(i,e^{\Delta})$.
\item\label{it::harmonic_measure_tip} The harmonic measure of the left (resp.\ right) side of $\gamma'([0,\tau])$ as seen from $i$ in $\h \setminus \gamma'([0,\tau])$ is at least $\tfrac{1}{2} - e^{-4\Delta}$.
\item For $q \in \{L,R\}$, let $\zeta^q$ (resp.\ $\wh{\zeta}^q$) be the first time that $f_\tau(\gamma^q)$ hits $\partial \D$ (resp.\ $\partial (\tfrac{1}{2} \D)$).  Then $\dist(f_\tau(\gamma^L([\wh{\zeta}^L,\zeta^L)), f_\tau(\gamma^R([\wh{\zeta}^R,\zeta^R))) \geq e^{-\Delta}$.
\item\label{it::paths_not_hit_small_ball} $\gamma^L$ and $\gamma^R$ intersect $\wh{\gamma}^L$ and $\wh{\gamma}^R$, respectively, before intersecting each other and also before leaving $B(i,e^{\Delta})$.  Moreover, $e^{\Delta-\beta} |f_\tau'(i)| \D$ is (completely) contained in the connected component of $\h \setminus f_\tau(\gamma^L \cup \gamma^R \cup \wh{\gamma}^L \cup \wh{\gamma}^R)$ which has $0$ on its boundary.
\end{enumerate}

Before we finish defining the perfect points of $\lightcone(\theta)$, we first record the following lemma.

\begin{lemma}
\label{lem::beta_ball_contained}
Suppose that we have the setup described just above.  There exists a constant $C_1 > 0$ such that the following is true.  On the event $\{\tau < \infty\}$, let $\varphi \colon \h \setminus \gamma'([0,\tau]) \to \h$ be the unique conformal transformation with $\varphi(\gamma'(\tau)) = 0$ and $\varphi(i) = i$.  For each $\varsigma \in (0,1)$ we have that $B(i,C_1 e^{(1-\varsigma) \beta}) \subseteq \varphi(B(i,e^{-\varsigma \beta/2}))$.
\end{lemma}
\begin{proof}
Throughout, we shall suppose that $\{\tau < \infty\}$ occurs.  Fix $\varsigma \in (0,1)$.  The probability that a Brownian motion starting from $i$ hits $\partial B(i,e^{-\varsigma \beta})$ before hitting $\partial \h \cup \gamma'([0,\tau])$ is $O(e^{-(1-\varsigma) \beta/2})$ by the \hyperref[thm::beurling]{Beurling estimate}.  By the conformal invariance of Brownian motion, the probability of the event $X$ that a Brownian motion starting from $i$ exits $\varphi(B(i,e^{-\varsigma \beta}))$ in $\varphi(\partial B(i,e^{-\varsigma \beta}))$ is also $O(e^{-(1-\varsigma) \beta/2})$.  Let
\[ d = \dist(\varphi(\partial B(i,e^{-\varsigma \beta})),i).\]
We claim $\p[X]\gtrsim d^{-1}$.  Indeed, $X_1 \cap X_2 \subseteq X$ where $X_1$ is the event that the Brownian motion exits $\partial B(0,d)$ before hitting $\partial \h$ at a point with argument in $[\tfrac{\pi}{4},\tfrac{3\pi}{4}]$ and $X_2$ is the event that it hits $\varphi(\partial B(i,e^{-\varsigma \beta}))$ after hitting $\partial B(0,d)$ before hitting $\partial \h$.  It is easy to see that $\p[X_1] \gtrsim d^{-1}$ and $\p[X_2 \giv X_1]\gtrsim 1$.  Consequently, $e^{-(1-\varsigma) \beta/2}\gtrsim d^{-1}$ hence $d \gtrsim e^{(1-\varsigma) \beta/2}$, as desired.
\end{proof}

We now define the perfect points of $\lightcone(\theta)$ using these events as follows.  We suppose that $u_1 \in \partial \h \setminus \{0\}$ and that $h_1$ is a GFF on $\h$ with boundary data given by
\[ h_1|_{(-\infty,u_1]} \equiv -\lambda'+2\pi\chi,\quad h_1|_{(u_1,0]} \equiv \lambda',\quad\text{and}\quad h_1|_{(0,\infty)} \equiv -\lambda' \quad\text{if}\quad u_1 < 0\]
and
\[ h_1|_{(-\infty,0]} \equiv \lambda',\quad h_1|_{(0,u_1]} \equiv -\lambda',\quad\text{and}\quad h_1|_{(u_1,\infty)} \equiv \lambda'-2\pi\chi \quad\text{if}\quad u_1 > 0.\]
These two possibilities correspond to the type of boundary data which arises by starting with a GFF on $\h$ with boundary data as in Figure~\ref{fig::counterflowline_hit_conformal_maps} and then applying a conformal change of coordinates which takes a given point $z \in \h$ to $i$ and leaves $0$ fixed; $u_1$ should be thought of as the image of $\infty$ under such a map.

Let $\eta_1'$ be the counterflow line of $h_1$ starting from $0$ with associated Loewner evolution $(g_t^1)$, Loewner driving function $W^1$, and let $f_t^1 = g_t^1 - W_t^1$ be its centered Loewner evolution.  Let $\tau_1$ be the first time that $\eta_1'$ hits $\partial B(i,e^{-\beta})$.  On $\{\tau_1 < \infty\}$, we let $\wt{h}_1 = h_1 \circ (f_{\tau_1}^1)^{-1} - \chi \arg ((f_{\tau_1}^1)^{-1})'$ and let $\wt{\eta}_1^L$ (resp.\ $\wt{\eta}_1^R$) be the flow line of $\wt{h}_1$ starting from $-e^{2\Delta-\beta} |f_{\tau_1}'(i)|$ (resp.\ $e^{2\Delta-\beta} |f_{\tau_1}'(i)|$) with angle $\tfrac{\theta}{2}$ (resp.\ $-\tfrac{\theta}{2}$).  We take $\eta_1^L = (f_{\tau_1}^1)^{-1}(\wt{\eta}_1^L)$ (resp.\ $\eta_1^R = (f_{\tau_1}^1)^{-1}(\wt{\eta}_1^R)$) and let $E_1 = E_{u_1}^{\beta,\Delta}(\eta_1',\eta_1^L,\eta_1^R,\R_-,\R_+)$.  Let $\tau_1^L$ (resp.\ $\tau_1^R$) be the first time that $\eta_1^L$ (resp.\ $\eta_1^R$) hits $\R_-$, $\partial f_{\tau_1}^{-1}( e^{\Delta-\beta} |f_{\tau_1}'(i)| \D)$, or $\partial B(i,e^\Delta)$ (resp.\ $\R_+$, $\partial f_{\tau_1}^{-1}( e^{\Delta-\beta} |f_{\tau_1}'(i)| \D)$, or $\partial B(i,e^\Delta)$).  Finally, we let $\varphi_1$ be the unique conformal transformation from the connected component of $\h \setminus \eta_1'([0,\tau_1])$ which contains $i$ to $\h$ with $\varphi_1(i) = i$ and $\varphi_1(\eta'(\tau_1)) = 0$.

Suppose that $k \in \N$ and that, for each $1 \leq j \leq k$, paths $\eta_j'$, $\eta_j^L$, $\eta_j^R$, $\wt{\eta}_j^L$, and $\wt{\eta}_j^R$, Loewner evolutions $(g_t^j)$ with driving functions $W^j$, centered Loewner evolutions $f_t^j = g_t^j - W_t^j$, conformal maps $\varphi_j$, stopping times $\tau_j$, $\tau_j^L$, $\tau_j^R$, GFFs $h_j$, points $u_j \in \partial \h \setminus \{0\}$, and events $E_j$ have been defined.  We then take $h_{k+1} = h_k \circ \varphi_k^{-1} - \chi \arg(\varphi_k^{-1})'$, $u_{k+1} = \varphi_{k}(u_k)$, let $\eta_{k+1}'$ be the counterflow line of $h_{k+1}$ starting from $0$, $(g_t^{k+1})$ its Loewner evolution, $W^{k+1}$ its Loewner driving function, $f_t^{k+1} = g_t^{k+1} - W_t^{k+1}$ its centered Loewner evolution, and let $\tau_{k+1}$ be the first time that $\eta_{k+1}'$ hits $\partial B(i,e^{-\beta})$.  We define $\eta_{k+1}^L$, $\eta_{k+1}^R$, and $\varphi_{k+1}$ analogously to $\eta_1^L$, $\eta_1^R$, and $\varphi_1$, respectively, and we take $\wh{\eta}_{k+1}^L = \varphi_k(\eta_k^L)$ and $\wh{\eta}_{k+1}^R = \varphi_k(\eta_k^R)$.  We let $\tau_{k+1}^L$ (resp.\ $\tau_{k+1}^R$) be the first time that $\eta_{k+1}^L$ (resp.\ $\eta_{k+1}^R$) hits $\wh{\eta}_{k+1}^L$, $f_{\tau_{k+1}}^{-1}( e^{\Delta-\beta} |f_{\tau_{k+1}}'(i)| \D)$, or $\partial B(i,e^\Delta)$ (resp.\ $\wh{\eta}_{k+1}^R$, $f_{\tau_{k+1}}^{-1}( e^{\Delta-\beta} |f_{\tau_{k+1}}'(i)| \D)$, or $\partial B(i,e^\Delta)$).
We let 
\begin{align*}
E_{k+1} &=
\begin{cases}
E_{u_{k+1}}^{\beta,\Delta}(\eta_{k+1}',\eta_{k+1}^L,\eta_{k+1}^R,\wh{\eta}_{k+1}^L,\wh{\eta}_{k+1}^R) \quad&\text{on}\quad E_k,\\
E_{u_{k+1}}^{\beta,\Delta}(\eta_{k+1}',\eta_{k+1}^L,\eta_{k+1}^R,\R_-,\R_+) \quad&\text{on}\quad E_k^c.
\end{cases}
\end{align*}
We also let
\[ E^{m,n} = \cap_{j=m+1}^n E_j \quad \text{and} \quad E^n = E^{0,n}.\]

\begin{remark}
\label{rem::ej}
We note that:
\begin{enumerate}[(i)]
\item $E^{m,n}$ can occur even if only some of or perhaps none of $E_1,\ldots,E_m$ occur and
\item the conformal maps $\varphi_j$ and starting points of the paths $\eta_j^L$ and $\eta_j^R$ are measurable with respect to $\eta_1'$.
\end{enumerate}
\end{remark}

\begin{remark}
\label{rem::harmonic_measure_tip}
The reason for assumption~\eqref{it::harmonic_measure_tip} in the definition of the events $E_j$ is that it implies that $\wh{\eta}_{j+1}^L$ (resp.\ $\wh{\eta}_{j+1}^R$) starts in $\R_-$ (resp.\ $\R_+$) for large enough values of $\Delta > 1$.  Indeed, recall that $\wh{\eta}_{j+1}^L = \varphi_j(\eta_j^L) = \varphi_j((f_{\tau_j}^j)^{-1}(\wt{\eta}_j^L))$ and the starting point of $\wt{\eta}_j^L$ is given by $-e^{2\Delta-\beta}|f_{\tau_j}'(i)|$.  Assumption~\eqref{it::harmonic_measure_tip} implies that there exists a constant $c_1 > 0$ such that $\varphi_j \circ (f_{\tau_j}^j)^{-1}$ maps $[-c e^{4\Delta} \im( f_{\tau_j}(i)),0]$ into $\R_-$.  The claim for $\wh{\eta}_{j+1}^L$ follows since Lemma~\ref{lem::distortion} implies that $\im( f_{\tau_j}(i)) \asymp |f_{\tau_j}'(i)| e^{-\beta}$ and the claim for $\wh{\eta}_{j+1}^R$ is proved analogously.
\end{remark}

For $z \in \h$, we let $\psi_z \colon \h \to \h$ be the unique conformal map with $\psi_z(0) = 0$ and $\psi_z(z) = i$.  We define the events $E_n(z)$, $E^{m,n}(z)$, and $E^n(z)$ exactly in the same manner as $E_n$, $E^{m,n}$, and $E^n$ except in terms of the paths which arise after applying the change of coordinates $\psi_z$.  We similarly define paths $\eta_{z,k}'$, $\eta_{z,k}^L$, $\eta_{z,k}^R$, $\wh{\eta}_{z,k}^L$, $\wh{\eta}_{z,k}^R$, stopping times $\sigma_{z,k}$, $\tau_{z,k}$, $\tau_{z,k}^L$, $\tau_{z,k}^R$, and conformal maps $\varphi_{z,j}$. (In other words, everything defined as above except starting with the GFF $h_1 \circ \psi_z^{-1} - \chi \arg( \psi_z^{-1})'$ in place of $h_1$.)

\subsection{Estimates of probabilities}

We are now going to give the one and two point estimates for the perfect points and then complete the proof of the lower bound for Theorem~\ref{thm::lightcone_dimension} and Theorem~\ref{thm::sle_kappa_rho_dimension}.  Throughout, we let
\[ \alpha = 2-d(\kappa,\theta)\]
where $d(\kappa,\theta)$ is as in~\eqref{eqn::lightcone_dimension_formula}.  Recall from the proof of Proposition~\ref{prop::dimension_upper_bound} that this is the value of $\alpha$ from~\eqref{eqn::non_intersection_exponent} with the choice of parameters $a=\lambda'$, $b=-\lambda'$, $\theta_1 = \tfrac{\theta}{2}$, and $\theta_2 = -\tfrac{\theta}{2}$.  Since $\lightcone(\theta)$ is increasing in $\theta$, to prove the theorem we may assume without loss of generality that $\theta \in [0,\pi]$ is such that $d(\kappa,\theta) < 2$.

\begin{proposition}
\label{prop::perfect_one_point}
We have that 
\[ \p[E^n] \asymp e^{-\beta(1+o_{\beta}(1)) \alpha n}\]
where the $o_{\beta}(1)$ term tends to $0$ as $\beta \to \infty$.  Moreover, the rate at which the $o_{\beta}(1)$ term tends to $0$ as $\beta \to \infty$ and the constants in $\asymp$ depend only on $\kappa$, $\theta$, $\Delta$, and $u_1$.
\end{proposition}

The proof of Proposition~\ref{prop::perfect_one_point} has two inputs.  The first is the following lemma and the second is Lemma~\ref{lem::perfect_conditional}.

\begin{lemma}
\label{lem::perfect_one_point}
There exists $\Delta_0 > 1$ such that for all $\beta > \Delta^2 > \Delta > \Delta_0$ we have that 
\begin{equation}
\label{eqn::e1}
\p[E_1] \asymp e^{-\beta(1+ o_{\beta}(1)) \alpha}
\end{equation}
where the $o_{\beta}(1)$ term tends to $0$ as $\beta \to \infty$.  Moreover, the rate at which the $o_{\beta}(1)$ term tends to $0$ as $\beta \to \infty$ and the constants in $\asymp$ depend only on $\kappa$, $\theta$, $\Delta$, and $u_1$.
\end{lemma}
\begin{proof}
Let $F_1$ be the event that $\{\tau_1 < \infty\}$, $\eta_1'([0,\tau_1]) \subseteq B(i,e^\Delta)$, and the harmonic measure of the left (resp.\ right) side of $\eta_1'([0,\tau_1])$ as seen from $i$ is at least $\tfrac{1}{2} - e^{-4\Delta}$.  By \cite[Equation~(3.8) of Lemma~3.4]{MW_INTERSECTIONS}, there exists a universal constant $C > 0$ such that
\begin{equation}
\label{eqn::0_plus_0_minus}
 |f_{\tau_1}(0^-)| + |f_{\tau_1}(0^+)| \leq C e^{\Delta} \quad\text{on}\quad F_1.
\end{equation}
For $q \in \{L,R\}$, let $\wt{\zeta}_1^q$ (resp.\ $\wh{\zeta}_1^q$) be the first time that $\wt{\eta}_1^q$ hits $\partial \D$ (resp.\ $\partial (\tfrac{1}{2} \D)$).  Let $F_2$ be the event that all of the following hold:
\begin{enumerate}[(i)]
\item $\wt{\eta}_1^L([0,\wt{\zeta}_1^L]) \cap \wt{\eta}_1^R([0,\wt{\zeta}_1^R]) = \emptyset$,
\item $\dist(\wt{\eta}_1^L([\wh{\zeta}_1^L,\wt{\zeta}_1^L]), \wt{\eta}_1^R([\wh{\zeta}_1^R,\wt{\zeta}_1^R])) \geq e^{-\Delta}$, and
\item $\wt{\eta}_1^q([0,\wt{\zeta}_1^q]) \cap \big( e^{\Delta-\beta} |f_{\tau_1}'(i)| \D \big) = \emptyset$ for $q \in \{L,R\}$.
\end{enumerate}
Lemma~\ref{lem::non_intersection_exponent_general} together with Lemma~\ref{lem::rn} implies that 
\[ \p[ F_2 \giv F_1, |f_{\tau_1}'(i)|] \asymp \big( e^{-\beta} |f_{\tau_1}'(i)| \big)^\alpha\]
where the constants in $\asymp$ depend only on $\kappa$, $\theta$, $\Delta$, and $u_1$.  Combining, we have that
\begin{equation}
\label{eqn::f1_given_f2_prob}
 \p[F_1 \cap F_2] \asymp e^{-\beta \alpha} \E[ |f_{\tau_1}'(i)|^\alpha \one_{F_1}].
\end{equation}
Applying Proposition~\ref{prop::derivative_estimate} as in the proof of Proposition~\ref{prop::dimension_upper_bound} with the value of $r$ as in~\eqref{eqn::r_choice} (except we use~\eqref{eqn::derivative_estimate_escape} in place of~\eqref{eqn::derivative_estimate}) we see that
\begin{equation}
\label{eqn::f1_f2_prob}
 \p[ F_1 \cap F_2] \asymp e^{-\beta(1+o_{\beta}(1))\alpha}
\end{equation}
where the rate at which the $o_\beta(1)$ term tends to zero as $\beta \to \infty$ and the constants in $\asymp$ depend only on $\kappa$, $\theta$, $\Delta$, and $u_1$.  Let $F_3$ be the event that $\wt{\eta}_1^L|_{[\wt{\zeta}_1^L,\infty)}$ (resp.\ $\wt{\eta}_1^R|_{[\wt{\zeta}_1^R,\infty)}$) hits the left (resp.\ right) component of $\R \setminus f_{\tau_1}(\eta'([0,\tau_1]))$ before intersecting $\wt{\eta}_1^R$ (resp.\ $\wt{\eta}_1^L$) and before intersecting $e^{\Delta-\beta}|f_{\tau_1}'(i)| \D$.  Then~\eqref{eqn::0_plus_0_minus}, \cite[Lemma~2.3]{MW_INTERSECTIONS}, and \cite[Lemma~2.5]{MW_INTERSECTIONS} together imply that there exists a constant $p_1 > 0$ depending only on $\kappa$, $\theta$, $\Delta$, and $u_1$ such that 
\[ \p[ F_3 \giv F_1,F_2] \geq p_1.\]
This proves the lemma since $E_1 = \cap_{i=1}^3 F_i$.
\end{proof}

\begin{figure}[ht!]
\begin{center}
\includegraphics[scale=0.85]{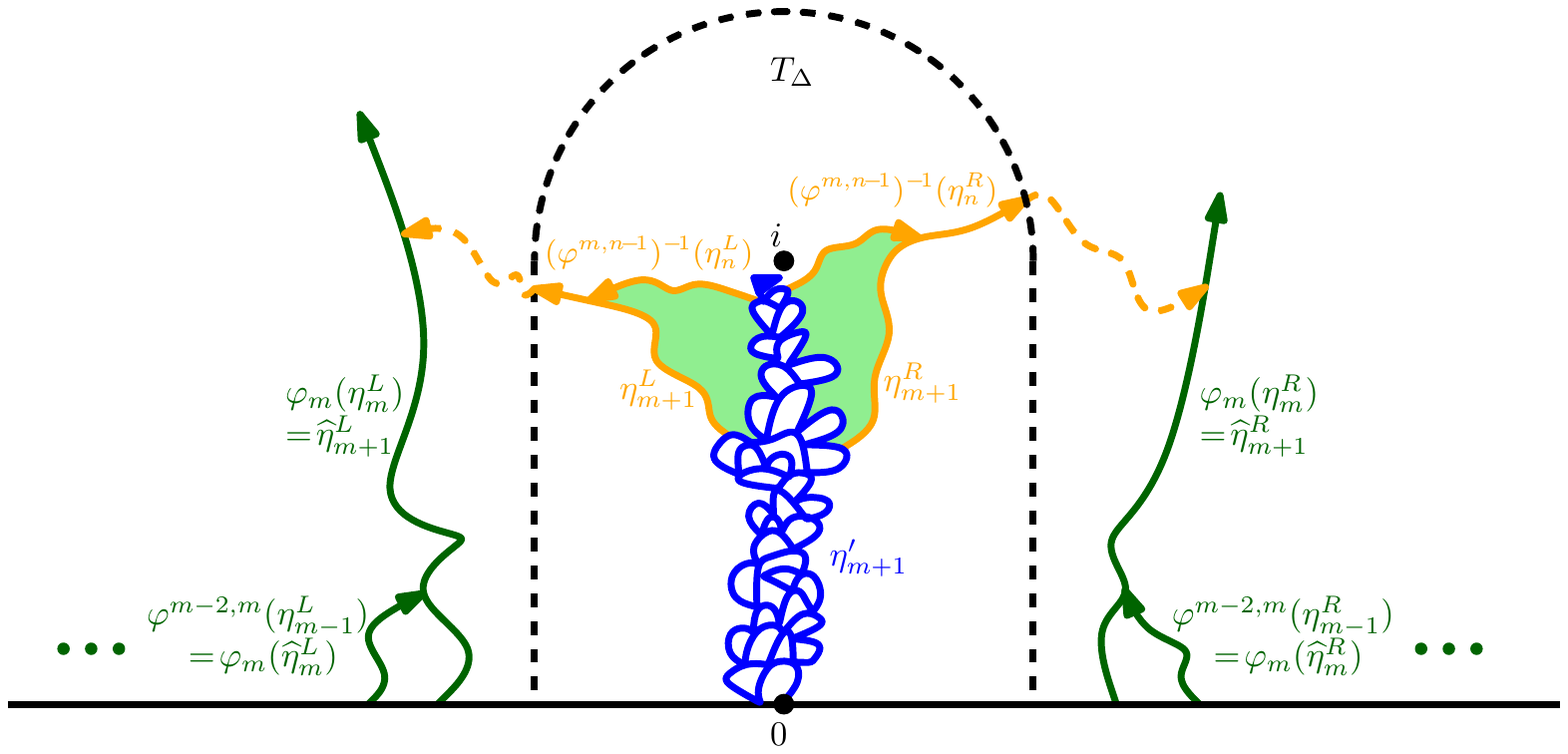}
\end{center}
\caption{\label{fig::perfect_conditional} {\small Setup for the proof of Lemma~\ref{lem::perfect_conditional}, which gives the approximate independence of the events $E^m$ and $E^{m,n}$.  Recall that $\varphi^{j,k} = \varphi^{j+1} \circ \cdots \circ \varphi^k$.  Let $T_\Delta$ be the $e^{-\Delta}$-neighborhood of $[0,i]$.  On the event that $\eta_{m+1}^L$ and $(\varphi^{m,n-1})^{-1}(\eta_n^L)$ merge as shown, the paths $(\varphi^{m,m+1})^{-1}(\eta_{m+2}^L),\ldots,(\varphi^{m,n-2})^{-1}(\eta_{n-1}^L)$ stopped upon merging with either $\eta_{m+1}^L$ or $(\varphi^{m,n-1})^{-1}(\eta_n^L)$ are contained in the light green region which is to the left of $\eta_{m+1}'$.  The analogous statement holds for the paths $(\varphi^{m,m+1})^{-1}(\eta_{m+2}^R),\ldots,(\varphi^{m,n-2})^{-1}(\eta_{n-1}^R)$.}}
\end{figure}

Let
\begin{equation}
\label{eqn::phi_iterate}
\varphi^{j,k} = \varphi_{j+1} \circ \cdots \circ \varphi_{k} \quad\text{for}\quad j < k \quad\text{and}\quad \varphi^k = \varphi^{0,k}.
\end{equation}

\begin{lemma}
\label{lem::perfect_conditional}
There exists $\Delta_0 > 1$ such that for all $\beta > \Delta^2 > \Delta > \Delta_0$ we have that
\[\p[E^{m,n} \giv E^k, E_m] \asymp \p[E^{n-m}] \quad\text{for}\quad 1 \leq k \leq m \leq n-1\]
where the constants in $\asymp$ depend only on $\kappa$, $\theta$, $\Delta$, and $u_1$.
\end{lemma}
\begin{proof}
See Figure~\ref{fig::perfect_conditional} for an illustration of the setup of the proof.  Let $T_\Delta$ the be the $e^{-\Delta}$-neighborhood of $[0,i]$.  Throughout, unless explicitly stated otherwise, we shall assume that the paths in the proof are stopped upon exiting $T_\Delta$.  By the definition of the stopping times $\tau_j^q$, we know that $\eta_j^q([0,\tau_j^q])$ does not intersect $(f_{\tau_j}^j)^{-1}(e^{\Delta - \beta} |(f_{\tau_j}^j)'(i)| \D)$ for $q \in \{L,R\}$ and $1 \leq j \leq m$.  By Lemma~\ref{lem::distortion}, we know that $\im(f_{\tau_j}^j(i)) \asymp |(f_{\tau_j}^j)'(i)|  e^{-\beta}$.  It follows that the probability that a Brownian motion starting from $f_{\tau_j}^j(i)$ hits $\partial (e^{\Delta-\beta} |(f_{\tau_j}^j)'(i)| \D)$ before hitting $\partial \h$ is $\asymp e^{-\Delta}$.  Consequently, it follows from the conformal invariance of Brownian motion that there exists a constant $c > 0$ such that $\varphi_j((f_{\tau_j}^j)^{-1}(e^{\Delta - \beta} |(f_{\tau_j}^j)'(i)| \D)) \cap B(i, c e^\Delta) = \emptyset$.  This, in turn, implies that $\varphi^{j-1,m}(\eta_j^q([0,\tau_j^q]))$ does not intersect $B(i,c e^{\Delta})$ for $q \in \{L,R\}$ and $1 \leq j \leq m$.  Let $\CG_m = \sigma(\varphi^{j-1,m}(\eta_j^q([0,\tau_j^q])) : q \in \{L,R\},\ 1 \leq j \leq m)$.  By Lemma~\ref{lem::rn}, we thus have that the law of $\eta_{m+1}'$ given $\CG_m$ and the law of $\eta_1'$ are mutually absolutely continuous with Radon-Nikodym derivative which is bounded from above and below by a finite and positive constant provided we make $\Delta_0 > 1$ sufficiently large.  Fix a path $\gamma'$ which is contained in the support of the law of $\eta_1'$.  Again applying Lemma~\ref{lem::rn}, we also have that the Radon-Nikodym derivative between the conditional law of $(\eta_1^L,\eta_1^R)$ given $\eta_1' = \gamma'$ and $(\eta_{m+1}^L,\eta_{m+1}^R)$ given $\CG_m$ and $\eta_{m+1}' = \gamma'$ is bounded from above and below by finite and positive constants.  Fix a pair of paths $(\gamma_1^L,\gamma_1^R)$ which are contained in the support of the law of $(\eta_1^L,\eta_1^R)$ given $\eta_1' = \gamma'$.  Applying Lemma~\ref{lem::rn} a final time, we have that the Radon-Nikodym derivative between the conditional law of $(\eta_{n-m}^L,\eta_{n-m}^R)$ given $\eta_1' = \gamma'$ and $(\eta_1^L,\eta_1^R) = (\gamma_1^L,\gamma_1^R)$ and the conditional law of $(\eta_n^L,\eta_n^R)$ given $\CG_m$, $\eta_{m+1}' = \gamma'$, and $(\eta_{m+1}^L,\eta_{m+1}^R) = (\gamma_1^L,\gamma_1^R)$ is bounded from above and below by finite and positive constants.  If $\gamma_1^L$ and $(\varphi^{n-m-1})^{-1}(\gamma_{n-m}^L)$ intersect each other and $\gamma_1^R$ and $(\varphi^{n-m-1})^{-1}(\gamma_{n-m}^R)$ intersect each other, then the conditional law of $(\varphi^{1,2})^{-1}(\eta_2^q),\ldots,(\varphi^{1,n-m-2})^{-1}(\eta_{n-m-1}^q)$ for $q \in \{L,R\}$ given $\eta_1'=\gamma'$, $(\eta_1^L,\eta_1^R) = (\gamma_1^L,\gamma_1^R)$, and $(\eta_{n-m}^L,\eta_{n-m}^R) = (\gamma_{n-m}^L,\gamma_{n-m}^R)$ is equal to the conditional law of  $(\varphi^{m,m+1})^{-1}(\eta_{m+2}^q),\ldots,(\varphi^{m,n-2})^{-1}(\eta_{n-1}^q)$ for $q \in \{L,R\}$ given $\CG_m$, $\eta_{m+1}'=\gamma'$, $(\eta_{m+1}^L,\eta_{m+1}^R) = (\gamma_1^L,\gamma_1^R)$, and $(\eta_n^L,\eta_n^R) = (\gamma_{n-m}^L,\gamma_{n-m}^R)$.

By \cite[Lemma~2.5]{MW_INTERSECTIONS}, we know that the conditional probability that $\eta_1^L$ and $\eta_1^R$ hit $\R_-$ and $\R_+$, respectively, before leaving $B(i,e^\Delta)$ and intersecting each other given their realization up until exiting $T_\Delta$ and the other paths is uniformly positive.  Similarly, \cite[Lemma~2.5]{MW_INTERSECTIONS} implies that, on $E^m$, the conditional probability that $\eta_{m+1}^L$ and $\eta_{m+1}^R$ merge into $\wh{\eta}_{m+1}^L$ and $\wh{\eta}_{m+1}^R$, respectively, before leaving $B(i,e^\Delta)$ given their realization up until exiting $T_\Delta$ and the other paths is uniformly positive.  Combining everything completes the proof.
\end{proof}

\begin{proof}[Proof of Proposition~\ref{prop::perfect_one_point}]
This follows by using Lemma~\ref{lem::perfect_conditional} to iterate the estimate from Lemma~\ref{lem::perfect_one_point}.
\end{proof}

Now that we have proved the one point estimate for the perfect points, we turn to establish the two point estimate.  We let
\begin{equation}
\label{eqn::phi_z_iterate}
\varphi_z^{j,k} = \varphi_{z,j+1} \circ \cdots \circ \varphi_{z,k} \quad\text{for}\quad j < k \quad\text{and}\quad \varphi_z^k = \varphi_z^{0,k}.
\end{equation}
For each $n \in \N$ and $z \in \h$, we also let
\begin{align*}
 \ul{V}_n(z) &= B(z,2^{-5 n-5} \im(z) e^{-n \beta}) \quad \text{and} \quad \ol{V}_n(z)  = B(z, 2^{5 n+5} \im(z) e^{-n \beta}).
\end{align*}

\begin{lemma}
\label{lem::v_contain}
There exists $\Delta_0 > 1$ such that for all $\beta >\Delta^2> \Delta \geq \Delta_0$, the following are true.
\begin{enumerate}[(i)]
\item\label{it::in_v} For each $m,n \in \N$ with $m \geq n+2$, on $\{\tau_{z,m} < \infty\}$ we have that $\psi_z^{-1} \circ (\varphi_z^{m-1})^{-1}(\gamma) \subseteq \ol{V}_n(z)$ for $\gamma = \eta_{z,m}^q|_{[0,\tau_{z,m}^q]}$ for $q \in \{L,R\}$ and for $\gamma = \eta_{z,m}'|_{[0,\tau_{z,m}]}$.
\item\label{it::out_v} For each $m,n \in \N$ with $m +2\leq n$, on $\{\tau_{z,m} < \infty\}$ we have that $\psi_z^{-1} \circ (\varphi_z^{m-1})^{-1}(\gamma) \cap \ul{V}_n(z) = \emptyset$ for $\gamma = \eta_{z,m}^q|_{[0,\tau_{z,m}^q]}$ for $q \in \{L,R\}$ and for $\gamma = \eta_{z,m}'|_{[0,\tau_{z,m}]}$.
\end{enumerate}
\end{lemma}
\begin{proof}
We are first going to give the proof in the case that $z=i$ and we will first establish part~\eqref{it::in_v}.  Fix $m,n \in \N$ with $m \geq n+2$.  Throughout, we shall assume that we are working on $E^m$.  It follows from Lemma~\ref{lem::ball_size} that if $r \in (0,\tfrac{1}{2}]$ then
\begin{align}
\label{eqn::varphi_ball}
 B(i,\tfrac{1}{16} r e^{-\beta}) \subseteq \varphi_k^{-1}(B(i,r)) \subseteq B (i,16 r e^{-\beta}) \quad&\text{for}\quad 1 \leq k \leq m.
\end{align}
Iterating~\eqref{eqn::varphi_ball} implies that
\begin{equation}
\label{eqn::varphi_iterate_contain}
\begin{split}
B(i,2^{-5 k} e^{-k\beta}) \subseteq (\varphi^k)^{-1}(B(i,\tfrac{1}{2})) \subseteq B(i,2^{5 k} e^{-k \beta }) \quad& \text{for}\quad 1 \leq k \leq m
\end{split}
\end{equation}
(provided we take $\Delta_0 > 1$ large enough).

Note that $\eta_m^q([0,\tau_m^q]) \subseteq B(i, e^{\Delta})$ for $q \in \{L,R\}$ by the definition of the events.  Consequently, it follows from Lemma~\ref{lem::beta_ball_contained} that $\varphi_{m-1}^{-1}(\eta_m^q([0,\tau_m^q])) \subseteq B(i,e^{-\beta/4})$ for $q \in \{L,R\}$ provided $\Delta_0 > 1$ is large enough.  We also assume that $\Delta_0 > 1$ is sufficiently large so that $e^{-\Delta_0/4} < \tfrac{1}{2}$.  Applying~\eqref{eqn::varphi_iterate_contain} proves part~\eqref{it::in_v} for $\eta_m^q|_{[0,\tau_m^q]}$ for $q \in \{L,R\}$; the proof for $\eta_m'|_{[0,\tau_m]}$ is analogous.  This proves part~\eqref{it::in_v} for $z=i$.  For the case that $z \neq i$, we note that applying Lemma~\ref{lem::ball_size} with $r \in (0,\tfrac{1}{2}]$ again yields,
\begin{equation}
\label{eqn::psi_z_contain}
B(i,\tfrac{1}{16} r \im(z)) \subseteq \psi_z^{-1}(B(i,r)) \subseteq B(i,16 r \im(z)).
\end{equation}
Combining~\eqref{eqn::varphi_iterate_contain} with~\eqref{eqn::psi_z_contain} gives part~\eqref{it::in_v}.  The proof of part~\eqref{it::out_v} is the same.
\end{proof}

\begin{proposition}
\label{prop::perfect_two_point}
Fix $\delta \in (0,\tfrac{\pi}{2})$.  Suppose that $z,w \in \h \cap \D$ are distinct with $\arg(z),\arg(w) \in (\delta,\pi-\delta)$.  Let $m$ be the smallest integer such that $\ol{V}_{m-1}(z) \cap \ol{V}_{m-1}(w) = \emptyset$.  Then we have that
\[ \p[E^n(z), E^n(w)] \lesssim e^{O(\beta) + \beta (1+o_{\beta}(1)) \alpha m} \p[E^n(z)]\p[E^n(w)]\]
where the~$o_{\beta}(1)$ term tends to zero as $\beta \to \infty$.  Moreover, the rate at which the~$o_\beta(1)$ term tends to zero and the constants in the~$O(\beta)$ term and~$\asymp$ depend only on~$\kappa$, $\theta$, $u_1$, and~$\delta$.
\end{proposition}

Before we prove Proposition~\ref{prop::perfect_two_point}, we will need to collect the following lemma which is the analog of Lemma~\ref{lem::perfect_conditional} in the setting of two points.  For each $z \in \h$ and $k \in \N$, we let $\CF_m^k(z)$ be the $\sigma$-algebra generated by $\eta_{z,j}'|_{[0,\tau_{z,j}']}$ for $1 \leq j \leq k$ and $\eta_{z,j}^q|_{[0,\tau_{z,j}^q]}$ for $q \in \{L,R\}$ and $1 \leq j \leq k$, $j \neq m-4,\ldots,m$.  See Figure~\ref{fig::perfect_two_point} for an illustration of the setup as well as the proof.

\begin{figure}[ht!]
\begin{center}
\includegraphics[scale=0.85]{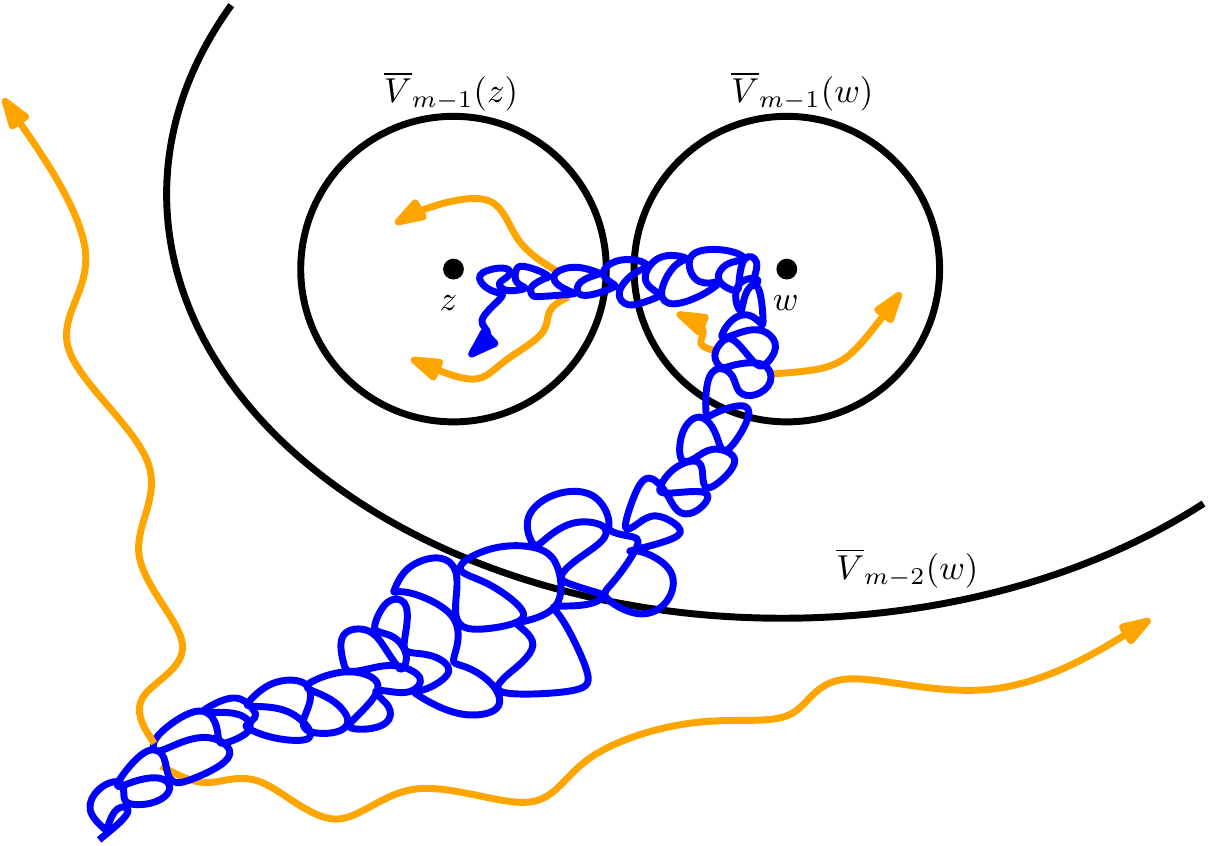}
\end{center}
\caption{\label{fig::perfect_two_point} Illustration of the setup for the two-point estimate proved in Lemma~\ref{lem::perfect_two_point_conditional}.  Shown is $\eta_1'$ on the event that it hits $\ol{V}_k(w)$ (not shown) before hitting $\ol{V}_{m-1}(z)$ as well as the auxiliary paths involved in the definition of the events $E^{m-5}(w)$, $E^{m,k}(w)$, and $E^{m,n}(z)$.  By Lemma~\ref{lem::v_contain}, these paths are contained in $\h \setminus \ol{V}_{m-2}(w)$, $\ol{V}_{m-1}(w)$, and $\ol{V}_{m-1}(z)$ respectively.  Since these regions are disjoint, we can use the Radon-Nikodym estimate Lemma~\ref{lem::rn} to show that the events $E^{m-5}(w)$, $E^{m,k}(w)$, and $E^{m,n}(z)$ are approximately independent.}
\end{figure}

\begin{lemma}
\label{lem::perfect_two_point_conditional}
There exists $\Delta_0 > 1$ such that $\beta > \Delta^2 > \Delta > \Delta_0$ implies that the following is true.  Fix $\delta \in (0,\tfrac{\pi}{2})$ and suppose that $z,w \in \h \cap \D$ are distinct with $\arg(z),\arg(w) \in (\delta,\pi-\delta)$.  Let $m$ be the smallest integer such that $\ol{V}_{m-1}(z) \cap \ol{V}_{m-1}(w) = \emptyset$.  Fix $n \geq m$ and let $P_w$ be the event that $\eta'$ hits $\ol{V}_k(w)$ before hitting $\ol{V}_{m-1}(z)$.  Let $E_m^k(w) = E^{m-5}(w) \cap E^{m,k}(w)$.  For all $k \geq m$, we have that
\begin{equation}
\label{eqn::perfect_two_point_conditional}
\begin{split}
 &\p[E^{m+1,n}(z) \giv \CF_m^k(w), E_{m+1}] \one_{E_m^k(w), P_w} \asymp \p[E^{n-m-1}] \one_{E_m^k(w), P_w}
\end{split}
\end{equation}
where the constants in $\asymp$ depend only on $\kappa$, $\theta$, $\Delta$, $u_1$, and $\delta$.
\end{lemma}
\begin{proof}
We assume that $\Delta_0 > 1$ is sufficiently large so that $\beta > \Delta_0$ implies that $\ul{V}_{n+1}(u) \subseteq \ol{V}_n(u)$ for all $n \in \N$ and $u \in \h$ and also so that Lemma~\ref{lem::v_contain} holds.  The proof is analogous to that of Lemma~\ref{lem::perfect_conditional}.  By applying $\psi_z$, we may assume without generality that $z=i$.  Note that the event $E^{m,n}$ is defined in terms of the paths $(\varphi^{j-1})^{-1}(\eta_j'|_{[0,\tau_j]})$ and $(\varphi^{j-1})^{-1}(\eta_j^q|_{[0,\tau_j^q]})$ for $q \in \{L,R\}$ and $m+1 \leq j \leq n$ and that the event $E_m^k(z)$ is defined in terms of the paths $(\varphi_w^{j-1})^{-1}(\eta_{w,j}'|_{[0,\tau_{w,j}]})$ and $(\varphi_w^{j-1})^{-1}(\eta_{w,j}^q|_{[0,\tau_{w,j}^q]})$ for $q \in \{L,R\}$ and $1 \leq j \leq k$, $j \neq m-4,\ldots,m$.  Lemma~\ref{lem::v_contain} implies that the paths involved in the definition of $E^{m,n}$ (resp.\ $E^{m,k}(w)$) are contained in $\ol{V}_{m-1}(i)$ (resp.\ $\ol{V}_{m-1}(w)$).  Lemma~\ref{lem::v_contain} also implies that the paths involved in the definition of $E^{m-5}(w)$ do not intersect $\ul{V}_{m-3}(w) \subseteq \ol{V}_{m-2}(w)$.  By the choice of $\Delta_0$, we have that $\ol{V}_{m-2}(w)$ contains $\ol{V}_{m-1}(i)$ and $\ol{V}_{m-1}(w)$.  That is, the paths involved in the definition of $E^{m,n}$ and those involved in the definition of $E_m^k(w)$ are disjoint on the event $P_w$.  Thus by conformally mapping back and using the Radon-Nikodym derivative estimate Lemma~\ref{lem::rn} as in the proof of Lemma~\ref{lem::perfect_conditional} it is not hard to see that~\eqref{eqn::perfect_two_point_conditional} holds, as desired.
\end{proof}

\begin{proof}[Proof of Proposition~\ref{prop::perfect_two_point}]
We are going to extract the result from Lemma~\ref{lem::perfect_two_point_conditional}.  Let $m$, $E_m^n(z)$, $P_z$, $E_m^n(w)$, and $P_w$ be as in the statement of Lemma~\ref{lem::perfect_two_point_conditional}  and assume that $\Delta_0 > 1$ is large enough so that Lemma~\ref{lem::perfect_one_point}, Lemma~\ref{lem::perfect_conditional}, and Lemma~\ref{lem::perfect_two_point_conditional} hold.  We have that,
\begin{align*}
      \p[ E^n(z), E^n(w)] 
 \leq &\E\left[ \p[E^{m+1,n}(z) \giv \CF_m^n(w), E_{m+1}(z)] \one_{E_m^n(w), P_w} \right] +\\
  &\E\left[ \p[E^{m+1,n}(w) \giv \CF_m^n(z), E_{m+1}(w)] \one_{E_m^n(z), P_z} \right].
\end{align*}
We are now going to explain how to bound the first summand above.  The second summand is bounded similarly, so this will complete the proof.  We have that,
\begin{align*}
&\E\left[ \p[E^{m+1,n}(z) \giv \CF_m^n(w), E_{m+1}(z)] \one_{E_m^n(w), P_w} \right]\\
\asymp& \p[E^{n-m-1}] \p[E_m^n(w)] \quad\text{(Lemma~\ref{lem::perfect_two_point_conditional})}\\
\lesssim& e^{O(\beta)} \p[E^{n-m}] \p[E^n(w)] \quad\text{(Lemma~\ref{lem::perfect_one_point} and
Lemma~\ref{lem::perfect_conditional})}\\
\asymp& \frac{e^{O(\beta)}}{\p[E^m]} \p[E^n(z)] \p[E^n(w)] \quad\text{(Lemma~\ref{lem::perfect_conditional})}\\
\leq& e^{O(\beta) + \beta(1+o_\beta(1)) \alpha m} \p[E^n(z)] \p[E^n(w)] \quad\text{(Proposition~\ref{prop::perfect_one_point})}.
\end{align*}
\end{proof}

\begin{figure}[ht!]
\subfigure[]{
\includegraphics[scale=0.85]{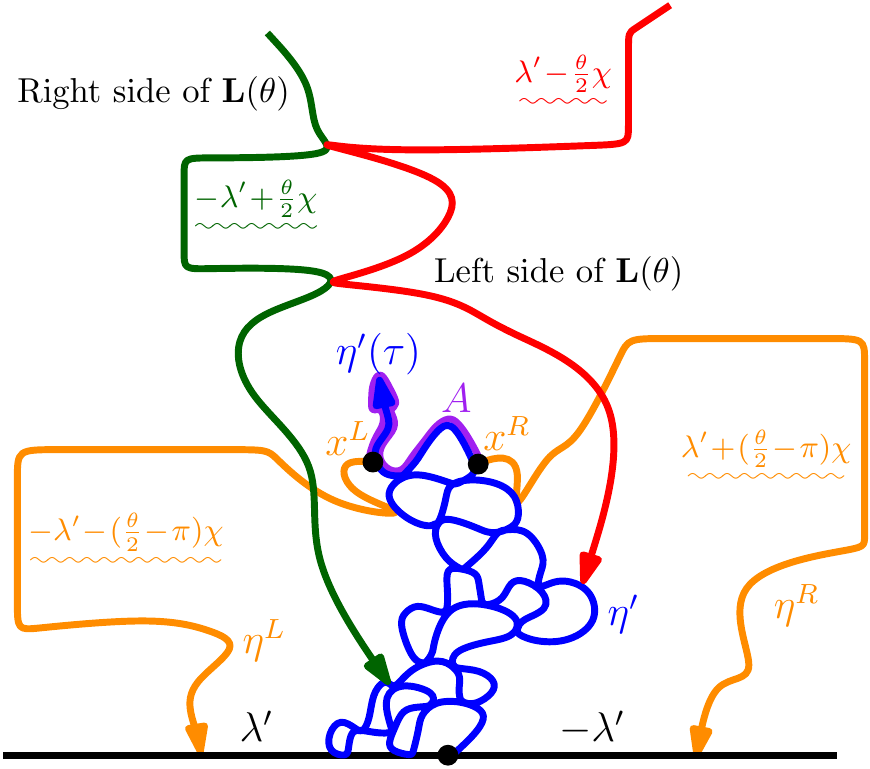}}
\hspace{0.05\textwidth}
\subfigure[Pocket $P$ of $\lightcone(\theta)$]{
\includegraphics[scale=0.85]{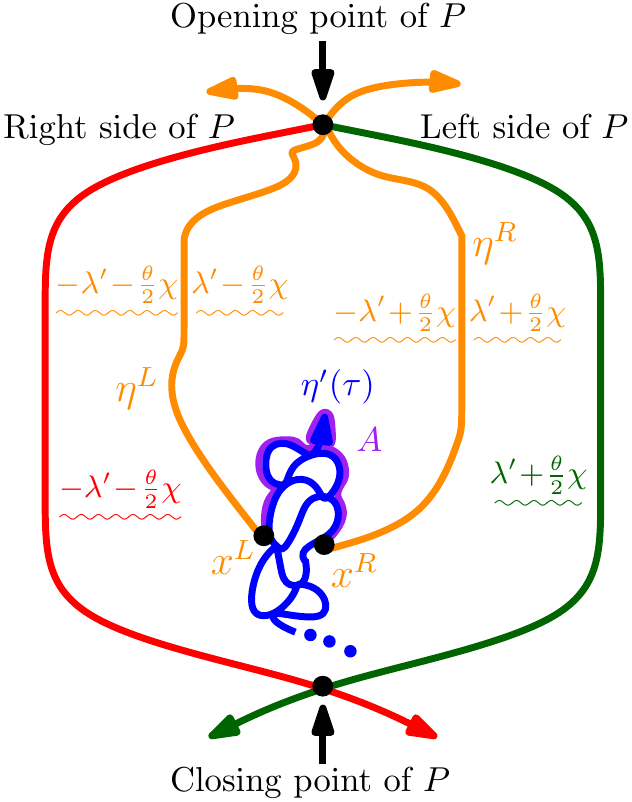}}
\caption{\label{fig::lightcone_hits}  Suppose that $\tau$ is an almost surely finite stopping time for $\eta'$.  Fix a point $x^L$ (resp.\ $x^R$) on the left (resp.\ right) side of $\eta'([0,\tau])$ and let $\eta^L$ (resp.\ $\eta^R$) be the flow line of $h$ starting from $x^L$ (resp.\ $x^R$) with angle $\tfrac{\theta}{2}$ (resp.\ $-\tfrac{\theta}{2}$) stopped upon hitting $\R_-$ (resp.\ $\R_+$).  Let $A$ be the clockwise segment of the outer boundary of $\eta'([0,\tau])$ which runs from $x^L$ to $x^R$ (purple in the illustration).  We show in Lemma~\ref{lem::lightcone_hits} that on the event that $\eta^L$ and $\eta^R$ do not intersect each other we have that $\lightcone(\theta) \cap A \neq \emptyset$ almost surely; this is the converse of the statement explained in Figure~\ref{fig::avoid_small_ball}.  The reason that the ``right'' and ``left'' sides of $\lightcone(\theta)$ appear to be flipped above is because $\lightcone(\theta)$ is growing from $\infty$.  Note that the angle of the flow line which makes up the left (resp.\ right) side of $\lightcone(\theta)$ is the same as that which makes up the right (resp.\ left) side of $P$.  In the right panel, we have only drawn part of $\eta'$ which is in $P$; the initial part of $\eta'$ has been omitted.  Note that the initial part of $\eta'$ in particular contains the closing point of $P$ because counterflow lines visit the ranges of flow lines in reverse chronological order.  This is what prevents $\eta^L,\eta^R$ from exiting $P$ at the closing point.
}
\end{figure}

See the left side of Figure~\ref{fig::lightcone_hits} for an illustration of the of the setup of the following lemma, which we will use to show that the perfect points are almost surely contained in $\lightcone(\theta)$.

\begin{lemma}
\label{lem::lightcone_hits}
Suppose that $\tau$ is an almost surely finite stopping time for $\eta'$.  Fix $x^L$ (resp.\ $x^R$) on the left (resp.\ right) side of $\eta'([0,\tau])$ and let $\eta^L$ (resp.\ $\eta^R$) be the flow line with angle $\tfrac{\theta}{2}$ (resp.\ $-\tfrac{\theta}{2}$) starting from $x^L $ (resp.\ $x^R$) stopped upon hitting $\R_-$ (resp.\ $\R_+$).  Let $A$ be the segment on the outer boundary of $\eta'([0,\tau])$ which runs from $x^L$ to $x^R$ with a clockwise orientation.  On the event that $\eta^L$ and $\eta^R$ do not intersect each other, we have that $\lightcone(\theta) \cap A \neq \emptyset$ almost surely.
\end{lemma}
\begin{proof}
It follows from the flow line interaction rules \cite[Theorem~1.5]{MS_IMAG} that the left side of $\lightcone(\theta)$ cannot cross $\eta^L$ from right to left (otherwise it would intersect with a height difference of $-\pi \chi$) and it cannot hit $\R_-$.  Similarly, the right side of $\lightcone(\theta)$ cannot cross $\eta^R$ from left to right and it cannot hit $\R_+$.  It thus follows that either the left or right side of $\lightcone(\theta)$ hits $A$ or $A$ is contained inside the region surrounded by the outer boundary of $\lightcone(\theta)$.  In the former case, there is nothing to prove so we shall assume that we are in the latter case.  Assume for contradiction that $\lightcone(\theta)$ does not intersect $A$.  Then $x^L$ and $x^R$ are both contained in a common complementary pocket $P$ of $\lightcone(\theta)$ as shown in the right panel of Figure~\ref{fig::lightcone_hits}.  It follows from the flow line interaction rules that $\eta^L$ cannot intersect the right side of $\partial P$ (otherwise it would intersect with a height difference of $-\pi \chi$, counted from right to left) and $\eta^R$ cannot intersect the left side of $\partial P$ (otherwise it would intersect with a height difference of $-\pi \chi$, counted from right to left).  Moreover, $\eta^L$ (resp.\ $\eta^R$) is prevented from intersecting the left (resp.\ right) side of $\partial P$ because doing so would force $\eta^L$ (resp.\ $\eta^R$) either to cross $\eta^R$ (resp.\ $\eta^L$) or to cross $\eta'$ (not shown in the right panel of Figure~\ref{fig::lightcone_hits}).  Therefore the only possibility is that both $\eta^L$ and $\eta^R$ exit $P$ from its opening point as shown in the right panel of Figure~\ref{fig::lightcone_hits}.  This is a contradiction because then $\eta^L$ and $\eta^R$ are forced to intersect at the pocket opening point.
\end{proof}

For each $\beta > 1$, let $\CD_n^\beta$ be the set of squares with corners in $e^{-\beta n} \Z^2$ which are contained in $[-1,1] \times [1,2]$.  As before, we let $z(Q)$ denote the center of a given square $Q$ and, for each $z \in [-1,1] \times [1,2]$ and $n \in \N$, we let $Q_n(z)$ denote the element of $\CD_n^\beta$ which contains $z$.  Let $\CE_n^\beta = \{z(Q) : Q \in \CD_n^\beta\}$.    For each $\Delta > 1$ such that $\beta > \Delta^2 > \Delta$, let $\CC_n^{\beta,\Delta}$ consist of those $Q \in \CD_n^\beta$ with $Q \subseteq [-1,1] \times [1,2]$ for which $E^n(z(Q))$ occurs.  Let
\[ \CP^{\beta,\Delta} = \ol{\bigcap_{n \in \N} \bigcup_{Q \in \CC_n^{\beta,\Delta}} Q}.\]

\begin{lemma}
\label{lem::perfect_contained}
There exists $\Delta_0 > 1$ such that for each $\beta,\Delta > 1$ with $\beta > \Delta^2 > \Delta > \Delta_0$ we have that $\CP^{\beta,\Delta} \subseteq \lightcone(\theta)$ almost surely.
\end{lemma}
\begin{proof}[Proof of Lemma~\ref{lem::perfect_contained}]
This follows from Lemma~\ref{lem::lightcone_hits} and the definition of the events.
\end{proof}

We can now complete the proof of Theorem~\ref{thm::lightcone_dimension}.

\begin{proof}[Proof of Theorem~\ref{thm::lightcone_dimension}]
Standard arguments for computing the Hausdorff dimension of a random fractal imply that an estimate of the form given in Proposition~\ref{prop::perfect_two_point} combined with Lemma~\ref{lem::perfect_contained} gives that, for each $\xi >0$, the probability of the event that $\dimH(\lightcone(\theta)) \geq d(\kappa,\theta)-2\xi$ is positive (see, for example, the arguments in \cite{DPRZ01, HMP10, MSW_CLE_GASKET, MW_INTERSECTIONS,multifractal_spectrum,extreme_nesting}).  For completeness, we will include the entire argument.  For each $n \in \N$, let $\mu_n$ be the measure on $X = [-1,1]  \times [1,2]$ defined by
\[ \mu_n(A)=\int_A \sum_{z\in \CE_n^\beta}\frac{\one_{E^n(z)}}{\p[E^n(z)]}\one_{Q_n(z)}(z')dz'\quad\text{for}\quad A\subseteq X \quad\text{Borel}.\]
Then $\E[ \mu_n(X)]=1$.  Recall that $\alpha = 2-d(\kappa,\theta)$.  Moreover, we have that
\begin{align*}
\E&[\mu_n(X)^2]=e^{-4 \beta n} \sum_{z,w\in\CE_n^\beta}\frac{\p[ E^n(z)\cap E^n(w)] }{\p[E^n(z)]\p[E^n(w)]}\\
&=e^{-4 \beta n} \sum_{\substack{z,w\in\CE_n^\beta \\ z\neq w}}\frac{\p[ E^n(z)\cap E^n(w)]}{\p[E^n(z)]\p[E^n(w)]}+e^{-4 \beta n}\sum_{z\in\CE_n^\beta}\frac{1}{\p[E^n(z)]}.
\intertext{If we choose $n$, $\beta$, and $\Delta$ large enough, then applying Proposition~\ref{prop::perfect_two_point} to the first summand and Proposition~\ref{prop::perfect_one_point} to the second summand yields that the above is bounded by}
&\lesssim e^{-4 \beta n} \sum_{\substack{z,w\in\CE_n^\beta \\ z\neq w}} |z-w|^{-\alpha-\xi}+e^{-4 \beta n}\sum_{z\in\CE_n^\beta} e^{(\alpha+\xi) \beta n} \lesssim 1
\end{align*}
Set $d_\xi=d(\kappa,\theta)-2\xi$.  Let $I_{d_\xi}(\mu)$ denote the $d_\xi$-energy of $\mu$.  We also have that
{\allowdisplaybreaks
\begin{align*}
\E&[I_{d_\xi}(\mu_n)]=\sum_{z,w\in\CE_n^\beta}\frac{\p[ E^n(z)\cap E^n(w)]}{\p[ E^n(z)]\p[E^n(w)]}\iint\limits_{Q_n(z)\times Q_n(w)}\frac{dz'dw'}{|z'-w'|^{d_\xi}}\\
&=\sum_{\substack{z,w\in\CE_n^\beta \\ z\neq w}}\frac{\p[ E^n(z)\cap E^n(w)]}{\p[E^n(z)]\p[E^n(w)]}\iint\limits_{Q_n(z)\times Q_n(w)}\frac{dz'dw'}{|z'-w'|^{d_\xi}}\\
&\quad\quad\quad\quad+\sum_{z\in\CE_n^\beta}\frac{1}{\p[ E^n(z)] }\iint\limits_{Q_n(z)\times Q_n(z)}\frac{dz'dw'}{|z'-w'|^{d_\xi}}\\
&\lesssim \sum_{\substack{z,w\in\CE_n^\beta \\ z\neq w}}\frac{\p[ E^n(z)\cap E^n(w)]}{\p[E^n(z)]\p[E^n(w)]}\left(\frac{e^{-4 \beta n}}{|z-w|^{d_\xi}}\right)+\sum_{z\in\CE_n^\beta}\frac{1}{\p[ E^n(z)]} \times e^{(d_\xi-4) \beta n}\\
&\lesssim \sum_{\substack{z,w\in\CE_n^\beta \\ z\neq w}} |z-w|^{-\alpha-\xi} \times e^{-4 \beta n} \times |z-w|^{-d_\xi}+\sum_{z\in\CE_n^\beta} e^{(\alpha+\xi)\beta n} \times e^{(d_\xi-4) \beta n}
 \lesssim 1.
\end{align*}}
Consequently, the sequence $(\mu_n)$ has a subsequence $(\mu_{n_k})$ that converges weakly to some measure $\mu$ which is non-zero with positive probability. It is clear that $\mu$ is supported on $\CP^{\beta,\Delta}$ and has finite $d_\xi$-energy. From \cite[Theorem~4.27]{BM}, we know that
\[ \p\left[ \dimH(\lightcone(\theta))\geq d_\xi \right]>0.\]

It is left to explain the $0$-$1$ law: that for each $d\in [0,2]$, $\p[\dimH(\lightcone(\theta))=d]\in\{0,1\}$.  We will use the same argument used in the proof of \cite[Theorem~1.5]{MW_INTERSECTIONS}.  By swapping the roles of $0$ and $\infty$ using the conformal transformation $z \mapsto -1/z$, we now assume that $\lightcone(\theta)$ grows from $0$ towards $\infty$ rather than from $\infty$ towards $0$.  For each $r>0$, we let $D_r=\dimH(\lightcone(\theta) \cap B(0,r) \cap \h)$. It is clear that $0 < r_1<r_2$ implies $D_{r_1}\le D_{r_2}$.  By the scale invariance of the setup, we have that $D_{r_1}$ has the same law as $D_{r_2}$. Thus $D_{r_1}=D_{r_2}$ almost surely for all $0 < r_1<r_2$. In particular, $\p[D_{\infty}=D_r]=1$ for all $r>0$. Thus the events $\{D_{\infty}=d\}$ and $\{D_r=d\}$ are the same up to a set of probability zero.  The latter is measurable with respect to the $h$ restricted to $B(0,r)$. Letting $r \downarrow 0$, we see that this implies that the event $\{D_{\infty}=d\}$ is trivial, which completes the proof.
\end{proof}

\bibliographystyle{hmralphaabbrv}
\bibliography{sle_kappa_rho}

\begin{thebibliography}{MWW16}

\bibitem[Bef08]{BEF_DIM}
V.~Beffara.
\newblock The dimension of the {SLE} curves.
\newblock {\em Ann. Probab.}, 36(4):1421--1452, 2008.
\newblock \arxiv{math/0211322}. \MR{2435854 (2009e:60026)}

\bibitem[DMS14]{dms2014mating}
B.~{Duplantier}, J.~{Miller}, and S.~{Sheffield}.
\newblock {Liouville quantum gravity as a mating of trees}.
\newblock {\em ArXiv e-prints}, September 2014, \arxiv{1409.7055}.

\bibitem[DPRZ01]{DPRZ01}
A.~Dembo, Y.~Peres, J.~Rosen, and O.~Zeitouni.
\newblock Thick points for planar {B}rownian motion and the {E}rdos-{T}aylor
  conjecture on random walk.
\newblock {\em Acta Math.}, 186(2):239--270, 2001. \MR{1846031}

\bibitem[Dub09a]{DUB_DUAL}
J.~Dub{\'e}dat.
\newblock Duality of {S}chramm-{L}oewner evolutions.
\newblock {\em Ann. Sci. \'Ec. Norm. Sup\'er. (4)}, 42(5):697--724, 2009.
\newblock \arxiv{0711.1884}. \MR{2571956 (2011g:60151)}

\bibitem[Dub09b]{DUB_PART}
J.~Dub{\'e}dat.
\newblock S{LE} and the free field: partition functions and couplings.
\newblock {\em J. Amer. Math. Soc.}, 22(4):995--1054, 2009.
\newblock \arxiv{0712.3018}. \MR{2525778 (2011d:60242)}

\bibitem[GHM15]{brownian_motion_kpz}
E.~{Gwynne}, N.~{Holden}, and J.~{Miller}.
\newblock {An almost sure KPZ relation for SLE and Brownian motion}.
\newblock {\em ArXiv e-prints}, December 2015, \arxiv{1512.01223}.

\bibitem[GHM16]{dimension_transformation}
E.~{Gwynne}, N.~{Holden}, and J.~{Miller}.
\newblock {Dimension transformation formula for conformal maps into the
  complement of an SLE curve}.
\newblock {\em ArXiv e-prints}, March 2016, \arxiv{1603.05161}.

\bibitem[GMS14]{multifractal_spectrum}
E.~{Gwynne}, J.~{Miller}, and X.~{Sun}.
\newblock {Almost sure multifractal spectrum of SLE}.
\newblock {\em ArXiv e-prints}, December 2014, \arxiv{1412.8764}.
\newblock {To appear in Duke Mathematical Journal}.

\bibitem[HMP10]{HMP10}
X.~Hu, J.~Miller, and Y.~Peres.
\newblock Thick points of the {G}aussian free field.
\newblock {\em Ann. Probab.}, 38(2):896--926, 2010.
\newblock \arxiv{0902.3842}. \MR{2642894 (2011c:60117)}

\bibitem[JVL12]{LV09}
F.~Johansson~Viklund and G.~F. Lawler.
\newblock Almost sure multifractal spectrum for the tip of an {SLE} curve.
\newblock {\em Acta Math.}, 209(2):265--322, 2012.
\newblock \arxiv{0911.3983}. \MR{3001607}

\bibitem[Law05]{LAW05}
G.~F. Lawler.
\newblock {\em Conformally invariant processes in the plane}, volume 114 of
  {\em Mathematical Surveys and Monographs}.
\newblock American Mathematical Society, Providence, RI, 2005.

\bibitem[LSW03]{LSW_RESTRICTION}
G.~Lawler, O.~Schramm, and W.~Werner.
\newblock Conformal restriction: the chordal case.
\newblock {\em J. Amer.\ Math.\ Soc.}, 16(4):917--955 (electronic), 2003.
\newblock \arxiv{math/0209343}. \MR{1992830 (2004g:60130)}

\bibitem[MP10]{BM}
P.~M{\"o}rters and Y.~Peres.
\newblock {\em Brownian motion}.
\newblock Cambridge Series in Statistical and Probabilistic Mathematics.
  Cambridge University Press, Cambridge, 2010.
\newblock With an appendix by Oded Schramm and Wendelin Werner. \MR{2604525
  (2011i:60152)}

\bibitem[MS16a]{MS_LIGHTCONE}
J.~{Miller} and S.~{Sheffield}.
\newblock {Gaussian free field light cones and SLE$_\kappa(\rho)$}.
\newblock {\em ArXiv e-prints}, June 2016, \arxiv{1606.02260}.

\bibitem[MS16b]{MS_IMAG}
J.~Miller and S.~Sheffield.
\newblock Imaginary geometry {I}: interacting {SLE}s.
\newblock {\em Probab. Theory Related Fields}, 164(3-4):553--705, 2016.
\newblock \arxiv{1201.1496}. \MR{3477777}

\bibitem[MS16c]{MS_IMAG2}
J.~Miller and S.~Sheffield.
\newblock Imaginary geometry {II}: reversibility of {${\rm
  SLE}_\kappa(\rho_1;\rho_2)$} for {$\kappa\in(0,4)$}.
\newblock {\em Ann. Probab.}, 44(3):1647--1722, 2016.
\newblock \arxiv{1201.1497}. \MR{3502592}

\bibitem[MS16d]{MS_IMAG3}
J.~Miller and S.~Sheffield.
\newblock Imaginary geometry {III}: reversibility of {$\rm SLE_\kappa$} for
  {$\kappa\in(4,8)$}.
\newblock {\em Ann. of Math. (2)}, 184(2):455--486, 2016.
\newblock \arxiv{1201.1498}. \MR{3548530}

\bibitem[MS17]{MS_IMAG4}
J.~Miller and S.~Sheffield.
\newblock Imaginary geometry {IV}: interior rays, whole-plane reversibility,
  and space-filling trees.
\newblock {\em Probab. Theory Related Fields}, 169(3-4):729--869, 2017.
\newblock \arxiv{1302.4738}. \MR{3719057}

\bibitem[MSW14]{MSW_CLE_GASKET}
J.~Miller, N.~Sun, and D.~B. Wilson.
\newblock The {H}ausdorff dimension of the {CLE} gasket.
\newblock {\em Ann. Probab.}, 42(4):1644--1665, 2014.
\newblock \arxiv{1206.0725}. \MR{3262488}

\bibitem[MSW17]{cle_percolations}
J.~Miller, S.~Sheffield, and W.~Werner.
\newblock C{LE} percolations.
\newblock {\em Forum Math. Pi}, 5:e4, 102, 2017.
\newblock \arxiv{1602.03884}. \MR{3708206}

\bibitem[MW17]{MW_INTERSECTIONS}
J.~Miller and H.~Wu.
\newblock Intersections of {SLE} paths: the double and cut point dimension of
  {SLE}.
\newblock {\em Probab. Theory Related Fields}, 167(1-2):45--105, 2017.
\newblock \arxiv{1303.4725}. \MR{3602842}

\bibitem[MWW16]{extreme_nesting}
J.~Miller, S.~S. Watson, and D.~B. Wilson.
\newblock Extreme nesting in the conformal loop ensemble.
\newblock {\em Ann. Probab.}, 44(2):1013--1052, 2016.
\newblock \arxiv{1401.0217}. \MR{3474466}

\bibitem[RS05]{RS05}
S.~Rohde and O.~Schramm.
\newblock Basic properties of {SLE}.
\newblock {\em Ann. of Math. (2)}, 161(2):883--924, 2005.
\newblock \arxiv{math/0106036}. \MR{2153402}

\bibitem[RY99]{RY04}
D.~Revuz and M.~Yor.
\newblock {\em Continuous martingales and {B}rownian motion}, volume 293 of
  {\em Grundlehren der Mathematischen Wissenschaften [Fundamental Principles of
  Mathematical Sciences]}.
\newblock Springer-Verlag, Berlin, third edition, 1999.

\bibitem[Sch00]{S0}
O.~Schramm.
\newblock Scaling limits of loop-erased random walks and uniform spanning
  trees.
\newblock {\em Israel J. Math.}, 118:221--288, 2000.
\newblock \arxiv{math/9904022}. \MR{1776084}

\bibitem[She07]{SHE06}
S.~Sheffield.
\newblock Gaussian free fields for mathematicians.
\newblock {\em Probab. Theory Related Fields}, 139(3-4):521--541, 2007.
\newblock \arxiv{math/0312099}. \MR{2322706}

\bibitem[She09]{SHE_CLE}
S.~Sheffield.
\newblock Exploration trees and conformal loop ensembles.
\newblock {\em Duke Math. J.}, 147(1):79--129, 2009.
\newblock \arxiv{math/0609167}. \MR{2494457}

\bibitem[She16]{SHE_WELD}
S.~Sheffield.
\newblock Conformal weldings of random surfaces: {SLE} and the quantum gravity
  zipper.
\newblock {\em Ann. Probab.}, 44(5):3474--3545, 2016.
\newblock \arxiv{1012.4797}. \MR{3551203}

\bibitem[SW05]{SCHRAMM_WILSON}
O.~Schramm and D.~B. Wilson.
\newblock S{LE} coordinate changes.
\newblock {\em New York J. Math.}, 11:659--669 (electronic), 2005.
\newblock \arxiv{math/0505368}. \MR{2188260 (2007e:82019)}

\bibitem[Wer04]{W03}
W.~Werner.
\newblock Random planar curves and {S}chramm-{L}oewner evolutions.
\newblock In {\em Lectures on probability theory and statistics}, volume 1840
  of {\em Lecture Notes in Math.}, pages 107--195. Springer, Berlin, 2004.
  \MR{2079672 (2005m:60020)}

\end{thebibliography}

\bigskip

\filbreak
\begingroup
\small
\parindent=0pt

\bigskip
\vtop{
\hsize=5.3in
Statistical Laboratory, DPMMS\\
University of Cambridge\\
Cambridge, UK}

 \endgroup \filbreak

\end{document}